\documentclass{amsart}
\usepackage[utf8]{inputenc}

\usepackage{subcaption}

\usepackage{graphicx}
\usepackage{amssymb}
\usepackage{standalone}
\usepackage{tikz}
\usetikzlibrary{fit}

\usepackage{verbatim}

\usepackage{todonotes}

\usepackage{leftidx}

\usepackage{enumitem}
\usepackage{hhline}

\usepackage{amsmath,amssymb,amsthm,mathtools}
\usepackage{thmtools, thm-restate}

\newtheorem{theorem}{Theorem}[section]
\newtheorem{proposition}[theorem]{Proposition}
\newtheorem{lemma}[theorem]{Lemma}

\newtheorem{definition}[theorem]{Definition}
\newtheorem{remark}[theorem]{Remark}
\newtheorem{corollary}[theorem]{Corollary}
\newtheorem{question}[theorem]{Question}

\makeatletter
\DeclareFontFamily{OMX}{MnSymbolE}{}
\DeclareSymbolFont{MnLargeSymbols}{OMX}{MnSymbolE}{m}{n}
\SetSymbolFont{MnLargeSymbols}{bold}{OMX}{MnSymbolE}{b}{n}
\DeclareFontShape{OMX}{MnSymbolE}{m}{n}{
    <-6>  MnSymbolE5
   <6-7>  MnSymbolE6
   <7-8>  MnSymbolE7
   <8-9>  MnSymbolE8
   <9-10> MnSymbolE9
  <10-12> MnSymbolE10
  <12->   MnSymbolE12
}{}
\DeclareFontShape{OMX}{MnSymbolE}{b}{n}{
    <-6>  MnSymbolE-Bold5
   <6-7>  MnSymbolE-Bold6
   <7-8>  MnSymbolE-Bold7
   <8-9>  MnSymbolE-Bold8
   <9-10> MnSymbolE-Bold9
  <10-12> MnSymbolE-Bold10
  <12->   MnSymbolE-Bold12
}{}

\let\llangle\@undefined
\let\rrangle\@undefined
\DeclareMathDelimiter{\llangle}{\mathopen}%
                     {MnLargeSymbols}{'164}{MnLargeSymbols}{'164}
\DeclareMathDelimiter{\rrangle}{\mathclose}%
                     {MnLargeSymbols}{'171}{MnLargeSymbols}{'171}
\makeatother

\DeclareMathOperator{\C}{\mathcal{C}}
\DeclareMathOperator{\N}{\mathcal{N}}

\DeclareMathOperator{\U}{\mathcal{U}}

\DeclareMathOperator{\NC}{\mathcal{N_C}}

\DeclareMathOperator{\PT}{\mathcal{PT}}
\DeclareMathOperator{\proC}{pro-\mathcal{C}}

\let\Im\relax
\DeclareMathOperator{\Im}{Im}
\DeclareMathOperator{\supp}{supp}
\DeclareMathOperator{\stab}{stab}

\DeclareMathOperator{\Aut}{Aut}

\DeclareMathOperator{\Inn}{Inn}

\DeclareMathOperator{\Sym}{Sym}

\DeclareMathOperator{\id}{id}

\DeclareMathOperator{\Sub}{Sub}

\DeclareMathOperator{\ord}{ord}

\DeclareMathOperator{\Lift}{\mathrm{Lift}}
\newcommand\Tree{\ensuremath{T}}

\DeclareMathOperator{\fileq}{\leq_{\text{f.i.}}}
\DeclareMathOperator{\normfileq}{\unlhd_{\text{f.i.}}}
\DeclareMathOperator{\normleq}{\unlhd}

\title{$\mathcal{C}$-Hereditarily conjugacy separable groups and wreath products}

\author{Alex Bishop}
\address[Alex Bishop]{School of Computer and Information Sciences, University of Newcastle, Callaghan, NSW 2308, Australia}
\email[Alex Bishop]{alexbishop1234@gmail.com}

\author{Michal Ferov}
\address[Michal Ferov]{Department of Computer Science, Faculty of Science, University of South Bohemia in \v{C}esk\'{e} Bud\v{e}jovice, Brani\v{s}ovsk\'{a} 1645/31a, 370 05 \v{C}esk\'{e} Bud\v{e}jovice, Czech Republic}
\email[Michal Ferov]{mferov@prf.jcu.cz}

\author{Mark Pengitore}
\address[Mark Pengitore]{IMPAN, ul. \'{S}niadeckich 8, 00-656 Warszawa, Poland}
\email[Mark Pengitore]{mpengito@gmail.com}

\usepackage{color}   
\usepackage{hyperref}
\hypersetup{
    colorlinks=true, 
    linktoc=all,     
    linkcolor=blue,  
    linktocpage,
    anchorcolor = blue,
    citecolor = blue,
    filecolor = blue,
    urlcolor = blue,
}

\renewcommand\leq\leqslant
\renewcommand\geq\geqslant

\begin{document}

\begin{abstract}
    We provide a necessary and sufficient condition for the restricted wreath product $A\wr B$ to be $\C$-hereditarily conjugacy separable where $\C$ is an extension-closed pseudovariety of finite groups. Moreover, we prove that the Grigorchuk group is 2-hereditarily conjugacy separable. As an application, we demonstrate that the lamplighter groups and $\mathbb{Z} \wr \mathbb{Z}$ are hereditarily conjugacy separable (but not $p$-conjugacy separable for any prime $p$). This provides infinitely many new examples of solvable, non-polycyclic hereditarily conjugacy separable groups. Furthermore, we study wreath products of cyclic subgroup separable groups and the derived length of iterated wreath products of solvable groups with an abelian base group and, as an application, we give an explicit construction of non-polycyclic hereditarily conjugacy separable groups of arbitrary derived length as iterated wreath products of abelian groups.
\end{abstract}

\maketitle

\tableofcontents

\section{Introduction}
A standard way of studying infinite discrete groups is to consider their finite quotients.
Thus, a naturally arising question is how much information about the group can be recovered from its finite quotients. A group $G$ is called \emph{residually finite}~(\emph{RF}) if for any non-trivial element $g \in G$, there exists a surjective homomorphism $\pi \colon G \to F$ onto a finite group $F$ such that $\pi(g)$ is non-trivial in $F$. Informally speaking, a group is residually finite if we can distinguish its individual elements from the identity in its finite quotients. A group $G$ is called \emph{conjugacy separable}~(\emph{CS}) if for any pair of non-conjugate elements $f,g \in G$, there is a surjective homomorphism $\pi \colon G \to F$ onto a finite group $F$ such that the elements $\pi(f)$ and $\pi(g)$ are not conjugate in $F$. Again, informally speaking, a group is conjugacy separable if we can distinguish conjugacy classes of its elements in its finite quotients.

Group properties of this type are called \emph{separability properties}. In this paper, we study $\C$-conjugacy separability (which is a more general version of conjugacy separability), $\C$-hereditary conjugacy separability (which is a strengthening of $\C$-conjugacy separability), and the closure properties of the class of $\C$-hereditarily conjugacy separable groups with respect to the construction of wreath products.

\subsection{Motivation}
Separability properties provide an algebraic analogue of the solvability of decision problems for finitely presented groups. Mal'tsev~\cite{malcev} proved that finitely presented residually finite groups have solvable word problem. Similarly, Mostowski~\cite{mostowski} showed that finitely presented conjugacy separable groups have solvable conjugacy problem.

Even though the definition of conjugacy separability is similar to that of residual finiteness, it is a much stronger condition. Indeed, it can be easily seen that every CS group is RF, but the implication in the opposite direction does not hold. A simple example of a RF group which is not CS was given independently by Stebe~\cite{stebe_sl3z} and Remeslennikov~\cite{remeslennikov} when they proved that $\text{SL}_3(\mathbb{Z})$ is not CS.

The following classes of groups are known to be conjugacy separable: 
\begin{itemize}
    \item virtually free groups (Dyer~\cite{dyer}),
    \item virtually polycyclic groups (Formanek~\cite{polycyclic_formanek}, Remeslennikov~\cite{polycyclic_remeslennikov}),
    \item virtually surface groups (Martino~\cite{armando}),
    \item limit groups (Chagas and Zalesskii~\cite{limit}),
    \item finitely generated right-angled Artin groups (Minasyan~\cite{ashot_raags}),
    \item even Coxeter groups whose diagram does not contain $(4,4,2)$-triangles (Caprace and Minasyan~\cite{racgs}),
    \item finitely presented residually free groups (Chagas and Zalesskii~\cite{chagas}),
    \item one-relator groups with torsion (Minasyan and Zalesskii~\cite{1-rel}), and
    \item fundamental groups of compact orientable 3-manifolds (Hamilton, Wilton, and Zalesskii~\cite{compact}).
\end{itemize} 
It is easy to see that being residually finite is a hereditary property, that is, if a group $G$ is residually finite, then every subgroup $H \leq G$ is residually finite as well. However, conjugacy separability does not behave well with respect to subgroup inclusion, not even of finite index. Martino and Minasyan~\cite{armando_ashot} showed that for every integer $m \geq 2$, there exists a finitely presented conjugacy separable group $T$ that contains a subgroup $S \leq T$ such that $|T:S| = m$ and $S$ is not conjugacy separable. We say that a group $G$ is \emph{hereditarily conjugacy separable} (\emph{HCS}) if $G$ and each of its finite-index subgroups, $H\leq G$, are conjugacy separable.

It is natural to study the closure properties of the class of HCS groups with respect to group-theoretic constructions. For example, it is clear from the definition that the class of conjugacy separable groups is closed under forming finite direct products. It was proven by Stebe~\cite{stebe} and independently by Remeslennikov~\cite{remeslennikov} that the class of CS groups is closed under taking free products, and thus, it follows that a free product of HCS groups is again a HCS group. The second named author~\cite{mf} showed that graph products of HCS groups, a natural generalisation of direct and free products of groups, are again HCS. Remeslennikov~\cite{remeslennikov} showed that if $A$ and $B$ are both CS groups, then their restricted wreath product $A \wr B$ is a conjugacy separable group if and only if either $B$ is finite, or $B$ is infinite with the additional properties that every cyclic subgroup of $B$ is separable and $A$ is abelian.

A natural way of generalising separability properties is to only consider finite quotients of a specific type, say finite nilpotent groups or finite $p$-groups for some prime $p$. Let $\mathcal{C}$ be a class of finite groups (we will always assume that classes of finite groups are closed under isomorphisms), and let $G$ be a group.	We say that $G$ is \emph{residually}-$\C$ if for every non-trivial element $g \in G$, there is a group $F \in \C$ and a surjective homomorphism $\phi \colon G \twoheadrightarrow F$ such that $\phi(g)$ is non-trivial in $F$. Similarly, we say that $G$ is \emph{$\C$-conjugacy separable} (\emph{$\C$-CS}) if for pairs of elements $f,g \in G$ such that $f$ is not conjugate to $g$ in $G$, there is a group $F \in \C$ and a surjective homomorphism $\phi \colon G \twoheadrightarrow F$ such that $\phi(f)$ is not conjugate to $\phi(g)$ in $F$.
In the case when $\C$ is an extension-closed pseudovariety of finite groups (see Subsection~\ref{subsection:proc topologies} for more details), we say that $G$ is \emph{$\C$-hereditarily conjugacy separable} if it is $\C$-conjugacy separable and every subgroup $H \leq G$, open in the pro-$\C$ topology, is $\C$-conjugacy separable as well ($H$ is open in the pro-$\C$ topology if and only if there is $K \normleq G$ such that $K \leq H$ and $G/K \in \C$; see Subsection~\ref{subsection:proc topologies}).

In this paper, we are interested in classes of finite groups $\C$ which are \emph{extension-closed pseudovarieties of finite groups}, that is, classes of finite groups which are closed under taking subgroups, direct products, quotients, and extensions. Three common examples of extension-closed pseudovarieties of finite groups are the class of all finite $p$-groups where $p$ is a prime number, the class of all finite solvable groups, and the class of all finite groups.  

We have the following examples of $\C$-conjugacy separable groups:
\begin{itemize}
    \item right-angled Artin groups and free-by-$\C$ groups are $\C$-conjugacy separable when $\C$ is the class of all finite $p$-groups for some prime $p$ (Toinet~\cite{toinet2013conjugacy}),
    \item finitely generated free-by-$\C$ groups are $\C$-conjugacy separable when $\C$ is an extension-closed pseudovariety of finite groups (Ribes and Zalesskii~\cite{zalesskii}) and
    \item general free-by-$\C$ groups are $\C$-conjugacy separable when $\C$ is an extension-closed pseudovariety of finite groups (Ribes~\cite[Theorem 14.1.4]{ribes2017profinite}).
\end{itemize}

In this paper, we study the behaviour of $\C$-hereditary conjugacy separability, where the class $\C$ is an extension-closed pseudovariety of finite groups, under the construction of wreath products. It is easy to see that a direct product of $\C$-CS groups is again a $\C$-CS group, and similarly for $\C$-hereditarily conjugacy separable groups. Using the result of Ribes and Zalesskii~\cite{zalesskii} that free-by-$\C$ groups are $\C$-CS, one could easily generalise the result of Stebe~\cite{stebe} and Remeslennikov~\cite{remeslennikov} to show that $\C$-CS and $\C$-HCS groups are closed under taking free products. Furthermore, the second named author~\cite{mf} showed that the classes of $\C$-CS and $\C$-HCS groups are closed under forming graph products. 

\subsection{Statement of results}
The following theorem was proven by Remeslennikov \cite[Theorem~1]{remeslennikov} in the case when $\C$ is the class of all finite groups and later generalised by the authors Ferov and Pengitore~\cite[Theorem~A]{quantifying} to the case when $\C$ is any extension-closed pseudovariety of finite groups, as follows.
The term \emph{$\C$-cyclic subgroup separable} used in the following theorem is defined in Definition~\ref{def:C-props}.

\begin{theorem}[Theorem~A in~\cite{quantifying}]\label{theorem:CCS}
    Let $\C$ be an extension-closed pseudovariety of finite groups, and suppose that $A$ and $B$ are non-trivial $\C$-CS groups. Then, the restricted wreath product $A \wr B$ is a $\C$-CS group if and only if either one of the following is true:
    \begin{enumerate}[label=(\roman*)]
        \item $B \in \C$ and $A$ is $\C$-CS, or
        \item $A$ is abelian and residually-$\C$, $B$ is $\C$-CS and every cyclic subgroup of $B$ is $\C$-separable in $B$.
    \end{enumerate}
\end{theorem}

The main result of this paper, which is proven at the end of Section~\ref{sec:wreath-prop-B-infinite}, is the following strengthening of the above theorem to the setting of $\C$-hereditarily conjugacy separable groups. Note that the term \emph{$\C$-separable} is defined in Definition~\ref{def:C-open-and-closed}.

\begin{restatable*}{theorem}{MainTheoremHCS}\label{theorem:HCS}
    Let $\C$ be an extension-closed pseudovariety of finite groups, and suppose that $A$ and $B$ are non-trivial $\C$-HCS groups. Then, the restricted wreath product $A \wr B$ is a $\C$-HCS group if and only if either one of the following is true:
    \begin{enumerate}[label=(\roman*)]
        \item $B\in \C$ and $A$ is $\C$-HCS, or
        \item\label{theorem:HCS/2} $A$ is abelian and residually-$\C$, $B$ is $\C$-HCS, and every cyclic subgroup of $B$ is $\C$-separable in $B$.
    \end{enumerate}
\end{restatable*}
Throughout the paper, unless stated otherwise, we will always assume that $\C$ is an extension-closed pseudovariety of finite groups.

Let $p$ be a prime. Then, to simplify the presentation of this paper, we say that a group is \emph{residually-$p$}, \emph{$p$-conjugacy separable}, or \emph{$p$-hereditarily conjugacy separable} if it is residually-$\C$, $\C$-conjugacy separable, or $\C$-hereditarily conjugacy separable, respectively, where $\C$ is the class of finite $p$-groups.

It is an easy observation (see Section~\ref{section:wreath products of CSS groups} for more detail) that, if a group $G$ contains an element of infinite order, then, for every prime number $p$, it contains a cyclic subgroup that is not $p$-separable.
This means that any infinite group $B$ satisfying property \ref{theorem:HCS/2} 
of Theorem~\ref{theorem:HCS}, 
i.e.\@ an infinite $p$-conjugacy separable group such that all cyclic subgroups are $p$-separable, must be a torsion group. There are examples of infinite torsion groups that are known to satisfy various separability properties, for example:
\begin{itemize}
    \item the first Grigorchuk group is conjugacy separable \cite{leonov1998conjugacy},
    \item Gupta-Sidki groups are conjugacy separable \cite{wilson1997conjugacy},
    \item the Gupta-Sidki group $\mathrm{GS}(3)$ is subgroup separable \cite{garrido_GS}, and
    \item Grigorchuk-Gupta-Sidki groups are subgroup separable \cite{francoeur2020subgroup}.
\end{itemize}
To the best of the authors' knowledge, none of the above examples in the literature are known to be hereditarily conjugacy separable.
Motivated by this gap in the literature, we prove the following strengthening of the result of Wilson and Zalesskii \cite{wilson1997conjugacy} in Section~\ref{section:grigorchuk}.

\begin{restatable*}{theorem}{GrigTheorem}\label{thm:hered-conj-sep}
    The Grigorchuk group is hereditarily conjugacy separable. In particular, the Grigorchuk group is $2$-hereditarily conjugacy separable.
\end{restatable*}
In the appendix, we give a characterisation of $\C$-HCS groups in terms of centralisers in the pro-$\C$ completion; see Corollary \ref{corollary:centralisers in the completion}. Combining Theorem \ref{thm:hered-conj-sep} with Corollary \ref{corollary:centralisers in the completion}, we immediately get the following statement, providing a complement to the characterisation of centralisers of elements in the Grigorchuk group by Rozhkov \cite{rozhkov_centralizer}.
\begin{restatable*}{corollary}{GrigCor}
    Let $G$ be the Grigorchuk group and let $\widehat{G}$ be its profinite completion. Then, for every $g \in G$, the centraliser $C_G(g)$ is dense in the centraliser $C_{\widehat{G}}(g)$.
\end{restatable*}

In Section 6, we study wreath products of cyclic subgroup separable groups. In particular, we give the following characterisation of when a wreath product of groups is a cyclic subgroup separable (CSS) group; see Definition \ref{def:C-props} for the formal definition of $\C$-CSS groups.
\begin{restatable*}{proposition}{PropositionCSS}\label{proposition:wreath products of cyclic subgroup separable}
            Let $\C$ be an extension-closed pseudovariety of finite groups and suppose that $A$ and $B$ are non-trivial $\C$-CSS groups. Then, $A \wr B$ is $\C$-CSS if and only if at least one of the following is true:
    \begin{enumerate}[label=(\roman*)]
        \item $B \in \C$, or
        \item $A$ is abelian.
    \end{enumerate}
\end{restatable*}
Using Magnus' embedding, we then immediately obtain the following corollary.
\begin{restatable*}{corollary}{freesolvableCSS}
    \label{corollary:free solvable groups are CSS}
    Let $\C$ be an extension-closed pseudovariety of finite groups. Then the free solvable group $S_{r,k}$ is $\C$-CSS if and only if $\C$ contains all finite cyclic groups. In particular, $ S_{r,k}$ is cyclic subgroup separable.
\end{restatable*}

\subsection{Applications}
\label{subsection:applications}
Using Theorem~\ref{theorem:HCS}, we construct many new examples of hereditarily conjugacy separable groups and $\C$-hereditarily conjugacy separable groups as follows.

It is easy to see that if $G$ is a $\C$-HCS group and $V \leq G$ is a $\C$-virtual retract of $G$ (see Subsection~\ref{subsection:proc topologies} for details and formal definition of $\C$-virtual retracts), then $V$ is a $\C$-HCS group as well. Minasyan \cite[Theorem 9.7]{virtual_retractions} showed that if $A$ is a semisimple finite abelian group, then every finitely generated subgroup of $G = A \wr \mathbb{Z}$ is a virtual retract of $G$. A finite abelian group $A$ is semisimple if and only if it is a direct sum of cyclic groups of prime order, i.e.\@ if and only if 
\begin{displaymath}
    A = C_{p_1} \oplus C_{p_2} \oplus \dots \oplus C_{p_n},
\end{displaymath}
where $p_1, p_2, \dots, p_n \in \mathbb{P}$ are prime numbers. In particular, all cyclic groups of prime order are semisimple. Combining \cite[Theorem 9.7]{virtual_retractions} with Theorem \ref{theorem:HCS}, we immediately get the following corollary.
\begin{corollary}
    \label{corollary:finitely generates subgroups of lamplighters}
    Let $A$ be a semisimple finite abelian group. Then every finitely generated subgroup of $A \wr \mathbb{Z}$ is a HCS group. In particular, all finitely generated subgroups of the lamplighter group $\mathbb{F}_p \wr \mathbb{Z}$, where $p \in \mathbb{P}$ is a prime, are hereditarily conjugacy separable.
\end{corollary}

In~\cite[Corollary~A.1]{quantifying}, the second and third named authors characterised when the wreath product of $p$-conjugacy separable groups is itself $p$-conjugacy separable.
Using this result and Theorem~\ref{theorem:HCS}, we characterise when the wreath product of $p$-hereditarily conjugacy separable groups is $p$-hereditarily conjugacy separable as in the following corollary.
\begin{corollary}\label{cor_1}
Suppose that $A$ and $B$ are $p$-hereditarily conjugacy separable groups for some prime $p$.
Then $A \wr B$ is $p$-hereditarily conjugacy separable if and only if at least one of the following is true:
\begin{enumerate}
    \item $B$ is a finite $p$-group, or
    \item $A$ is abelian, and $B$ is a $p$-group.
\end{enumerate}
\end{corollary}

We have the following immediate application of the above corollary.
\begin{corollary}\label{cor_2}
    Suppose that $G$ is the Grigorchuk group. Then the group $\mathbb{Z}^m \wr G$ is $2$-hereditarily conjugacy separable for any $m \in \mathbb{N}$.
\end{corollary}
\begin{proof}
    This follows from Corollary~\ref{cor_1} and Theorem~\ref{thm:hered-conj-sep}.
\end{proof}

It is well known that all finitely generated nilpotent groups, and, more generally, polycyclic groups, are hereditarily conjugacy separable and have separable cyclic subgroups. As a consequence of Corollary~\ref{cor_1}, in the following corollary, we produce infinitely many new examples of torsion-free hereditarily conjugacy separable solvable groups that are not polycyclic, and are not $p$-conjugacy separable for any prime $p$.

\begin{corollary}\label{cor_3}
    Suppose that $G$ is an infinite torsion-free finitely generated nilpotent group or an infinite torsion-free polycyclic group, and let $A$ be an abelian group. Then $A \wr G$ is hereditarily conjugacy separable but is not $p$-conjugacy separable for any prime $p$.
\end{corollary}

We have the following immediate corollary to Corollary~\ref{cor_3}.
\begin{corollary}\label{cor_4}
     If $C$ is an arbitrary cyclic group, finite or infinite, then the wreath product $C \wr \mathbb{Z}$ is hereditarily conjugacy separable, but not $p$-hereditarily conjugacy separable for any prime $p \in \mathbb{P}$. 
\end{corollary}
Recall that the groups in the statements of Corollary \ref{corollary:finitely generates subgroups of lamplighters} and Corollary~\ref{cor_4} are all solvable of derived length 2. To the best of the authors' knowledge, these groups are the first known examples of non-polycyclic solvable groups that are hereditarily conjugacy separable. Motivated by this observation, in Section~\ref{section: Wreath prouducts of solvable groups}, we give an explicit construction of torsion-free hereditarily conjugacy separable non-polycyclic solvable groups of arbitrary derived length as iterated wreath products of infinite cyclic groups.

\begin{restatable*}{corollary}{IteratedWreath}\label{corollary:iterated wreath}
     Let $G_1 = \mathbb{Z}$ and set $G_{k+1} = \mathbb{Z} \wr G_k$ for $k > 1$. Then $G_k$ is a torsion-free, hereditarily conjugacy separable, and cyclic subgroup separable $k$-step solvable group. Furthermore, $G_k$ is not polycyclic for $k>1$.
\end{restatable*}

\subsection*{Acknowledgements}

Michal Ferov would like to thank Ashot Minasyan for suggesting the study of wreath products of hereditarily conjugacy separable groups and for useful email conversations.
The majority of the project was done when Michal Ferov was based at the University of Newcastle, Australia, where he was partially supported by the Australian Research Council Laureate Fellowship FL170100032 of Professor George Willis.
Alex Bishop acknowledges support from Swiss NSF grant 200020-200400, and from an honorary fellowship with the University of Technology Sydney.
Mark Pengitore is supported by the National Science Center Grant Maestro-13 UMO-2021/42/A/ST1/00306.
We would like to thank the anonymous referees for their comments and useful suggestions.

\subsection*{AI tool and computational resource disclosure}
All mathematical content presented in the manuscript is the original work of the authors and did not involve the use of generative AI tools. However, large language models such as \texttt{ChatGPT} and \texttt{Claude} were used for checking grammar, spelling, punctuation, and consistent use of language throughout the text; all editing was done manually by the authors.

\section{Preliminaries}
Given a group $G$ and elements $f,g \in G$, we will write $f \sim_G g$ to denote that $f$ and $g$ are conjugate in $G$, i.e.~that there is an element $c \in G$ such that $f = cgc^{-1}$. Our commutator convention will be given by $[x,y] = xyx^{-1}y^{-1}$. For $x_1, \ldots, x_s$ where $s > 2$, we write $[x_1, \ldots, x_s] = [[x_1, \ldots, x_{s-1}], x_s]$ where $[x_1, \ldots, x_{s-1}]$ is defined inductively. For subgroups $H_1, \ldots, H_s$ of $G$, we define $[H_1, \ldots, H_s]$ to be the subgroup generated by elements of the form $[h_1, \ldots, h_s]$ where $h_i \in H_i$. Given a subgroup $H \leq G$, we will use $C_H(g)$ to denote the $H$-centraliser of $g$, i.e.\@ $C_H(g) = \{h \in H \mid hg = gh\}$.

Given a group $G$, we say that a subgroup $R \leq G$ is a \emph{retract} in $G$ if there is a surjective homomorphism $\rho \colon G \to R$ such that $\rho|_R = \id_{R}$. Thus, if $R \leq G$ is a retract, then $G$ splits as a semidirect product $G = K \rtimes R$ where $K$ is the kernel of the associated retraction map $\rho \colon G \to R$.

Unless stated otherwise, all actions considered will be left actions.

\subsection{Profinite topologies on groups}
\label{subsection:proc topologies}

To ensure that this paper is self-contained, we present in this subsection a background on the basic facts about $\proC$ topologies on groups.
Experts in the field may feel free to skip this subsection. Proofs of all statements can be found in the classic book by Ribes and Zalesskii~\cite{rz} or in the second named author's doctoral thesis~\cite{mf_thesis}.

Let $\mathcal{C}$ be a class of groups, and let $G$ be a group. We say that a normal subgroup $N \unlhd G$ is a \textbf{co-$\mathcal{C}$ subgroup} of $G$ if $G/N \in \mathcal{C}$. We then write
$\NC(G)$ for the class of all co-$\mathcal{C}$ subgroups of $G$. Following the notation in~\cite{mf_thesis}, we define the following three closure properties that a class of groups $\C$ may satisfy: 
	\begin{itemize}
	    \item[(c0)] $\mathcal{C}$ is closed under taking finite subdirect
	    products,
	    \item[(c1)]	$\mathcal{C}$ is closed under taking subgroups,
	    \item[(c2)]	$\mathcal{C}$ is closed under taking finite direct products.
	\end{itemize}
Recall that a subgroup $S \leq A \times B$ is a subdirect product of $A$ and $B$ if $\pi_A(S) = A$ and $\pi_B(S) = B$ where  $\pi_A \colon A \times B \to A$ and $\pi_B \colon A \times B \to B$ are the natural projections onto the direct factors. Notice that 
\[(c0) \Rightarrow  (c2)\qquad\text{and}\qquad (c1)+ (c2) \Rightarrow (c0).\]
We also note that in \cite[p.~19]{rz}, the properties (c0), (c1), and (c2) are instead labelled as ($\mathcal{C}$4), ($\mathcal{C}$1), and ($\mathcal{C}$2), respectively.

\begin{remark}\label{remark:intersections}
If the class of groups $\mathcal{C}$ satisfies (c0), then for every group $G$, the set $\NC(G)$ is closed under finite intersections. That is, if $N_1, N_2 \in \NC(G)$, then $N_1\cap N_2\in\NC(G)$.
\end{remark}

Following Remark~\ref{remark:intersections}, we see that whenever the class $\C$ satisfies property (c0), the set $\NC(G)$ is a base at $1$ for a topology on $G$. Hence, the group $G$ can be equipped with a group topology given by a base of open sets
\[\{gN \mid g\in G,\  N \in \NC(G)\}.\]
This topology is known as the \textbf{pro-$\C$ topology} on $G$ and is denoted pro-$\C(G)$. In the case when $\C$ is the class of all finite groups, the pro-$\C$ topology on $G$, usually denoted as $\PT(G)$, is called the \emph{profinite} topology on $G$.

If the class $\mathcal{C}$ satisfies (c1) and (c2), or equivalently, (c0) and (c1), then  equipping a group $G$ with its pro-$\mathcal{C}$ topology is a faithful functor from the category of groups to the category of topological groups, as witnessed by the following lemma.
\begin{lemma}[Lemma~2.2 in~\cite{mf_thesis}]
	\label{lemma:continuous}
	Let $\C$ be a class of groups satisfying (c1) and (c2). Given groups $G$ and $H$, every morphism $\varphi\colon G\to H$ is a continuous map with respect to the corresponding pro-$\mathcal{C}$ topologies of $G$ and $H$. Furthermore, if $\varphi$ is a group isomorphism, then it is a homeomorphism of the topologies.
\end{lemma}

We have the following additional terminology for subsets of pro-$\C(G)$.

\begin{definition}\label{def:C-open-and-closed}
A subset $X \subseteq G$ is \textbf{$\mathcal{C}$-closed} in $G$ if $X$ is closed in pro-$\mathcal{C}(G)$. We then say that a subset $X \subseteq G$ is \textbf{$\C$-separable} if it is $\C$-closed. Accordingly, a subset is \textbf{$\mathcal{C}$-open} in $G$ if it is open in pro-$\mathcal{C}(G)$.
\end{definition}

We then have the following alternative characterisation of a $\C$-closed subset.
\begin{lemma}[Lemma~2.1 in~\cite{mf_thesis}]
Suppose that $G$ is a finitely generated group equipped with the pro-$\C$ topology and that $X \subseteq G$ is a non-empty subset. Then $X$ is $\mathcal{C}$-closed if and only if for every element $g \notin X$, there exists a subgroup $N \in \NC(G)$ such that $\pi_N(g) \notin \pi_N(X)$ in $G/N$ where $\pi_N \colon G \to G/N$ is the natural projection.
\end{lemma}

The following easy lemma will be useful for shortening proofs.
\begin{lemma}
    \label{lemma:lazy lemma separability}
    Let $G$ be a group, and let $X \subseteq G$ be a subset. Then $X$ is $\C$-separable in $G$ if and only if for every $g \notin X$, there is a homomorphism $\pi \colon G \to Q$ such that $\pi(g) \notin \pi(X)$ and $\pi(X)$ is $\C$-separable in $Q$.
\end{lemma}
\begin{proof}
    Let $X \subseteq G$ be as above. By assumption, for every $g \in G \setminus X$, there is a homomorphism $\pi_g \colon G \to Q_g$ such that $\pi_g(x) \notin \pi_g(X)$ and $\pi_g(X)$ is $\C$-closed in $Q_g$. This means that $\pi_g^{-1}(\pi_g(X))$ is $\C$-closed in $G$ as it is a preimage of a closed set. Furthermore, we clearly see that $g\notin \pi_g^{-1}(\pi_g(X))$. It immediately follows that 
    \begin{displaymath}
        X = \bigcap_{g \in G\setminus X} \pi_g^{-1}(\pi_g(X))
    \end{displaymath}
    is closed in $\proC(G)$ as it is an intersection of $\C$-closed subsets.
\end{proof}

We can now describe various separability properties in topological terms by specifying what types of subsets are required to be closed in the pro-$\C$ topology.
\begin{definition}\label{def:C-props}
We say that a group $G$ is
\begin{itemize}
    \item \textbf{residually-$\C$} if $\{1\}$ is a $\C$-closed subset of $G$;
    \item \textbf{$\C$-conjugacy separable} ($\C$-CS) if every conjugacy class is $\C$-closed;
    \item \textbf{$\C$-hereditarily conjugacy separable} if ($\C$-HCS) $G$ is $\C$-CS and every $\C$-open subgroup of $G$ is $\C$-CS as well; and
    \item \textbf{$\C$-cyclic subgroup separable} ($\C$-CSS) if every cyclic subgroup is $\C$-closed.
\end{itemize}
\end{definition}

In this paper, we consider classes of finite groups such as the class of all finite groups or the class of all finite $p$-groups where $p$ is some prime. These two classes of finite groups are examples of extension-closed pseudovarieties of finite groups as seen in the following definition.

\begin{definition}
A class of finite groups that is closed under subgroups, finite direct products, quotients, and extensions is called an \textbf{extension-closed pseudovariety of finite groups}.
\end{definition}
From this point onward, we assume that the class $\C$ is an extension-closed pseudovariety of finite groups. 
\begin{lemma}
    \label{lemma:finite subgroups}
    Let $\C$ be an extension-closed pseudovariety of finite groups, and let $G$ be a residually-$\C$ group. If $F \leq G$ is finite, then $F \in \C$.
\end{lemma}
\begin{proof}
    Let $F \leq G$ be finite. Since $G$ is residually-$\C$, there is $K \in \NC(G)$ such that the natural projection $\pi \colon G \to G/K$ is injective on $F$, i.e.\ $F \simeq \pi(F) \leq G/K$. Since $G/K \in \C$, we see that $\pi(F) \in \C$ as the class $\C$ is closed under taking subgroups.
\end{proof}

In the profinite case, i.e.~when $\C$ is the class of all finite groups, one can easily demonstrate that a subgroup $H \leq G$ is open in the profinite topology on $G$ if and only if it has finite index in $G$. Furthermore, one can show that $H$ is closed in $G$ if and only if it is a (possibly infinite) intersection of subgroups of finite index. An analogous statement can be made in the general case as seen in the following lemma which combines \cite[Theorem~3.1]{hall} and \cite[Theorem~3.3]{hall}.
\begin{lemma}
	\label{lemma:closed/open subgroups}
	Let $G$ be a group and let $H \leq G$. Then $H$ is $\mathcal{C}$-open in $G$ if and only if there is $N \in \NC(G)$ such that $N \leq H$. In particular, every $\mathcal{C}$-open subgroup is $\mathcal{C}$-closed in $G$ and is of finite index in $G$. Furthermore, $H$ is $\mathcal{C}$-closed in $G$ if and only if it is an intersection of $\mathcal{C}$-open subgroups of $G$.
 \end{lemma}

Recall that a subgroup $R \leq G$ is a retract in $G$ if there is a surjective homomorphism $\rho \colon G \to R$ such that $\rho|_R = \id_R$. Indeed, given two elements $x,y \in R$, we see that $x \sim_R y$ if and only if $x \sim_G y$. In particular, if $y = cxc^{-1}$ for some $c \in G$, then $y = \rho(y) = \rho(cxc^{-1}) = \rho(c) x \rho(c)^{-1}$. It then follows that if $G$ is a $\C$-CS group and $R \leq G$ is a retract, then $R$ is $\C$-CS as well. Given $x,y \in R$ such that $x \not\sim_R y$, we have $x \not\sim_G y$, and thus we take the appropriate quotient of $G$ and restrict the projection map to $R$. 

It follows from the definition that if $G$ is a $\C$-HCS group, then every $\C$-open subgroup of $G$ is in fact $\C$-HCS as well. We say that $R \leq G$ is a $\C$-virtual retract of $G$ if there is a $\C$-open subgroup $H \leq G$ such that $R$ is a retract in $H$. Following the previous paragraph, it can be easily seen that the property of being $\C$-HCS is passed down to $\C$-virtual retracts as well.
\begin{remark}
    \label{remark:retracts are C-HCS}
    Let $G$ be a $\C$-hereditarily conjugacy separable group, and suppose that $R \leq G$ is a $\C$-virtual retract of $G$. Then $R$ is a $\C$-HCS group as well.
\end{remark}

Recall that a group $G$ is $\C$-cyclic subgroup separable ($\C$-CSS for short) if every cyclic subgroup of $G$ is $\C$-closed in $G$. In other words, $G$ is $\C$-CSS if for every $g,h \in G$ such that $h \notin \langle g \rangle$, there exists a subgroup $K \in \NC(G)$ such that $h \notin \langle g \rangle K$ in $G$. In the case when $\C$ is the class of all finite $p$-groups, a residually-$\C$ group $G$ is $\C$-CSS if and only if $G$ is a torsion group, i.e.~$G$ cannot contain an element of infinite order. This is a consequence of the following observation. If $a \in G$ is of infinite order and $q$ is a prime number distinct from $p$, then the subgroup $\langle a^q \rangle$ is dense in $\langle a \rangle$, meaning that for any homomorphism $\pi \colon G \to Q$, where $Q$ is a finite $p$-group, we have $\pi(a) \in \langle \pi (a^q) \rangle$. Thus, the element $a$ cannot be separated from the subgroup $\langle a^q\rangle$ even though $a \notin \langle a^q \rangle$. 
\begin{lemma}
    \label{lemma:order is a multiple}
    Let $G$ be a group, and suppose that every cyclic subgroup of $G$ is $\C$-closed in $G$. Then for every $g \in G$ with $\ord(g) = \infty$ and every integer $N >1$, there exists a subgroup $K \in \NC(G)$ and $c \in \mathbb{N}$ such that $\langle g\rangle \cap K = \langle g^{cN}\rangle$, i.e.~the order of the image of $g$ under the natural projection onto $G/K$ is a multiple of $N$.
\end{lemma}
\begin{proof}
    Denote $D = \{d \in \mathbb{N} \mid d<N \mbox{ and } d|N\}$. Clearly, $g^d \notin \langle g^N \rangle$ for each $d \in D$. Thus, as every cyclic subgroup of $G$ is $\C$-separable in $G$, for every $d \in D$ there is a subgroup $K_d \in \NC(G)$ such that $g^d \notin \langle g^N\rangle K_d$. Furthermore, as $G$ is residually-$\C$, there is a subgroup $K_0 \in \NC(G)$ such that $g^N \notin K_0$. We then set
    \begin{displaymath}
        K = K_0 \cap \bigcap_{d \in D} K_d.
    \end{displaymath}
    Hence, $g^d \notin \langle g^N \rangle K$ for all $d\in D$ and $g^N \notin K$.

    As $K \in \NC(G)$, and thus is a finite-index subgroup of $G$, we can write $\langle g \rangle \cap K = \langle g^m\rangle$ for some $m \in \mathbb{N}$. Denote $d' = \gcd(m,N)$. By B\'ezout's identity, we have $\langle g^{d'} \rangle \subseteq \langle g^N \rangle K$, but by construction, we have $g^d \notin \langle g^N \rangle K$ for all proper divisors of $N$. Thus, we see that $\gcd(m,N) = N$ by necessity, because any proper divisor $d|N$ would contradict the construction of $K$. In particular, we see that $\langle g \rangle \cap K = \langle g^{cN} \rangle$ for some $c \in \mathbb{N}$. 
\end{proof}

Let $G$ be a group, and let $H \leq G$. We say that $\proC(H)$ is a \emph{restriction} of $\proC(G)$ if a subset $X \subseteq H$ is $\C$-closed in $H$ if and only if it is $\C$-closed in $G$. Note that if $\proC(H)$ is a restriction of $\proC(G)$, then $H$ is $\C$-closed in $G$ as $H$ is $\C$-closed in $H$ by definition. The following straightforward lemma was proven in \cite[Lemma~2.4]{separating_cyclic_subgroups_graph_products}.
\begin{lemma}
    \label{lemma:restriction}
    Let $G$ be a group, and let $H \leq G$ be a subgroup. Then $\proC(H)$ is a restriction of $\proC(G)$ if and only if every $N \in \NC(H)$ is $\C$-closed in $G$.
\end{lemma}
The following lemma can be shown using the proof of \cite[Lemma~3.1.5]{rz}.
\begin{lemma}
    \label{lemma:retract_restriction}
    Let $G$ be a residually-$\C$ group, and suppose that $R \leq G$ is a retract. Then $\proC(R)$ is a restriction of $\proC(G)$.
\end{lemma}
The following lemma was proven in \cite[2.6]{separating_cyclic_subgroups_graph_products}.
\begin{lemma}
    \label{lemma:restriction_product}
    Let $G_1, G_2$ be groups, and let $H_1 \leq G_1$, $H_2 \leq G_2$ be subgroups. Suppose that $\proC(H_1)$ is a restriction of $\proC(G_1)$ and that $\proC(H_2)$ is a restriction of $\proC(G_2)$. Then $\proC(H_1 \times H_2)$ is a restriction of $\proC(G_1 \times G_2)$.
\end{lemma}

\subsection{Centraliser conditions}
\label{subsection:centraliser condition}
Hereditary conjugacy separability is tightly connected to the ``growth'' of centralisers in finite quotients. Therefore, we have the following definition.
\begin{definition}
    We say that a group $G$ satisfies the \textbf{$\C$-centraliser condition} if for every $g \in G \setminus\{1\}$ and every subgroup $K \in \NC(G)$, there is a subgroup $L \in \NC(G)$ such that $L \leq K$ and 
\begin{displaymath}
    C_{G/L}(\pi_L(g)) \subseteq \pi_L(C_G(g)K) \text{ in } G/L,
\end{displaymath}
where $\pi_L \colon G \to G/L$ is the natural projection map.
\end{definition}

The following theorem was first proven by Minasyan~\cite{ashot_raags} in the case when $\C$ is the class of all finite groups, later by Toinet~\cite{toinet2013conjugacy} in the case when $\C$ is the class of all finite $p$-groups for some prime number $p$, and by the second named author~\cite{mf} in the general case when $\C$ is an extension-closed pseudovariety of finite groups. 
\begin{theorem}[Theorem~4.2 in~\cite{mf}]
    \label{theorem:CC_HCS}
    Let $G$ be a group. Then the following are equivalent:
    \begin{enumerate}[label=(\roman*)]
        \item $G$ is $\C$-hereditarily conjugacy separable,
        \item $G$ is $\C$-conjugacy separable and satisfies $\C$-centraliser condition.
    \end{enumerate}
\end{theorem}
Let us note that the centraliser condition can also be characterised in terms of the pro-$\C$ completion; see the appendix for more details.

Before we proceed, we introduce some more technical refinements of centraliser conditions as follows.
\begin{definition}
We say that an element $g \in G$ satisfies the $\C$-centraliser condition ($\C$-CC) in $G$ if for every subgroup $K \in \NC(G)$, there is a subgroup $L \in \NC(G)$ such that $L \leq K$ and 
\begin{displaymath}
    C_{G/L}(\pi_L(g)) \subseteq \pi_L(C_G(g)K) \mbox{ in } G/L,
\end{displaymath}
where $\pi_L \colon G \to G/L$ is the natural projection map.
\end{definition}
Clearly, a group $G$ satisfies the $\C$-centraliser condition if and only if every element $g \in G$ satisfies the $\C$-centraliser condition in $G$. 

For the ease of notation, and to improve readability, we introduce the following notation. Let $G$ be a group, and suppose that $g \in G$ and $K \in \NC(G)$ are arbitrary. We say that a subgroup $L \in \NC(G)$ is a $\C$-CC \emph{witness} for the pair $(g,K)$ if
\begin{displaymath}
    C_{G/L}(\pi_L(g)) \subseteq \pi_L(C_G(g)K) \mbox{ in } G/L
\end{displaymath}
where $\pi_L \colon G \to G/L$ is the natural projection map. Clearly, an element $g$ satisfies $\C$-CC in $G$ if and only if for every $K \in \NC(G)$, the pair $(g,K)$ has a $\C$-CC witness. Similarly, a group $G$ satisfies $\C$-CC if and only if every pair $(g,K)$, where $g \in G$ and $K \in \NC(G)$, has a $\C$-CC witness.

As the following lemma demonstrates, the property of being a witness for a given pair is inherited by subgroups.
\begin{lemma}
    \label{lemma:witness goes down}
    Let $G$ be a group and let $g \in G$ and $K, L \in \NC(G)$ be arbitrary. If $L$ is a $\C$-CC witness for the pair $(g,K)$, then every subgroup $L' \in \NC(G)$ such that $L' \leq L$ is also a $\C$-CC witness for $(g,K)$.
\end{lemma}
\begin{proof}
    Suppose that $L$ is a $\C$-CC witness for $(g,K)$, and let $L' \in \NC(G)$ be an arbitrary subgroup such that $L' \leq L$. Let $\pi \colon G \to G/L$ and $\pi' \colon G \to G /L'$ be the corresponding quotient maps. By assumption, 
    \begin{displaymath}
        C_{G/L}(\pi(g)) \subseteq \pi(C_G(g)K) \mbox{ in } G/L.
    \end{displaymath}
    Let $\lambda \colon G/L' \to G/L$ denote the map defined by $\lambda(xL') = xL$ for all $x \in G$. Thus, $\lambda$ is an epimorphism and $\pi = \lambda \circ \pi'$, and it can be seen that $C_{G/L'}(\pi'(g)) \subseteq \lambda^{-1}(C_{G/L}(\pi(g))$.  Altogether, we see that
    \begin{align*}        
        C_{G/L'}(\pi'(g))   &\subseteq \lambda^{-1}(C_{G/L}(\pi(g))\\
                            &\subseteq \lambda^{-1}(\pi(C_G(g)K)) \\
                            &= \pi'(C_G(g)K)L' = \pi'(C_G(g)K),
    \end{align*}
    meaning that $L'$ is a $\C$-CC witness for $(g,K)$.
\end{proof}

Since showing that the centraliser condition holds can be tedious, it is sometimes convenient to reduce the problem to constructing a well-behaved homomorphism to a group that satisfies the centraliser condition. Thus, we have the following lemma, which is an immediate consequence of \cite[Lemma~4.5]{mf}.

\begin{lemma}\label{lemma:lazy_lemma_CC}
    Let $G$ be a group. An element $g\in G$ satisfies $\C$-CC if and only if for every subgroup $K \in \NC(G)$, there exists a group $Q$ and a surjective homomorphism $\pi \colon G \to Q$ such that $\ker(\pi) \leq K$, $\pi(g)$ satisfies $\C$-CC in $Q$, and
    \begin{displaymath}
        C_{Q}(\pi(g)) \subseteq \pi(C_G(g)K) \mbox{ in }Q.
    \end{displaymath}
\end{lemma}
Note that there is no assumption on the group $Q$, though it will usually be infinite. In particular, if $Q = G$ and $\pi \in \Aut(G)$, we get the following immediate corollary of Lemma~\ref{lemma:lazy_lemma_CC}.
\begin{remark}
    \label{remark:lazy lemma automorphism}
    Let $G$ be a group, and let $g \in G$ be an arbitrary element. If there is an automorphism $\alpha \in \Aut(G)$ such that $\alpha(g)$ satisfies $\C$-CC in $G$, then $g$ satisfies $\C$-CC in $G$.
\end{remark}

\subsection{Wreath products}
Let $A$ and $B$ be groups. We denote the restricted wreath product of $A$ and $B$ as
    \begin{displaymath}
        A \wr B = \left(\prod_{b \in B} A \right) \rtimes B,
    \end{displaymath}
where $\prod$ denotes the restricted direct product, and $B$ acts on coordinates by left translation. An element $f \in \prod_{b \in B} A$ is understood as a function $f \colon B \to A$ such that $f(b) \neq 1$ for only finitely many elements $b \in B$. With a slight abuse of notation, we will use $A^B$ to denote $\prod_{b \in B} A$. The action of $B$ on $A^B$ is then realised as $b \cdot f(x) = f(b x)$. In the case when the group $A$ is abelian, we will write $A \wr B = \left(\bigoplus_{b \in B} A \right) \rtimes B$.

The \textbf{support} of $f$, i.e.~the elements on which $f$ does not vanish, is denoted as
\begin{displaymath}
    \supp(f) = \{b \in B \mid f(b) \neq 1\}.
\end{displaymath}

The proofs of the following two lemmas can be found in \cite[Subsection 4.1]{quantifying}.
\begin{lemma}[Lemma~4.2 in~\cite{quantifying}]
    \label{lemma:map_extension}
    Let $A, B$ be groups, and let $N \unlhd B$. If $A$ is abelian, then the natural projection $\pi_B \colon B \to B/N$ extends to a projection $\pi \colon A \wr B \to A \wr (B/N)$ with $\ker(\pi) = K_N \rtimes N$ where 
    \[
        K_N = \left\{f \in A^B \ \middle|\  \mathrm{for\ each}\ x\in B, \mathrm{\ we\ have}\ \prod_{b \in N} f(bx) = 1\right\}.
    \]
\end{lemma}
When the base group $A$ is abelian, it is sometimes convenient to use the additive notation for the group operation in $A$ and $A^B$, respectively. In that case, the subgroup $K_N$ from the lemma above is given by 
\begin{displaymath}
    K_N = \left\{f \in A^B \ \middle|\  \text{for each } x\in B, \text{ we have }\sum_{b \in N} f(bx) = 0\right\}.
\end{displaymath}
\begin{lemma}[Lemma~4.3 in~\cite{quantifying}]
\label{lemma:factoring_through_base}
    Let $G = A \wr B$ be a restricted wreath product of groups $A$ and $B$ such that $B \in \C$. If $K \in \NC(G)$ is a subgroup, then there is a subgroup $K_A \in \NC(A)$ such that $(K_A)^B \in \NC(G)$ and $(K_A)^B \leq K$.
\end{lemma}

\subsection{Centralisers in wreath products}\label{subsec: centraliser wreath product}
Centralisers in restricted wreath products have been studied by Meldrum in \cite{meldrum_centralizers}. In this subsection, we recall some of the notation introduced in \cite{meldrum_centralizers} and restate the results in the setting relevant to our paper, i.e.~either when the acting group $B$ is finite or when the base group $A$ is abelian. All the statements given in this subsection are direct applications of Meldrum's results, once one replaces the roles of the elements present according to the table
\begin{center}
\begin{tabular}{|c|c|}
    \hline
     Meldrum & this paper  \\\hline\hline
     $f$ & $f$\\\hline
     $g$ & $b$\\\hline
     $d$ & $g$\\\hline
     $h$ & $c$\\\hline
\end{tabular}
\end{center}
and replaces the right actions by left translations.

For $h \in A^B$, we define a function $\overline{h} \colon B \times B \to A$  as
\begin{displaymath}
    \overline{h}(d, x) =    \begin{cases}
                                h(x) h(dx) \cdots h(d^{n-1}x) & \mbox{if $\ord(d) < \infty$}\\
                                \prod_{i\in \mathbb{Z}} h(d^i x) &\mbox{if $\ord(d) = \infty$}, 
                            \end{cases}
\end{displaymath}
where $n = \ord(d)$ in the case when $d$ is of finite order. Since $h$ has non-trivial values only on finitely many elements of $B$, the above is well-defined.

The following lemma is a special restatement of \cite[Theorem~15]{meldrum_centralizers}.
\begin{lemma}
    \label{lemma:centralisers_finite}
    Let $A$ and $B$ be groups with $|B| < \infty$. Let $f,g \in A^B$ and $b,c \in B$ be arbitrary elements. Then $gc \in C_{A \wr B}(fb)$ if and only if all of the following hold:
    \begin{enumerate}[label=(\roman*)]
        \item\label{lemma:centralisers_finite/1}  $c\in C_B(b)$,
        \item\label{lemma:centralisers_finite/2} $\overline{f}(b, cx) \sim_A \overline{f}(b, x)$ for all $x\in B$,
        \item\label{lemma:centralisers_finite/3} $\overline{f}(b, cx) =  g(x)^{-1}\overline{f}(b, x) g(x)$ for all $x\in B$, and
        \item\label{lemma:centralisers_finite/4} $g(bx) = f^{-1}(x) g(x) f(cx)$ for all $x\in B$.
    \end{enumerate}
\end{lemma}

Let us note that if $c = 1$, then $\overline{g}(c, x) = g(x)$. In this particular case, we see that \ref{lemma:centralisers_finite/2} and \ref{lemma:centralisers_finite/3} in the statement of Lemma~\ref{lemma:centralisers_finite} are equivalent, i.e.~$fb \in C_{A \wr B}(g)$ if and only if $b \in C_B(c)$ and $g(bx) = f(x)^{-1}g(x)f(x)$ for every $x$. Similarly, if $g = 1$, then $fb \in C_{A \wr B}(c)$ if and only if $b \in C_B(c)$ and $f(cx) = f(x)$. In general, \ref{lemma:centralisers_finite/1} specifies admissible elements $b \in C_B(c)$, \ref{lemma:centralisers_finite/2} states that all possible values of $f(x)$ must belong to the same left coset of $C_A(\overline{g}(c,x))$, and \ref{lemma:centralisers_finite/3} gives a way of determining $f(cx)$ from $g(x), f(x)$, and $g(bx)$.

When $A$ is abelian, the value of $\overline{g}(c,x)$ is independent of the order in which we multiply the individual elements. In particular, given $h \in A^B$ and $d \in B$, if we fix the first input of $\overline{h}$ to $d$ (also known as ``currying'' $\overline{h}$ by $d$), we get a function $\overline{h}(d) \colon B \to A$ which is constant on right cosets of $\langle d \rangle$. In fact, $\overline{h}(d)$ can be seen as a function $\overline{h}(d) \colon \langle d\rangle \backslash B \to A$.

We can now state the following lemma, which is a combination of~\cite[Theorem~3 and Theorem~15]{meldrum_centralizers} in the case when the base group $A$ is abelian.
\begin{lemma}
\label{lemma:centralisers_abelian_base}
    Let $A$ and $B$ be groups, and suppose that $A$ is abelian.  Let $f,g \in A^B$ and $b,c \in B$ be arbitrary elements. Then $gc \in C_{A \wr B}(fb)$ if and only if all of the following hold:
    \begin{enumerate}
        \item[(i)]  $c\in C_B(b)$,
        \item[(ii)] $\overline{f}(b, cx) = \overline{f}(b, x)$ for all $x\in B$,
        \item[(iii)] $g(bx) =  g(x) f(x)^{-1} f(cx)$ for all $x\in B$,
        \item[(iv)] if $\ord(b) < \infty$ and $\overline{f}(c, x) \neq 1$ for some $x \in B$, then $\ord(c) < \infty$,
        \item[(v)] if $\ord(b) = \infty$, then for all $x \in B$ we have that
        \begin{displaymath}
            g(b^{k+n}x) = f(b^{k + n-1}x)^{-1} \cdots f(b^{k}x)^{-1}f(b^{k}cx) \cdots f(b^{k-n+1}cx)
        \end{displaymath}
        for all $n \geq 1$, where $k \in \mathbb{Z}$ is maximal such that $g(b^{k'}x) = 1$ for all $k' \leq k$.
    \end{enumerate}
\end{lemma}
Note that 
item (v) of the above lemma is a seemingly stronger statement than (ii) of \cite[Theorem~3]{meldrum_centralizers}, but a careful inspection of the proof of \cite[Theorem~3]{meldrum_centralizers} in fact shows that it is equivalent.

\section{Wreath products of HCS groups \texorpdfstring{$A \wr B$ where $|B| < \infty$}{where B is finite}}
The aim of this section is to show the following proposition.
\begin{proposition}
    \label{proposition:C_HCS_finite_wreath}
    If $A$ is a group and $B \in \C$, then the wreath product $G = A \wr B$ is a $\C$-hereditarily conjugacy separable group if and only if $A$ is a $\C$-hereditarily conjugacy separable group.
\end{proposition}
The proof of the implication from left to right is straightforward: if $A \wr B$ is a $\C$-hereditarily conjugacy separable group, then $A$ is a $\C$-hereditarily conjugacy separable group by Remark \ref{remark:retracts are C-HCS} because it is isomorphic to a $\C$-virtual retract of $A \wr B$. In particular, $A^B \in \NC(G)$ since $G/A^B = B \in \C$ and $A$ is a direct factor of $A^B$. The implication in the opposite direction, however, requires substantially more work. The proof relies on Theorem~\ref{theorem:CC_HCS}, meaning that we need to show that every element of the wreath product $A \wr B$ satisfies the $\C$-centraliser condition. Lemma~\ref{lemma:CC - B in C,g in B} deals with the case when an element belongs to $B$, Lemma~\ref{lemma:CC - B in C,g in A^B} deals with the case when an element belongs to $A^B$, and finally, Lemma~\ref{lemma:CC - B in C, g general} deals with the general case. Formally speaking, one only needs Lemma~\ref{lemma:CC - B in C, g general}, but we felt that, for the sake of readability, it would be beneficial to treat those cases separately.

Throughout this section, we will assume that $G  = A \wr B$, where $B \in \C$. The following lemma demonstrates that any element in the acting group $B$ satisfies the $\C$-centraliser condition in $G$.
\begin{lemma}
    \label{lemma:CC - B in C,g in B}
   Let $A$ be a $\C$-HCS group and suppose that $B \in \C$. If $b \in B$, then $b$ satisfies the $\C$-centraliser condition in $G = A\wr B$.
\end{lemma}
\begin{proof}

    Let $K \in \NC(G)$ be a subgroup. Following Lemma~\ref{lemma:factoring_through_base}, we see that there is a subgroup $L_A \in \NC(A)$ such that $(L_A)^B \in \NC(G)$ and $(L_A)^B \leq K$. Denote $L = (L_A)^B$. We note that $G/L = (A/L_A) \wr B$, and thus we let $\pi \colon A \wr B \to (A/L_A) \wr B$ denote the natural projection. 
    
    Let $g \in A^B$ and $c \in B$ be elements satisfying $\pi(g)c \in C_{G/L}(b)$. Following Lemma~\ref{lemma:centralisers_finite}, we see that this is the case if and only if
     \begin{enumerate}
        \item[(i)]  $c\in C_B(b)$, and
        \item[(iv)] $\pi(g)(x) = \pi(g)(bx)$ for all $x \in B$.
    \end{enumerate}
    Hence, $\pi(g)c \in C_{G/L}(b)$ if and only if $c$ commutes with $b$ and $\pi(g)$ is constant on right cosets of $\langle b \rangle$ in $B$. If $gc \notin C_G(b)$, there must be an element $x \in B$ such that $g$ is not constant on the coset $\langle b \rangle x$. By assumption, we then see 
    \begin{displaymath}
        \pi(g)(x) = \pi_A(g(x)) = g(x)L_A  = g(bx)L_A = \pi_A(g(bx)) = \pi(g)(bx)
    \end{displaymath}
    where $\pi_A \colon A \to A/L_A$ is the natural projection. Thus, we have that $g(x)^{-1} g(bx) \in L_A$ for every $x \in B$. 
    
    We now construct a function $k_x \colon B \to A$ so that $k_x$ is constant on the right coset $\langle b \rangle x$. Towards this end, let $i_i$ be the minimal index such that $g(b^{i_1 + 1}x) \neq g(b^{i_1}x)$ where $n = \ord(b)$,  and define a function $k_1 \colon B \to A$ as
    \begin{displaymath}
        k_1(y) =  \begin{cases}
                    g(y)^{-1}g(by) & \mbox{ if $y = b^{i_1+1}x$,}\\
                    1                   & \mbox{ otherwise.}
                \end{cases}
    \end{displaymath}
    Since $g(y)^{-1}g(bx) \in L_A$, we have $k_1 \in L$. We then have $$gk_1(b^{i + 1}x) = g(b^i x) = gk_1(b^i x)$$ for all $i \leq i_1 + 1$. Now, pick a minimal index $i_2 \in \{0, \dots, n-1\}$ such that $gk_1(b^{i_2 + 1}x) \neq gk_1(b^{i_2} x)$. It is straightforward that $i_1 < i_2$. Again, we define a function $k_2 \colon B \to A$ as
    \begin{displaymath}
        k_2(y) =  \begin{cases}
                    gk_1(y)^{-1}gk_1(by) & \mbox{ if $y = b^{i_2+1}x$,}\\
                    1                   & \mbox{ otherwise.}
                \end{cases}
    \end{displaymath}
    As before, we note that $k_2 \in L$ and $gk_1 k_2(b^{i + 1}x) = g k_1(b^i x) = gk_1 k_2(b^i x)$ for all $i \leq i_2 + 1$. Repeating this argument at most $n-1$ times, we get a finite sequence of functions $k_1, k_2, \dots, k_m \in L$, where $m \leq n-1$, such that $(g k_1 \dots k_m)(b^{i+1}x) = (g k_1 \dots k_m)(b^{i+1}x)$. We see that the function $k_x = k_1 \dots k_m \in L$ is such that $gk_x$ is constant on the right coset $\langle b \rangle x$.
    
    By repeating the previous procedure for every coset in $\langle b \rangle \backslash B$, we eventually construct a function $k \in A^B$ such that $gk(bx) = gk(x)$ for all $x \in B$ and $g L = gk L$. Thus, we have that $gk c \in C_G(b)$ and $gc = gkcL \leq  C_G(b) L$. Therefore,
    \begin{displaymath}
        \pi_L^{-1}(C_{G/L}(b)) \subseteq C_G(b)L \subseteq C_G(b)K
    \end{displaymath}
    and we see that $b$ satisfies $\C$-centraliser condition in $G$.
\end{proof}
Before we continue, we note that, using methods presented in the appendix, one might prove the lemma above using topological arguments in the pro-$\C$ completion of $A \wr B$. While this would give a shorter and less technical proof, we decided to use only elementary methods that do not use the pro-$\C$ completion.

This next lemma demonstrates that any function $f \colon B \to A$ of finite support satisfies the $\C$-centraliser condition in $G$.
\begin{lemma}
    \label{lemma:CC - B in C,g in A^B}
    If $f \in A^B$ where $B \in \C$ and $A$ is $\C$-HCS, then $f$ satisfies the $\C$-centraliser condition in $G$.
\end{lemma}
\begin{proof}
    Let $K \in \NC(G)$ be an arbitrary subgroup, and let $K_A \in \NC(A)$ be the subgroup of $A$ we obtain by applying Lemma~\ref{lemma:factoring_through_base} to $K$. For $g \in A^B$, denote
    \begin{displaymath}
         C^-_A(g) = \{c \in B \mid \mbox{there is an element $x_c\in B$ such that $g(x_c) \not\sim_A g(cx_c)$} \}.
    \end{displaymath}
    In particular, the set $C^-_A(g)$ is the collection of elements $c$ in $B$ satisfying $g \neq c^{-1} \cdot g$ as functions. In particular, $c \notin C_G(g)$. For each $c \in C^-_A(g)$, set 
    $$
    X_c = \{x \in B \mid g(c) \not\sim_A g(cx) \},
    $$
    which is the set of elements in $B$ such that $g(x) \neq (c^{-1} \cdot g)(x)$. Denote $C^+_A(g) = B \setminus C^-_A(g)$.
    As $A$ is $\C$-conjugacy separable, we see that for every $c \in C^-_A(g)$, there is a subgroup $N_c \in \NC(A)$ such that $\pi_{N_c}(g(x)) \not\sim_{A/N_c} \pi_{N_c}(g(cx))$ for all $x \in X_c$.
    Set $$N = K_A \cap  \bigcap_{c \in B^-}N_c \in \NC(A).$$ As the group $A$ satisfies the $\C$-centraliser condition, we see that for every $b \in B$, there is a subgroup $L_b \in \NC(A)$ such that $L_b \leq N \leq K_A$ and
    \begin{displaymath}
        C_{A/L_b}(\pi_b(g(b))) \subseteq \pi_b(C_A(g(b))N) \subseteq \pi_b(C_A(g(b))K_A),
    \end{displaymath}
    where $\pi_b \colon A \to A/N_b$ is the natural projection. Set $$L_A = \bigcap_{b \in B} L_b$$ and denote $L = (L_A)^B$. It is straightforward that $G/L = (A/L_A) \wr B \in \C$, and thus we see that $L_A \in \NC(G)$. Let $\pi \colon G \to G/L$ be the natural projection. With a slight abuse of notation, we will also use $\pi$ to denote the natural projection $A \to A/L_A$.
    
    Suppose that $g \in A^B$ and $c \in B$ are given such that $\pi(g)c \in C_{G/L}(\pi(f))$. Following Lemma~\ref{lemma:centralisers_finite}, we see that this is the case if and only if all of the following are true:
    \begin{enumerate}
        \item[(ii)] $\pi(f)(cx) \sim_{A/L_A} \pi(f)(x)$ for all $x\in B$,
        \item[(iii)] $\pi(f)(cx) =  (\pi(g)(x))^{-1}\pi(f)(x) \pi(g)(x)$ for all $x\in B$.
    \end{enumerate}
    
    As $L_A \leq N$, we see that $c$ cannot belong to $C^-_A(f)$. In particular, if we denote
    \begin{displaymath}
        C^-_{A/L_A}(\pi(f)) = \{c \in B \mid \mbox{there is $x_c\in B$ such that $\pi(f)(x_c) \not\sim_{A/L_A} \pi(f)(cx_c)$} \},
    \end{displaymath}
    and by setting $C^+_{A/L_A}(\pi(f)) = B \setminus C^-_{A/L_A}(\pi(f))$, we then see that $C^+_{A/L_A}(\pi(f)) = C^+_{A}(f)$.
    
    For each $c' \in C^+_{A}(f)$, pick some $h_{c'} \in A^B$ such that $f(c'x) = h_{c'}(x)^{-1} g(x) h_{c'}(x)$. Following Lemma~\ref{lemma:centralisers_finite}, we see that
    \begin{align*}
        C_G(f) &= \bigcup_{c' \in C^+_A(f)} h_{c'} C_{A^B}(g) c',\\
        C_{G/L}(\pi(f)) &= \bigcup_{c' \in C^+_A(f)} \pi(h_{c'}) C_{A^B /L}(\pi(f)) c'.
    \end{align*}
    
    Recall that, from construction of $L$, we have 
    \begin{displaymath}
        C_{A^B/L}(\pi(f)) = \prod_{b \in B} C_{A/L_A}(\pi(f(b))) \subseteq \prod_{b \in B}\pi(C_A(f(b))N) = \pi(C_{A^B}(f)N^B).
    \end{displaymath}
    In particular, we see that
    \begin{displaymath}
        \pi(gc) \in  \pi(g_{c}) C_{A^B /L}(\pi(f)) c \subseteq \pi(g_c)\pi(C_{A^B}(f)N^B) \subseteq \pi(C_G(f) N^B) \subseteq \pi(C_G(f)K).
    \end{displaymath}
    As $g \in A^B$, $c \in B$ were arbitrary, it follows $C_{G/L}(\pi(f)) \subseteq \pi(C_G(f)K)$, which means the element $f$ satisfies the $\C$-centraliser condition in $G$.
\end{proof}

For the following remark, see Subsection~\ref{subsec: centraliser wreath product} for the definition of $f(b,x)$.
\begin{remark}
    \label{remark:collecting_coset_values_in_a quotient}
Let $N \unlhd A$ be an arbitrary subgroup. Then, for every $f \in A^B$ and $b \in B$, we have $\overline{\pi_N(f)}(b,x) = \pi_N(\overline{f}(b,x))$ for all $x \in B$.
\end{remark}

The proof of the following lemma utilises the techniques used in the proofs of Lemmas~\ref{lemma:CC - B in C,g in B} and \ref{lemma:CC - B in C,g in A^B}.
\begin{lemma}
    \label{lemma:CC - B in C, g general}
    Let $b \in B$ and $f \in A^B$ be non-trivial elements. If $A$ is $\C$-hereditarily conjugacy separable and $B \in \C$, then the element $fb$ satisfies the $\C$-centraliser condition in $G$.
\end{lemma}
\begin{proof}
    Let $K \in \NC(G)$ be a subgroup, and let $K_A \in \NC(A)$ be the subgroup of $A$ obtained by applying Lemma~\ref{lemma:factoring_through_base} to $K$. Denote
    \begin{displaymath}
        C_A^+(f,b) = \{c \in B \mid \overline{f}(b,cx) \sim_A \overline{f} \mbox{ for all $x \in B$}\}
    \end{displaymath}
    and set $C_A^-(f,c) = B \setminus C_A^+(f,b)$. For $c \in C_A^-(f,c)$, denote 
    $$
    X_c = \{x \in B \mid \overline{f}(b,cx) \not\sim_A \overline{f}(b, x)\}.
    $$ 
    
    As $A$ is $\C$-conjugacy separable, we see that for each $c \in C_A^-(f,c)$, there is a subgroup $N_c \in \NC(A)$ such that $\pi_{N_c}(\overline{f}(b,cx)) \not\sim_{A/N_C} \pi_{N_c}(\overline{f}(b,x))$ for all $x \in X_c$.  
    
    Set 
    $$
    N = K_A \cap \bigcap_{c \in C_A^-(f,c)} N_c.
    $$
    As $A$ is $\C$-hereditarily conjugacy separable, we see that it satisfies the $\C$-centraliser condition by Theorem~\ref{theorem:CC_HCS}. It follows that for every $x \in B$, there is a subgroup $L_x \in \NC(A)$ such that $L_x \leq N$ and
    \begin{displaymath}
        C_{A/L_x}(\pi_{L_x}(\overline{f}(b,x))) \subseteq \pi_{L_x}(C_A(\overline{f}(b,x))N).
    \end{displaymath}
    Set $$L_A = \bigcap_{x \in B} L_x,$$ and let $\pi_A \colon A \to A/L_A$ be the natural projection. Let $\pi \colon A\wr B \to (A/L_A) \wr B$ be the canonical extension of $\pi_A$, and denoting $L = \ker(\pi)$, we see that $L = L_A^B$. We claim that $C_{G/L}(\pi(fb)) \subseteq \pi(C_G(fb)K)$. Suppose that $g \in A^B$ and $c \in B$ are given such that $\pi(gc) \in C_{G/L}(\pi(fb))$. Following Theorem~\ref{lemma:centralisers_finite}, we have that this is the case if and only if all of the following are satisfied:
    \begin{enumerate}
        \item[(i)] $c \in C_B(b)$,
        \item[(ii)] $\overline{\pi(f)}(b,cx) \sim_{A/L_A} \overline{\pi(f)}(b,x)$ for all $x \in B$,
        \item[(iii)] $\overline{\pi(f)}(b,cx) = \pi(g)(x)^{-1}\overline{\pi(f)}(b,x)\pi(g)(x)$ for all $x \in B$,
        \item[(iv)] $\pi(g)(bx) = \pi(f)(x)^{-1} \pi(g)(x) \pi(f)(cx)$.
    \end{enumerate}
    Following Remark \ref{remark:collecting_coset_values_in_a quotient}, we see that $\overline{\pi(f)}(x) = \pi(\overline{f}(x))$ for all $x \in B$. If we denote
    \begin{displaymath}
        C^+_{A/L_A}(\pi(f),b) = \left\{c \in B \ \middle|\  \overline{\pi(f)}(b,cx) \sim_{A/L_A} \overline{\pi(f)}(b,x) \mbox{ for all $x\in B$}\right\}
    \end{displaymath}
    and $C^-_{A/L_A}(\pi(f),b) = B \setminus C^+_{A/L_A}(\pi(f),b)$, we see that (ii) is equivalent to $c \in C^+_{A/L_A}(\pi(f),b)$. From the construction of $L_A$, we see that $C^+_{A/L_A}(\pi(f),b) = C^-_A(f,b)$. In particular, $\overline{f}(b,cx) \sim_A \overline{f}(b,x)$ for all $x \in B$.
    
    By (iii), we have that $\overline{\pi(f)}(b,cx) = \pi(g)(x)^{-1} \overline{\pi(f)}(b,x) \pi(g)(x)$. Suppose that $g_0 \in A^B$ is such that $\overline{f}(b,cx) = g_0(x)^{-1} \overline{f}(b,x) g_0(x)$ for all $x \in B$. It then follows that $\pi(g)(x) \in \pi(g_0)(x)C_{A/L_A}(\overline{\pi(f)}(b,x))$ for all $x \in B$. To ease notation, we define a new function $f_b \in A^B$ given by $f_b(x) = \overline{f}(b,x)$. From the construction of $L_A$, we see that
    \begin{displaymath}
        g(x) \in g_0(x) \pi_A^{-1}(C_{A/L_A}(\pi(f_b)(x))) \subseteq g_0(x) C_A(f_b(x))N \subseteq g_0(x) C_A(f_b(x))K_A.
    \end{displaymath}

    Finally, by (iv), we see that $\pi(g)(bx) = \pi(f)(x)^{-1} \pi(g)(x) \pi(f)(cx)$. This means that $g(bx)^{-1}f(x)^{-1}g(x)f(cx) \in L_A$ for all $x \in B$. Since $c \in C_B(b)$, we have $c b^i x = b^i cx$.
    Suppose there is an element $x \in B$ such that the values of $g_0$ on the right coset of $\langle b \rangle x$ do not satisfy the rule $g(b^{i+1}x) = f(b^i x)^{-1} g(b^i) f(b^i c x)$ for $i = 0, \dots, n-1$ where $n = \ord(b)$. Let $i_0 \in \{0, \dots, n-1\}$ be the smallest index such that $g_0(b^{i_0+1}x) \neq f(b^i_0 x)^{-1} g_0(b^i_0) f(b^i_0 c x)$. Define a function $k_0 \in A^B$ by
    \begin{displaymath}
        k_0(y) =    \begin{cases}
                        g_0(y)^{-1} f(b^{-1}y)^{-1} g_0(b^{-1}y)f(b^{-1}xy)  &\mbox{ if $y = b^{i_0+1}x$},\\
                        1   &\mbox{otherwise}.
                    \end{cases}
    \end{displaymath}
    It follows that $k_0 \in L \leq K$. Setting $g_1 = k_0 g_0$, we have that $$g_1(b^{i+1}x) = f(b^i x)^{-1} g_1(b^ix) f(b_i cx)$$ for all $0 \leq i \leq i_0$. Pick $i_1 \in \{i_0 +1 , \dots, n-2\}$ to be the smallest index such that $g_1(b^{i_1+1}x) \neq f(b^i_1 x)^{-1} g_0(b^i_1) f(b^i_1 c x)$. Define a function $k_1 \in A^B$ by
    \begin{displaymath}
        k_1(y) =    \begin{cases}
                        g_1(y)^{-1} f(b^{-1}y)^{-1} g_1(b^{-1}y)f(b^{-1}xy)  &\mbox{ if $y = b^{i_1+1}x$},\\
                        1   &\mbox{otherwise}.
                    \end{cases}
    \end{displaymath}
    We again see that $k_1 \in L \leq K$. Thus, by setting $g_2 = k_1 g_1 = k_1 k_0 g_0$, it is clear that $g_2(b^{i+1}x) = f(b^i x)^{-1} g_2(b^ix) f(b_i cx)$ for all $0 \leq i \leq i_1$. By repeating this process at most $(n-2)$ times, we get a sequence of functions $k_0, \dots, k_m \in L$, for some $m \in \{0, \dots, n-2\}$, such that $g_{m+1}(b^{i+1} x) = f(b^i x)^{-1}g_{m+1}(b^i x) f(b^i cx)$ for all $0 \leq i \leq n-2$. We can then write
    \begin{displaymath}
        f(b^{n-1}x)^{-1} g_{m+1}(b^{n-1}x) f(b^{n-1}cx) = \overline{f}(b,x)^{-1} g_0(x) \overline{f}(b,cx).
    \end{displaymath}
    However, by assumption, we have that $\overline{f}(b,cx) = g_0(x)^{-1} \overline{f}(b,x) g_0(x)$. In particular, we have that
    \begin{align*}
        f(b^{n-1}x)^{-1} g_{m+1}(b^{n-1}x) f(b^{n-1}cx) &= g_0(x)\\
                                                        &= g_{m+1}(b^n x).
    \end{align*}
    This means that $g_{m+1}c \in C_G(fb)$. Now, we have that $g_0 \in g_{m+1}L$ and $g \in g_0 C_A(f_b)K_A$. Thus, we see that $gc \in C_G(fb) K_A \subseteq C_G(fb) K$.
\end{proof}

\begin{proof}[Proof of Proposition~\ref{proposition:C_HCS_finite_wreath}]
    As stated before, if $A \wr B$ is $\C$-hereditarily conjugacy separable, then $A$ must be $\C$-hereditarily conjugacy separable as it is isomorphic to a $\C$-virtual retract of $A \wr B$.
    
    For the implication in the opposite direction, let us first note that $A \wr B$ is $\C$-conjugacy separable by Theorem~\ref{theorem:CCS}. If $A$ is $\C$-hereditarily conjugacy separable, then, combining Lemmas \ref{lemma:CC - B in C,g in B}, \ref{lemma:CC - B in C,g in A^B}, and \ref{lemma:CC - B in C, g general}, we see that every element of $A \wr B$ satisfies the $\C$-centraliser condition. It then follows by Theorem~\ref{theorem:CC_HCS} that $A \wr B$ is $\C$-hereditarily conjugacy separable.
\end{proof}

\section{Wreath products of HCS groups \texorpdfstring{$A \wr B$ where $B$}{A wr B where B} is infinite}\label{sec:wreath-prop-B-infinite}
The aim of this section is to prove the following proposition.
\begin{proposition}
    \label{proposition:infinite acting group}
    Let $A$ and $B$ be residually-$\C$ groups such that $A$ is abelian and $B$ is infinite. If $B$ satisfies $\C$-CC and every cyclic subgroup of $B$ is $\C$-separable in $B$, then $A \wr B$ satisfies $\C$-CC.
    
    Moreover, if $A$ and $B$ are groups such that $B$ is infinite, then the wreath product $A \wr B$ is a $\C$-HCS group if and only if both of the following hold:
    \begin{itemize}
        \item[(i)] $A$ is an abelian residually-$C$ group,
        \item[(ii)] $B$ is $\C$-HCS and every cyclic subgroup of $B$ is $\C$-closed in $B$.
    \end{itemize}
\end{proposition}
In line with the assumptions of Proposition~\ref{proposition:infinite acting group}, within this section, unless stated otherwise, we will assume that $G = A \wr B$ is a wreath product of groups such that the base group $A$ is abelian.

The section is organised as follows: in Subsection~\ref{subsection:centralisers with abelian base}, we give a full characterisation of centralisers in wreath products with abelian base group, in Subsection~\ref{subsection:Infinite act preliminaries}, we prove several technical lemmas about finding finite quotients of the acting group $B$ such that the extended homomorphism $A \wr B \to A \wr (B/K)$ has desired properties, and in Subsection~\ref{subsection:infinite act main statements}, we use these results to establish the $\C$-centraliser condition for $A \wr B$, thus proving Proposition~\ref{proposition:infinite acting group} and Theorem~\ref{theorem:HCS}.

\subsection{Centralisers in wreath products with abelian base group}
\label{subsection:centralisers with abelian base}
In this subsection, we introduce new notation that will allow us to simplify the statement of Lemma~\ref{lemma:centralisers_abelian_base} and, as a byproduct, we give a full characterisation of centralisers in wreath products with abelian base group.

First,  given a group $B$ and an element $b \in B$, the centraliser $C_B(b)$ has a left action on the space of right cosets $\langle b \rangle \backslash B$ given by $c \cdot \langle b \rangle x = \langle b \rangle cx$. Given a finitely supported function $f \colon B \to A$, where $A$ is an abelian group, it makes sense to consider the set of $c \in C_B(b)$ that stabilise the right cosets of $\langle b \rangle$ covering non-trivial fibres of $f$. This motivates the following definition.
\begin{definition}
    \label{definition:C(f,b)}
    Let $A$ and $B$ be groups such that $A$ is abelian. Given elements $f \in A^B$ and $b \in B$, we define $C^+_B(f,b) \subseteq B$ as
\begin{displaymath}
    C^+_B(f,b) = \{c \in B \mid \langle b \rangle c f^{-1}(a) = \langle b \rangle f^{-1}(a) \mbox{ for every } a \in \Im(f)\setminus\{1\}\}.
\end{displaymath}
We then define $C_B(f,b) = C^+_B(f,b) \cap C_B(b)$.
\end{definition}
We note that if $b = 1$, then
\begin{displaymath}
    C_B(f,1) = \{c \in B \mid c f^{-1}(a) = f^{-1}(a) \mbox{ for all } a \in \Im(f)\setminus\{1\}\}.
\end{displaymath}
Similarly, if $f = 1$, then $C_B(1, b) = C_B(b)$. As mentioned before, $C_B(f,b)$ can be equivalently described as the subgroup of $C_B(b)$ that stabilises (setwise) every collection of right cosets of $\langle b \rangle$ covering each non-trivial fibre of $f$. In other words,
\begin{displaymath}
    C_B(f,b) = \bigcap_{a \in \Im(f)\setminus\{1\}} \stab_{C_B(b)}(\langle b\rangle f^{-1}(a)),
\end{displaymath}
where $\stab_{C_B(b)}(\langle b\rangle f^{-1}(a))$ denotes the setwise stabiliser.

\begin{definition}\label{def:reduced}
    We say that an element $fb \in A \wr B$, where $f \in A^B$ and $b \in B$, is \emph{reduced} if all elements of $\supp(f)$ lie in distinct right cosets of $\langle b \rangle$ in $B$, i.e.~if $\supp(f) = \{s_1,\ldots,s_n\} \subseteq B,$ then $\left\langle b \right\rangle s_i \neq \left\langle b \right\rangle  s_j$ whenever $s_i \neq s_j$.
\end{definition}
The following lemma allows us to always assume that $fb$ is reduced.
\begin{lemma}[Lemma~5.9 in~\cite{quantifying}]\label{lemma:support in distinct cosets}
           Let $G = A \wr B$ with $A$ abelian, and let $w = fb \in A \wr B$ where $f \in A^B \backslash \{1\}$ and $b \in B \backslash \{1\},$ be arbitrary. Then there exists an element $w^\prime = f^\prime b \in w^{A^B}$ where $f^\prime \in A^B$ such that the element $f'b$ is reduced.
\end{lemma}

These notions allow us to simplify the statement of Lemma~\ref{lemma:centralisers_abelian_base}.
\begin{lemma}
    \label{lemma:simplified centraliser}
    Let $G = A \wr B$ with $A$ abelian, and let $f \in A^B$, $b \in B$ be elements such that element $fb$ is reduced. For $g \in A^B$ and $c \in B$, we have $gc \in C_G(fb)$ if and only if both of the following hold:
    \begin{enumerate}
        \item[(i)] $c \in C_B(f,b)$;
        \item[(ii)] $g(bx) = g(x) f(x)^{-1} f(cx)$ for all $x \in B$.
    \end{enumerate}
\end{lemma}
\begin{proof}
    Following Lemma~\ref{lemma:centralisers_abelian_base}, we see that $gc \in C_G(fb)$ if and only if all of the following are true:
    \begin{itemize}
        \item[(i)] $c \in C_B(b)$;
        \item[(ii)] $\overline{f}(b,cx) = \overline{f}(b,x)$ for all $x \in B$;
        \item[(iii)] $g(bx) = g(x) f(x)^{-1}f(cx)$.
    \end{itemize}
    As the elements of $\supp(f)$ lie in distinct cosets of $\langle b \rangle$, we see that $\overline{f}(b,s) = f(s)$ for any $s \in \supp(f)$. In particular, for any $s,s' \in \supp(f)$, it follows that $\overline{f}(b,s) = \overline{f}(b,s')$ if and only if $f(s) = f(s')$. Therefore, for any $x,x' \in B$ we see that $\overline{f}(b,x) = \overline{f}(b,x')$ if and only if either $x,x' \in \langle b \rangle f^{-1}(a)$ for some $a \in \Im(f)\setminus\{1\}$ or 
    $$
    \langle b \rangle x \cap \supp(f) = \emptyset = \langle b \rangle x' \cap \supp(f).
    $$
    We see that $\overline{f}(b,cx) = \overline{f}(b,x)$ for all $x \in B$ if and only $\langle b \rangle c f^{-1}(a) = \langle b \rangle f^{-1}(a)$ for all $x \in B$, meaning that the condition (ii) of Lemma~\ref{lemma:centralisers_abelian_base} is equivalent to $s \in C^+_B(f,b)$, assuming that the element $fb$ is reduced.
\end{proof}

As noted before, the centraliser $C_B(b)$ permutes right cosets of $\langle b \rangle$ by multiplication on the left and the subgroup $C_B(f,b) \leq C_B(b)$ consists of the elements that setwise fix collections of cosets covering non-trivial fibres of $f$. It makes sense to consider all potential permutations with this property, which motivates the following definition. For the following discussion, given a non-empty set $X$, we denote $\Sym(X)$ to be the group of permutations of $X$.
\begin{definition}
    \label{definition:sigma}
    Given non-trivial elements $f \in A^B$ and $b \in B$, such that the element $fb$ is reduced, we define $\Sigma(f,b)$ as
    \begin{displaymath}
        \Sigma(f,b) = \left\{\sigma \in \Sym(\supp(f)) \mid f^{-1}(a) = \sigma(f^{-1}(a)) \mbox{ for all $a \in \Im(f) \setminus \{1\}$}\right\}.
    \end{displaymath}
\end{definition}
Note that $\Sigma(f,b)$ is the subgroup of $\Sym(\supp(f))$ consisting of permutations that fix each non-trivial fibre of $f$ setwise, i.e.
\begin{displaymath}
    \Sigma(f,b) = \bigcap_{a \in \Im(f) \setminus \{1\}} \stab_{\Sym(\supp(f))}(f^{-1}(a)).
\end{displaymath}

Given some enumeration of $\supp(f) = \{s_1, \dots, s_n\}$, where $n = |\supp(f)|$, we will usually slightly abuse notation and consider elements of $\Sigma(f,b)$ as elements of $\Sym(|\supp(f)|)$, permuting the indices.

We note that the assumption that the element $fb$ is reduced is essential, as the subsets of the form $\langle b \rangle f^{-1}(a)$ might not be disjoint.

\begin{lemma}
    \label{lemma:virtually cyclic}
    If $f \in A^B$ and $b \in B$ are non-trivial elements such that $fb$ is reduced, then $C_B(f,b)$ is virtually cyclic. Indeed, the following short sequence is exact:
    \begin{displaymath}
         1   \rightarrow \langle b\rangle \rightarrow C_B(f,b) \rightarrow \Sigma(f,b) \rightarrow 1.
    \end{displaymath}
\end{lemma}
\begin{proof}
    Let $X \subseteq \langle b\rangle \backslash B$ be the collection of right cosets of $\langle b \rangle$ in $B$ corresponding to the elements of $\supp(f)$. As stated earlier, elements of $C_B(f,b)$ act on $X$ by $c \cdot \langle b\rangle x = \langle b\rangle cx$. Suppose that the elements $c_1, c_2 \in C_B(f,b)$ realise the same permutation on $X$. We then see that $\langle b \rangle c_1 s = \langle b \rangle c_2 s$ for all $s \in \supp(f)$, meaning that $c_1 \langle b \rangle = c_2 \langle b \rangle$. Clearly, $\langle b \rangle$ is the kernel of the action and $C_B(f,b)/\langle b \rangle = \Sigma(f,b)$.
\end{proof}

 Now, we can give a full characterisation of the structure of centralisers in the case when the base group $A$ is abelian.
\begin{proposition}
    \label{proposition:centraliser characterisation}
    Let $G = A \wr B$, where $A$ is abelian, and let $f \in A^B$ and $b \in B$ be arbitrary elements such that the element $fb$ is reduced. Then
    \begin{itemize}
        \item[(i)] if $\ord(b) = \infty$ and $f = 1$, then $C_G(b)=C_B(b)$,
        \item[(ii)] if $\ord(b) = \infty$ and $f \neq 1$, then the projection map $\pi_B \colon G \to B$ defines an isomorphism from $C_G(fb)$ to $C_B(f,b)$. In particular, $C_G(fb)$ is virtually cyclic;       
        \item[(iii)] $\ord(b) < \infty$, then 
        \begin{displaymath}
            C_G(b) \simeq  \left( \bigoplus_{i \in \langle b\rangle \backslash B} A \right) \rtimes C_B(f,b).
        \end{displaymath}
    \end{itemize}
\end{proposition}
\begin{proof}
    Let $g \in A^B$ and $c \in B$ be arbitrary elements such that $gc \in C_B(fb)$. By Lemma~\ref{lemma:simplified centraliser}, we see that
    \begin{enumerate}
        \item[(a)] $c \in C_B(f,b)$ 
        \item[(b)] $g(bx) = g(x) f(x)^{-1} f(cx)$ for all $x \in B$.
     \end{enumerate}

    Let $\{s_1,s_2, \dots, s_n \} = \supp(f)$ be some enumeration of the support of $f$. As $c \in C_B(f,b)$, we see that $\langle b \rangle cs \cap \supp(f) \neq \emptyset$ if and only if $\langle b \rangle s \cap \supp(f) \neq \emptyset$. In particular, as the element $fb$ is reduced, for each $s_i \in S$ there are unique $m_i, n_i \in \mathbb{Z}$ such that $f(b^{m_i}s_i) \neq 1$ and $f(b^{n_i}cs_i) \neq 1$ (in case if $\ord(b) < \infty$, we require $m_i, n_i \in \{0, 1, \dots, \ord(b)-1\}$). We note that $f(b^{m_i}s_i) = f(b^{n_i}cs_i)$.
     
    If (i) holds, then by (a), we see  that $c \in C_B(b)$, and by (b), we see that $g$ is constant on right cosets of $\langle b\rangle$. This means that $g(x) = 1$ for all $x \in B$ as $b$ is of infinite order and $g$ must be finitely supported.

    Suppose that (ii) holds. By (a), we see that $c \in C_B(f,b)$. We will now demonstrate that $g$ is fully determined by $f,b$, and $c$. From (b), we immediately see that $g(b^e x) = 1$ for all $e \in \mathbb{Z}$ whenever $\langle b \rangle x \cap \supp(f) = \emptyset$. Similarly, if $m_i = n_i$, then we see that $g(b^e s_i) = 1$ for all $e \in \mathbb{Z}$. If $m_i < n_i$, we then see that
    \begin{displaymath}
        g(b^e s_i) = \begin{cases}
            1 & \mbox{ if $e < m_i $},\\
            f(b^{m_i}s_i)^{-1} & \mbox{ if $m_i \leq e < n_i$},\\
            1 & \mbox{ if $n_i < e$}.
        \end{cases}
    \end{displaymath}
    Conversely, if $n_i < m_i$, we then see that
    \begin{displaymath}
        g(b^e s_i) = \begin{cases}
            1 & \mbox{ if $e < n_i $},\\
            f(b^{m_i}s_i) & \mbox{ if $n_i \leq e < m_i$},\\
            1 & \mbox{ if $m_i < e$}.
        \end{cases}
    \end{displaymath}
    Thus, the element $g$ is fully determined by $f,b$, and $c$, meaning that the restriction map $\pi_B \colon G \to B$ has a well-defined inverse map when restricted to $C_G(fb)$. In particular, for every $c \in C_B(f,b)$, we can set $\pi_B^{-1}(c) = g c$, where the function $g \colon B \to A$ is given as above. The assertion that $C_G(fb)$ is virtually cyclic then immediately follows, as $C_B(f,b)$ is virtually cyclic by Lemma~\ref{lemma:virtually cyclic}.

    Now, suppose that (c) is the case, i.e.~$\ord(b) < \infty$. We will use $\bigoplus_{i \in \langle b\rangle \backslash B} A$ to denote the set of all functions in $A^B$ that are constant on right cosets of $\langle b \rangle$, i.e.
    \begin{displaymath}
        \bigoplus_{i \in \langle b\rangle \backslash B} A = \{h \in A^B \mid h(bx) = h(x) \mbox{ for all $x \in B$}\}.
    \end{displaymath}
    We will show that  $g$ can be  decomposed as $g = h g'$, where $h \in \bigoplus_{i \in \langle b\rangle \backslash B} A$ and $g'$ is fully determined by $f,b$, and $c$.
    
    Let $X = \{x_i \in B \mid i \in \langle b\rangle \backslash B\}$ be some right transversal for $\langle b \rangle$ in $B$, and let $I \subseteq X$ be collection of representatives corresponding to the elements of $\supp(f)$. Without loss of generality, we may assume that $x_i = s_i$ for some $s_i \in \supp(f)$ whenever $i \in I$. Following (b), we see that $g$ is constant on the coset $\langle b \rangle x_i$ whenever $i \notin I$ or if $i \in I$ and $m_i = n_i$. For each $i \in I$ such that $m_i \neq n_i$, we pick some $e_i \in \{0, \dots, \ord(b)-1\} \setminus \{\min\{m_i, n_i\}, \max\{m_i, n_i\} - 1\}$, and then for each $x_i \in X$, we set 
    \begin{displaymath}
        a_i = \begin{cases}
            g(x_i)          &\mbox{ if $i \notin I$ or $m_i = n_i$},\\
            g(b^{e_i} x_i)  &\mbox{ if $i \in I$ and $m_i \neq n_i$}. 
        \end{cases}
    \end{displaymath}
    From (b), we see that for all $i \in I$ such that $m_i \neq n_i$, if $m_i < n_i$ have
    \begin{displaymath}
        g(b^e s_i) = \begin{cases}
            a_i & \mbox{ if $e < m_i $},\\
            a_i f(b^{m_i}s_i)^{-1} & \mbox{ if $m_i \leq e < n_i$},\\
            a_i & \mbox{ if $n_i < e$};
        \end{cases}
    \end{displaymath}
    and, conversely, if $n_i < m_i$, then we have
    \begin{displaymath}
        g(b^e s_i) = \begin{cases}
            a_i & \mbox{ if $e < n_i $},\\
            a_i f(b^{m_i}s_i) & \mbox{ if $n_i \leq e < m_i$},\\
            a_i & \mbox{ if $m_i < e$}.
        \end{cases}
    \end{displaymath}
    We set $h \in A^B$ to be the function given by $h(x) = a_i$ whenever $x \in \langle b \rangle x_i$. Thus, we see that $h \in \bigoplus_{i \in \langle b\rangle \backslash B} A$. Next, we set $g' = h^{-1} g$. Hence, if $i\notin I$ or $m_i = n_i$, then $g'$ is constant on the coset $\langle b \rangle x_i$. Furthermore, if $m_i < n_i$, then we see that
    \begin{displaymath}
        g'(b^e s_i) = \begin{cases}
            1 & \mbox{ if $e < m_i $},\\
            f(b^{m_i}s_i)^{-1} & \mbox{ if $m_i \leq e < n_i$},\\
            1 & \mbox{ if $n_i < e$};
        \end{cases}
    \end{displaymath}
    and, conversely, if $n_i < m_i$, we then see that
    \begin{displaymath}
        g'(b^e s_i) = \begin{cases}
            1 & \mbox{ if $e < n_i $},\\
            f(b^{m_i}s_i) & \mbox{ if $n_i \leq e < m_i$},\\
            1 & \mbox{ if $m_i < e$},
        \end{cases}
    \end{displaymath}
    meaning that $g'$ is fully determined by $f,b$, and $c$.
\end{proof}
We note that $g'$ in the case (iii) is not unique. However, it is unique if we require that $g'(x) = i$ whenever $i \notin I$ or $m_i = n_i$ and that $g'(b^e x_i)$ whenever $e \in \{0, \dots, \ord(b)-1\} \setminus \{\min\{m_i, n_i\}, \max\{m_i, n_i\}-1\}$. 

\subsection{Preliminary technical statements}
\label{subsection:Infinite act preliminaries}
In this subsection, we give several technical statements that will be crucial for proving Proposition~\ref{proposition:infinite acting group}.
\begin{lemma}
    \label{lemma:permuting cosets}
    Let $B$ be a residually-$\C$ group. Suppose that the element $b \in B$ is such that $\langle b \rangle$ is $\C$-closed in $B$, and suppose that $S_1, \dots, S_r$ is a collection of mutually disjoint finite subsets of $B$ such that $\langle b \rangle s_i \neq \langle b \rangle s_j$ whenever $s_i \neq s_j$, for any $s_i, s_j \in S = S_1 \cup \dots \cup S_r$.
    
    Then there exists a subgroup $L \in \NC(B)$ such that for the natural projection $\pi \colon B \to B/L$ we have $\langle \pi(b)\rangle \pi(s_i) \neq \langle \pi(b) \rangle \pi(s_j)$ whenever $s_i \neq s_j$. Furthermore, $L$ can be chosen so that for every $c \in B$ satisfying $\langle \pi(b) \rangle \pi(S_i) = \langle \pi(b) \rangle \pi(c)\pi(S_i)$ for all $i \in \{1,\dots, r\}$, we have $c = c'l$ for some $l \in L$ and $c' \in B$ satisfying $\langle b \rangle cS_i = \langle b \rangle S_i$ for all $i \in \{1, \dots, r\}$. Furthermore, $\pi(c) \in C_{B/L}(\pi(b))$ if and only if $c' \in C_B(b)$.
\end{lemma}
\begin{proof}
    Denote
    \begin{displaymath}
        \Sigma = \{ \sigma \in \Sym(S) \mid \sigma(S_i) = S_i\}
    \end{displaymath}
    and let $S = \{s_1, \dots, s_k\}$ be an enumeration of the set $S$. Set 
    $$
    D = \{s_i s_j^{-1} \mid s_i, s_j \in S \mbox{ and } i\neq j\}.
    $$ 
    As $\langle b \rangle s_i \neq \langle b \rangle s_j$ whenever $i \neq j$ by assumption, we see that $D \cap \langle b \rangle = \emptyset$. As $\langle b \rangle$ is $\C$-closed in $B$, there is a subgroup $L_1 \in \NC(B)$ such that $DL_1 \cap \langle b \rangle L_1 = \emptyset$. Suppose that $\langle b \rangle s_i L_1 = \langle b \rangle s_j L_1$ for some $i \neq j$. This means that $s_i s_j^{-1} L_1 \in \langle b \rangle L_1$. However, $s_i s_j^{-1} L_1 \subseteq DL_1$ and $D L_1 \cap \langle b\rangle L_1 = \emptyset$ by construction, which is a contradiction. We see that $\langle b \rangle s_i L_1 \neq \langle b \rangle s_j L_1$ whenever $i \neq j$, which proves the first part of the statement.
    
    For each permutation $\sigma \in \Sigma$, we denote 
    \begin{displaymath}
        D_\sigma = \{s_{\sigma(i)}s_i^{-1} s_j s_{\sigma(j)}^{-1} \mid i,j \in \{1, \dots, k\}\}.
    \end{displaymath}
    We set $\Sigma^+ = \{\sigma  \in \Sigma \mid D_\sigma \subseteq \langle b\rangle \}$ and $\Sigma^- = \Sigma \setminus \Sigma^+$. We see that $\Sigma^+$ denotes the set of all permutations of the set $S$ that can be realised as a left action of some element $c$ on the right coset space $\langle b\rangle\backslash B$ given by $c \cdot \langle b \rangle x = \langle b \rangle cx$ and $\Sigma^-$ are the permutations that cannot. Indeed, if there is an element $c \in B$ such that for all $s_i \in S$ we have $\langle b \rangle c s_i = \langle b \rangle s_{\sigma(i)}$, we then have
    \begin{displaymath}
        \langle b \rangle s_{\sigma(1)}s_1^{-1} = \dots = \langle b \rangle s_{\sigma(k)}s_k^{-1} = \langle b \rangle c.
    \end{displaymath}
    It then follows that, for any $i,j$, we must have $s_{\sigma(i)}s_i^{-1} s_j s_{\sigma(j)}^{-1} \in \langle b \rangle$. One can easily check that the implication holds in the opposite direction as well. In particular, if $D_\sigma \subseteq \langle b \rangle$, then there is $c \in B$ such that $\langle b \rangle cs_i = \langle b \rangle s_{\sigma(i)}$. Thus, if $D_\sigma  \subseteq \langle b\rangle$, then
    \begin{displaymath}
        \langle b \rangle s_{\sigma(1)}s_1^{-1} = \dots = \langle b \rangle s_{\sigma(k)}s_k^{-1}.
    \end{displaymath}
    Hence, there is some $c \in B$ such that $\langle b \rangle s_{\sigma(i)}s_i^{-1} = \langle b \rangle c$ for all $i$, meaning that $\langle b \rangle c s_i = \langle b \rangle s_{\sigma(i)}$.

    We set $D^- = \cup_{\sigma \in \Sigma^-} D_\sigma \setminus \langle b \rangle$, which one can see is a finite and non-empty subset of $B$, and, since $\langle b \rangle$ is $C$-closed in $B$, there exists a subgroup $L_2 \in \NC(B)$ such that $D^- L_2 \cap \langle b \rangle L_2 = \emptyset$. In particular, $s_{\sigma(i)}s_i^{-1}s_j s_{\sigma(j)}^{-1} \in \langle b \rangle L_2$ for all pairs $i,j \in \{1, \dots, k\}$ if and only if $\sigma \in \Sigma^+$.

    Before we proceed with the final part of the construction, let us make one crucial observation. All elements realising the same permutation $\sigma$ lie in the same right coset of $\langle b \rangle$. Indeed, given $\sigma \in \Sigma^+$ and $c_1, c_2 \in B$ such that $\langle b\rangle c_1 s_i = \langle b\rangle s_{\sigma(i)} = \langle b\rangle c_2 s_i$ for all $i$, we see that $\langle b \rangle c_1 = \langle b \rangle c_2$. In fact, it can be easily seen that any element from the same coset realises the same permutation. As $\langle b \rangle \leq C_B(b)$, we see that either $\langle b \rangle c \leq C_B(b)$ or $\langle b \rangle c \cap C_B(b) = \emptyset$, meaning that either all elements realising $\sigma$ commute with $b$ or none do.
    
    For each $\sigma \in \Sigma^+$, let us pick some element $c_\sigma \in B$ that realises $\sigma$.  If $[c_\sigma, b] \neq 1$, then there is a subgroup $L_\sigma \in \NC(B)$ such that $[c_\sigma, b] \notin L_\sigma$. We set $L_3 = \bigcap_{\sigma \in \Sigma^+} L_\sigma$. Finally, set $L = L_1 \cap L_2 \cap L_3$ and let $\pi \colon B \to B/L$ be the natural projection. 
    
    Now, suppose that there exists an element $c \in B$ such that $\langle b \rangle cS_i L = \langle b \rangle S_i L$ for all $i \in \{1, \dots, r\}$. This means that there is some permutation $\sigma \in \Sigma$ such that $\langle b \rangle c s_i L = \langle b \rangle s_{\sigma(i)}L$ for all $i \in \{1, \dots, k\}$. Hence, 
    \begin{displaymath}
         \langle b \rangle c s_{\sigma(1)}s_1^{-1}L = \dots = \langle b \rangle c s_{\sigma(k)}s_k^{-1}L = \langle b \rangle cL, 
    \end{displaymath}
    meaning that $s_{\sigma(i)}s_i^{-1}s_j s_{\sigma(j)}^{-1} \in \langle b \rangle L$ for any pair $i,j \in  \{1, \dots, k\}$. It then follows from the construction of $L$ that $\sigma \in \Sigma^+$. Indeed, if $\sigma \in \Sigma^-$, then there are some $i,j$ such that $s_{\sigma(i)}s_i^{-1}s_j s_{\sigma(j)}^{-1} \notin \langle b \rangle,$ and, therefore, $s_{\sigma(i)}s_i^{-1}s_j s_{\sigma(j)}^{-1} \notin \langle b\rangle L$. In particular, we see that 
    \begin{displaymath}
        \langle b \rangle s_{\sigma(1)}s_1^{-1} = \dots = \langle b \rangle s_{\sigma(k)}s_k^{-1} \subseteq \langle b \rangle c L.
    \end{displaymath}
    Hence, there is an element $l' \in L$ such that $\langle b \rangle s_{\sigma(i)}s_i^{-1} = \langle b \rangle cl'$ for all $i \in \{1, \dots, k\}$. We then set $c' = cl'$ and $l = (l')^{-1}$. It then follows that $c = c' l$ and $\langle b \rangle c S_i = \langle b \rangle S_i$ for all $i \in \{1, \dots, r\}$, which proves the second part of the statement.

    Finally, suppose that $c' \notin C_B(b)$.  Let $c_\sigma \in \langle b \rangle c'$ be the element we picked for the construction of $L_3$. Following the previous argument, we see that $c_\sigma \notin C_B(b)$. By construction, $\pi(c_\sigma) \notin C_{B/L}(\pi(b))$. By the same argument as before, we see that $\langle \pi(b) \rangle \pi(c_\sigma) = \langle \pi(b) \rangle \pi(c) \cap C_{G/L}(\pi(b))$, meaning that $\pi(c) \notin C_{B/L}(\pi(b))$. We see that $\pi(c) \in C_{B/L}(\pi(b))$ if and only if $c' \in C_B(b)$, which concludes the proof.
\end{proof}

Applying Lemma~\ref{lemma:permuting cosets} to the setting when $b = 1$ immediately gives us the following corollary.
\begin{corollary}   
    \label{corollary:permuting support}
    Let $B$ be a residually-$\C$ group, and suppose that $S_1, \dots, S_r$ is a collection of mutually disjoint finite subsets of $B$. Then there exists a subgroup $L \in \NC(B)$ such that the natural projection $\pi \colon B \to B/L$ is injective on the finite set $S = \cup_{i = 1}^r S_i$. Moreover, for every $c \in B$ satisfying $cS_i L = S_i L$ for all $i \in \{1,\dots, r\}$, we have $c = c'l$ for some $l \in L$ and $c' \in B$ satisfying $cS_i = S_i$ for all $i \in \{1, \dots, r\}$.
\end{corollary}

The lemma and the corollary we have just proven will be crucial in establishing $\C$-CC for elements of $A \wr B$ that have a non-trivial function part. Recall that, by Lemma~\ref{lemma:centralisers_abelian_base}, if $gc \in C_B(f)$, where $f, g \in A^B$ and $c \in B$, then $c$ must permute elements of $\supp(f)$ (or the corresponding right cosets of $\langle b \rangle$) by left multiplication. Informally speaking, Lemma~\ref{lemma:permuting cosets} and Corollary~\ref{corollary:permuting support} will allow us to construct a quotient of the acting group $B$ such that no new permutations can be realised in the quotient. Using the somewhat simpler criteria of Lemma~\ref{lemma:simplified centraliser}, we will be able to construct a quotient map $\lambda \colon B \to B/L$ such that $C_{B/L}(\lambda(f), \lambda(b)) = \lambda(C_B(f,b))$ and $\Sigma(f,b) \simeq \Sigma(\lambda(f), \lambda(b))$.
\begin{lemma}
    \label{lemma:admissible permutations}
    Let $f \in A^B$ and $b \in B$ be given such that the element $fb$ is reduced in $G = A \wr B$, and suppose that $L_A \unlhd A$ and $L_B \unlhd B$ are subgroups satisfying $\Im(f) \cap L_A = \{1\}$ and the following properties
    \begin{itemize}
        \item[(i)] for any $s, s' \in \supp(f)$ we have $\langle b \rangle s L_B = \langle b \rangle s' L_B$ if and only if $\langle b \rangle s = \langle b \rangle s'$;
        \item[(ii)] for any $c \in B$ such that $\langle b \rangle c f^{-1}(a) L_B = \langle b \rangle f^{-1}(a) L_B$ for all $a \in \Im(f)$ there is $c' \in cL_B$ such that $\langle b \rangle c f^{-1}(a) = \langle b \rangle f^{-1}(a)$ for all $a \in \Im(f)$;
        \item[(iii)] for any such $c$ we have $c' \in C_B(b)$ if and only if $c L_B \in C_{B/L_B}(b L_B)$.
    \end{itemize}
    Then the element $\lambda(f)\lambda(b)$ is reduced in $A/L_A \wr B/L_B$ and $C_{G/L}(\lambda(f),\lambda(b)) = \lambda(C_B(f,b))$, where $L$ is the kernel of the natural projection $\lambda \colon A \wr B \to A/L_A \wr B/L_B$.
\end{lemma}
\begin{proof}
    By assumption, we have that the elements of the support of $f$ lie in distinct cosets of $\langle b \rangle$, no two cosets of $\langle b \rangle$ corresponding to elements of $\supp(f)$ collide in $B/L$, and that the natural projection $\lambda_A \colon A \to A/L_A$ is injective on $\Im(f)$. We immediately see that $s\supp(\lambda(f)) = \lambda(\supp(f))$. Furthermore, we immediately see that the elements of $\supp(\lambda(f))$ lie in distinct cosets of $\langle \lambda(b)\rangle$ in $B/L_B$, meaning that the element $\lambda(f)\lambda(b)$ is reduced.

    It is straightforward that if $c\in C_B(f,b)$, then $\lambda(c)\in C_{B/L_B}(\lambda(f), \lambda(b))$. Now, suppose that $c \in B$ is given such that $\lambda(c) \in C_{B/L_B}(\lambda(f),\lambda(b))$. By definition, that means that $\langle b \rangle c \lambda(f)^{-1}(\overline{a}) L_B$ for all $\overline{a} \in \Im(\lambda(f))$ and $\lambda(c) \in C_{B/L_B}(\lambda(b))$. Since the natural projection $\lambda_A \colon A \to A/L_A$ is injective on $\Im(f)$, we see that $\lambda(f)^{-1}(\lambda_A(a)) = \lambda(f^{-1}(a)) = f^{-1}(a) L_B$ for every $a \in \Im(f) \setminus \{1\}$. Thus, we can rewrite the previous equality as $\langle b \rangle c f^{-1}(a) L_B = \langle b \rangle f^{-1}(a) L_B$ for every $a \in \Im(f) \setminus \{1\}$. By assumption, there is an element $c' \in c L_B$ such that $\langle b \rangle c' f^{-1}(a) = \langle b \rangle f^{-1}(a)$ for every $a \in \Im(f) \setminus \{1\}$ and $c' \in C_B(b)$, meaning that $c \in C_B(f,b)$, which concludes the proof.
\end{proof}

\subsection{Main statements}\label{subsection:infinite act main statements}

Throughout this section, we will assume that the acting group $B$ is an infinite $\C$-HCS group such that all of its cyclic subgroups are $\C$-separable. When appropriate, the proofs given in this section will be split into two cases:
\begin{enumerate}
    \item[(i)] $\ord(b) < \infty$, i.e.~$b$ is a torsion element;
    \item[(ii)] $\ord(b) = \infty$, i.e.~$b$ has infinite order.
\end{enumerate}

\begin{lemma}
    \label{lemma:elements of B have C-CC}
    Let $b \in B$ be arbitrary. Then $b$ satisfies the $\C$-centraliser condition in $G = A \wr B$.
\end{lemma}
\begin{proof}
    Let $K \in \NC(G)$ be a subgroup. We set $K_b = K\cap B$, and note that $K_B \in \NC(B)$. Since $b \in B$ satisfies the $\C$-centraliser condition in $B$, there exists a subgroup $L_B \in \NC(B)$ such that $L_B$ is a $\C$-CC witness for $(b, K_B)$ in $B$. 

    First, assume that $b$ is a torsion element. Since $\langle b\rangle$ is finite and $B$ is residually-$\C$, there is a subgroup $L'_B \in \NC(B)$ such that $\langle b \rangle \cap L'_B = \langle b \rangle$, i.e.~the natural projection map $B \to B/L'_B$ is injective on $\langle b \rangle$. Without loss of generality, we may replace $L_B$ by $L_B \cap L'_B$ which by Lemma~\ref{lemma:witness goes down} will still enjoy the property of being a $\C$-CC witness for $(g, K_B)$.

    Now, set 
    $$
    L_A = \left\{f \in A^B \ \middle|\  \forall c\in B \quad \prod_{l \in L_B}f(lc)=1 \right\}
    $$ 
    and let $L = L_A L_B$. Finally, let $\lambda \colon G \to G/L_AL_B = A \wr (B/L_B)$ be the natural projection.

    Let $h \in A^B$ and $c \in B$ be arbitrary elements. Following Lemma~\ref{lemma:centralisers_abelian_base}, we see that $hc \in C_G(b)$ if and only if the following are true:
    \begin{itemize}
      \item[(i)] $c \in C_B(b)$, and
        \item[(ii)] $h(x) = h(bx)$ for all $x\in B$.
    \end{itemize}
    In particular, we see that $h$ must be constant on right cosets of $\langle b\rangle$ in $B$. Similarly, we see that $\lambda(hc)\in C_{G/L}(\lambda(b))$ if and only if
    \begin{itemize}
        \item[(i)] $\lambda(c) \in C_{B/L_B}(\lambda(b))$, and
        \item[(ii)] $\lambda(h)(xL_b) = \lambda(h)(b)xL_b$ for all $x \in B$. 
    \end{itemize}
    Again, this means that $\lambda(h)$ must be constant on right cosets of $\langle \lambda(b)\rangle$ in $B/L_B$. Now, suppose that $\lambda(hc) \in C_{G/L}(\lambda(b))$. From the construction of the subgroup $L_B$, we see that $c \in C_B(b)K_B$. Thus, we will show that $h \in C_G(b)L$. First, let us recall that
    \begin{displaymath}
        \lambda(h)(xL_b) = \prod_{l\in L_B}h(lx).
    \end{displaymath}
    Let $x_1, x_2, \dots, x_n \in B$ be representatives of the double coset decomposition $\langle b\rangle \backslash B /L_B$ corresponding to the right cosets of $\langle\lambda(b) L_B\rangle$ in $B/L_B$ such that $\lambda(h)(x_i L_b) \neq 1$. We note that $\lambda_B$ is injective on $\langle b \rangle$ by construction, and thus it is injective on every right coset of $\langle b \rangle$ in $B$. We define a function $h' \colon B \to A$ as
    \begin{displaymath}
        h'(x) = \begin{cases}
                    h(x_i)  &\mbox{if $x \in \langle b\rangle x_i$},\\
                    1       &\mbox{ otherwise}.
                \end{cases}
    \end{displaymath}
    We see that $h'$ is constant on right cosets of $\langle b \rangle$ in $B$, and therefore, $h' \in C_G(b)$. Furthermore, $\lambda(h) = \lambda(h')$, which implies $h = h'l$ for some $l \in L$. As $L$ is normal, we see that $lc$ = $cl'$ for some $l' \in L$. Altogether, we see that
    \begin{displaymath}
        hc = h'lc = h'cl' \in C_G(b) C_B(b)K_b L \leq C_G(b)K,
    \end{displaymath}
    which means $C_{G/L}(\lambda(b))\subseteq \lambda(C_G(b)K)$. By construction, $G/L = A \wr (B/L_B)$ is a $\C$-HCS group by Proposition~\ref{proposition:C_HCS_finite_wreath}, and therefore, Theorem~\ref{theorem:CC_HCS} implies that $\lambda(b)$  satisfies the $\C$-centraliser condition in $G/L$. Following Lemma~\ref{lemma:lazy_lemma_CC}, we see that the element $b$ satisfies $\C$-CC in $G$.

    Now, let's assume that (ii) is the case, i.e.~the element $b$ has infinite order. Let $K_A \in \NC(A)$ be a maximal subgroup (with respect to inclusion) such that
    \begin{displaymath}
        \ker\left(\kappa' \colon A \wr B \to A/K_A \wr B/K_B \right) \leq K,
    \end{displaymath}
     and let $K'$ be the kernel of the projection described above. By construction, $K' \leq K$ and $K' \in \NC(A \wr B)$. By Lemma~\ref{lemma:order is a multiple}, there exists a subgroup $N_B \in \NC(B)$ such that the order of $b N_B$ in $B/N_B$ is a multiple of $|A/K_A| \ord(\kappa'(b))$. Set $L_B' = L_B \cap N_B$, and let $L$ be the kernel of the natural projection
    \begin{displaymath}
        \lambda \colon A \wr B \to A/K_A \wr B/L'_B.
    \end{displaymath}
    We see that $L \leq K$ and $L'_B$ is a $\C$-CC witness for $(b, K_B)$ in $B$ by Lemma~\ref{lemma:witness goes down}. Furthermore, the order of $bL'_B$ in $B/L'_B$ is $C' |A/K_A| \ord(\kappa'(b))$ for some $C' \in \mathbb{N}$. We will demonstrate that $L$ is a $\C$-CC witness for $(b, K)$ in $G$. Let $g \in A^B$ and $c \in C$ be arbitrary elements such that $\lambda(gc) \in C_{G/L}(\lambda(b))$. Following Lemma~\ref{lemma:centralisers_abelian_base}, we see that $\lambda(gc) \in C_{G/L}(\lambda(b))$ if and only if $\lambda(c) \in C_{B/L'_B}(\lambda(b))$ and $\lambda(g)$ is constant on right cosets of $\langle \lambda(b)\rangle$.

    Let $\kappa \colon A/K_A \wr B/L'_B \to A/K_A \wr B/L'_B$ be the unique projection map such that $\kappa = \kappa' \circ \lambda$. We will now demonstrate that $\kappa(\lambda(g)) = \kappa(g') = 1$, i.e.~that $g \in K$. Let $\overline{x} \in B/L'_B$ be arbitrary and, for the ease of writing, denote $\overline{g} = \lambda(g)$ and $\overline{b} = \lambda(b)$. We observe that
    \begin{displaymath}
        \kappa(\overline{g})(\kappa(\overline{x})) = \prod_{\overline{y} \in \kappa^{-1}(\kappa(\overline{x}))} \overline{g}(\overline{y}). 
    \end{displaymath}
    Note that if $\overline{y} \in \kappa^{-1}(\kappa(\overline{x}))$, then $\overline{b}^{\ord(\kappa'(b))}\overline{y} \in \kappa^{-1}(\kappa(\overline{x}))$. In particular, we see that if $\overline{y} \in \kappa^{-1}(\kappa(\overline{x}))$, then $\left\langle \overline{b}^{\ord(\kappa'(b))} \right\rangle \overline{y} \subseteq \kappa^{-1}(\kappa(\overline{x}))$. Consequently, there are some elements $\overline{x}_1, \dots, \overline{x}_n$ such that
    \begin{displaymath}
        \kappa^{-1}(\kappa(\overline{x})) = \bigsqcup_{i = 1}^{n}\left\langle \overline{b}^{\ord(\kappa'(b))} \right\rangle \overline{x_i}.
    \end{displaymath}
    As $\overline{g}(\overline{y}) = \overline{g}(\overline{b}\overline{y})$ for all $\overline{y} \in B/L'_B,$ we see that
    \begin{displaymath}
                \kappa(\overline{g})(\kappa(\overline{x})) = \prod_{i =1}^{n} \overline{g}(\overline{x}_i)^N,     
    \end{displaymath}
    where $N = \left|\left\langle \overline{b}^{\ord(\kappa'(b))} \right\rangle\right|$ in $B/L'_B$. Recall that $\ord(\overline{b}) = C' |A/K_A| \ord(\kappa'(b))$, meaning that $N = C' |A/K_A|$. This means that we can further write
    \begin{displaymath}
                \kappa(\overline{g})(\kappa(\overline{x})) = \prod_{i =1}^{n}\left( \overline{g}(\overline{x}_i)^{|A/K_A|}\right)^{C'} = \prod_{i =1}^{n}\left( 1\right)^{C'} = 1,     
    \end{displaymath}
    demonstrating that $\overline{g} \in \ker(\kappa)$, as $\overline{x} \in B/L'_B$ was arbitrary.
    This means that $\kappa'(g) = \kappa \circ \lambda(g) = 1$, i.e.~$g \in \ker(\kappa') = K' \leq K$.
    Finally, we note that $\lambda(c) \in \lambda(C_B(b))K_B$, as $L'_B$ is a $\C$-CC witness for $(b, K_B)$.
    We see that $gc \in C_B(b)K' \subseteq C_G(b)K$, meaning that $L$ is a $\C$-CC witness for $(b, K)$ in $G$.
    This shows that the element $b$ satisfies $\C$-CC in $G$.
\end{proof}

We now demonstrate that every element of the form $f \in A^B$ satisfies the $\C$-centraliser condition in $A \wr B$.
\begin{lemma}
    \label{lemma:elements in A^B have C-CC}
    Let $f \in A^B$ be arbitrary. Then $f$ satisfies the $\C$-centraliser condition in $G$.
\end{lemma}
\begin{proof}
    Let $K \in \NC(G)$ be an arbitrary subgroup, and set $K_B = K \cap B$. We then see that 
    $$
    \supp(f) = \bigcup_{a \in \Im(f) \setminus \{1\}} f^{-1}(a),
    $$ 
    which is a finite subset. By Corollary~\ref{corollary:permuting support}, there is a subgroup $L'_B \in \NC(B)$ such that the natural projection $\lambda'_B \colon B \to B/L'_B$ is injective on $\supp(f)$ and for every $c \in B$ such that $c f^{-1}(a) L'_b = f^{-1}(a)L'_B$ in $B/L'_B$ for all $a \in \Im(f) \setminus \{1\}$, there is an element $c' \in B$ such that $\lambda_B(c') = cL'_B$ and $c' f^{-1}(a) = f^{-1}(a)$ for all $a \in \Im(f)$. 

    Set $L_B = K_B \cap L'_B$, and let 
    $$
    L_A = \left\{f \in A^B : \forall c\in B \quad \prod_{l \in L_B}f(lc)=1 \right\}.
    $$
    Also set $L = L_A L_B$. Let $\lambda \colon G \to G/L_AL_B = A \wr B/L_B$ be the natural projection. Since $A$ is abelian, both $A^{B}$ and $A^{B/L_B}$ are abelian as well. That means 
    \begin{displaymath}
        C_{A^{B/B_L}}(\lambda(f)) = A^{B/L_B} = \lambda(A^B) = \lambda(C_{A^B}(f)).
    \end{displaymath}
    Since $G/L = A^{B/L_B} \rtimes B/L_B$ and $G = A^B \rtimes B$, we see that 
    $$
    C_{G/L}(\lambda(f)) = A^{B/L_B} C_{B/L_B}(\lambda(f))
    $$
    and $C_G(f) = A^B C_B(f)$. Suppose that an element $c \in B$ is given such that $\lambda(c) \in C_{B/L_B}(\lambda(f))$. As noted in the proof of Lemma~\ref{lemma:elements of B have C-CC}, this is the case if and only if $\lambda(f)(\lambda(c) x) = \lambda(f)(x)$ for all $x \in B/L_B$. That means
    \begin{equation}
        \label{eq:commutes}
        \prod_{l \in L_B} f(cx'l) = \lambda(f)(\lambda(c) x) = \lambda(f)(x) = \prod_{l \in L_B} f(x'l),
    \end{equation}
    where $x' \in B$ is an arbitrary element such that $\lambda(x') = x$. Therefore, $\lambda(c)\lambda(f)^{-1}(a) = \lambda(f)^{-1}(a)$ for all $a \in \Im(\lambda(f))$. Since $\lambda$ is injective on the set $\supp(f)$, we have $\supp(\lambda(f)) = \supp(f)L_B$, $\Im(\lambda(f)) = \Im(f)$, and $\lambda(f)^{-1}(a) = f^{-1}(a)L_B$ for all $a \in \Im(f) \setminus \{1\}$. Keeping this in mind, equation \ref{eq:commutes} is equivalent to saying that $c f^{-1}(a) L_B = f^{-1}(a) L_B$ for all $a \in \Im(f)\setminus \{1\}$. From the construction of $L_B$, we see that there is an element $c' \in B$ such that $\lambda(c') = \lambda(c)$ and $c' f^{-1}(a) = f^{-1}(a)$ for all $a \in \Im(f)\setminus \{1\}$. That means $f(c'x) = f(x)$ for all $x \in B$, and therefore, $c' \in C_B(f)$. Hence,
    \begin{displaymath}
        C_{G/L}(\lambda(f)) = A^{B/L_B} C_{B/L_B}(\lambda(f)) = \lambda(A^B C_B(f)) = \lambda(C_G(f)) \subseteq \lambda(C_G(f)K).
    \end{displaymath}
    By construction, $G/L = A \wr (B/L_B)$ is a $\C$-HCS group by Proposition~\ref{proposition:C_HCS_finite_wreath} and, therefore, $\lambda(f)$ satisfies the $\C$-centraliser condition in $G/L$ by Theorem~\ref{theorem:CC_HCS}. Following Lemma~\ref{lemma:lazy_lemma_CC}, we see that the element $f$ satisfies $\C$-CC in $G$.
    \end{proof}

The proof of the following lemma will rely heavily on Lemma~\ref{lemma:permuting cosets}, in a way similar to how the proof of Lemma~\ref{lemma:elements in A^B have C-CC} relied on Corollary~\ref{corollary:permuting support}. Recall that, by Lemma~\ref{lemma:centralisers_abelian_base}, if $gc \in C_B(fb)$, where $f, g \in A^B$ and $f, c \in B$, then, by assuming that elements of $\supp(f)$ lie in distinct right cosets of $\langle b \rangle$, $c$ must permute right cosets of $\langle b \rangle$ the form $\langle b \rangle x$ for which $\overline{f}(b, x) \neq 1$. In particular, if we set $S_i = f^{-1}(a_i)$, where $\{a_1, \dots, a_n\} = \Im(f)\setminus \{1\}$, then $\langle b \rangle c S_i = \langle b \rangle S_i$ for every $i$. Informally speaking, Lemma~\ref{lemma:permuting cosets} will allow us to construct a quotient of the group $B$ such that no new permutations can be realised in the quotient.
\begin{lemma}\label{lemma:general elements have C-CC}
    Let $h \in A^B$ and $b \in B$ be non-trivial. Then the element $hb$ satisfies $\C$-CC in $G$.
\end{lemma}
\begin{proof}
    Let $K \in \NC(G)$ be an arbitrary subgroup, and set $K_B = K \cap B$. Following Lemma~\ref{lemma:support in distinct cosets}, we see that there is an inner automorphism $\phi \in \Inn(G)$ such that $\phi(hb) = fb$ and the elements of $\supp(f)$ lie in distinct right cosets of $\langle b \rangle$ in $B$. We will show that the element $fb$ satisfies $\C$-centraliser condition in $B$.

    By Lemma~\ref{lemma:permuting cosets}, there is a subgroup $N_B \in \NC(B)$ such that
    \begin{itemize}
        \item[(i)] if $\langle b \rangle s N_B = \langle b \rangle s' N_B$, then $s = s'$ for all $s, s' \in S$;
        \item[(ii)] if there is some element $c \in B$ such that $\langle b \rangle c f^{-1}(a) N_B= \langle b \rangle f^{-1}(a) N_b$ for all non-trivial elements $a \in \Im(f)$, then there is $c' \in c N_B$ such that $\langle b \rangle c f^{-1}(a) = \langle b \rangle f^{-1}(a)$ for all non-trivial $a \in \Im(f)$;
        \item[(iii)] for all such elements $c$ we have $c N_B \in C_{B/N_B}(b N_B)$ if and only $c' \in C_B(b)$.
    \end{itemize}
    We now split the proof into two cases, depending on whether $b$ is a torsion element or not. \newline

    \noindent \textbf{Case 1:} $\ord(b) < \infty$. \newline
    There is a subgroup $H_B \in \NC(B)$ such that $\langle b \rangle \cap H_B = \{1\}$, i.e.~the cyclic subgroup $\langle b \rangle$ embeds into $B/H_B$.
    Set $L_B = K_B \cap N_B$, and let 
    \begin{displaymath}
         L_A = \left\{f \in A^B \mid \forall c\in B \quad \prod_{l \in L_B}f(lc)=1 \right\}.
    \end{displaymath}
    Also set $L = L_A L_B$. Define $\lambda \colon G \to G/L_AL_B = A \wr B/L_B$ as the natural projection. Now, let $g \in A^B$ and $c \in B$ be given such that $\lambda(gc) \in C_{G/L}(\lambda(fb))$. We will show that there are elements $g' \in A^B$ and $c' \in B$ such that $g'c' \in C_B(fb)$ and $\lambda(g'c') = \lambda(gc)$. Following Lemma~\ref{lemma:admissible permutations}, we see that the element $\lambda(g)\lambda(c)$ is reduced and that $C_{G/L}(\lambda(f), \lambda(b)) = \lambda(C_B(f,b))$.
    Following Lemma~\ref{lemma:simplified centraliser}, we see that $\lambda(gc) \in C_{G/L}(\lambda(fb))$ if and only if all of the following are true:
    \begin{itemize}
        \item[(i)] $\lambda(c) \in C_{B/L_B}(\lambda(f), \lambda(b))$,
        \item[(ii)] $\lambda(g)(bx L_B) = \lambda(g)(x L_b) \lambda(f)(xL_b)^{-1} \lambda(f)(cx L_B)$ for all $x \in B$.
    \end{itemize}
    As mentioned before, $C_{B/L_B}(\lambda(f), \lambda(b)) = \lambda(C_B(f,b))$ by Lemma~\ref{lemma:admissible permutations}. This means that there is an element $c' \in C_B(f,b)$ such that $\lambda(c') = \lambda(c)$, i.e.~$c' \in c L_B$.
    
    Finally, let $x_1, \dots, x_r \in B$ be some representatives of the double cosets of $\langle b\rangle \backslash B /L_B$ corresponding to the elements of $\supp(\lambda)(g)$ in $B/L_B$. If $x_i L_B = s L_B$ for some $s \in \supp(f)$, then set $x_i = s$. Similarly, if $c'x_iL_b = s L_B$ for some $s \in \supp(f)$, then set $x_i = c'^{-1}s$. We now define a function $g' \colon B \to A$ in the following way: for each $x_i$, we set $g'(x_i) = \lambda(g)(x_iL_b),$ and then, for $j \in \{0, \dots, \ord(b)-2\},$ we set inductively $g'(b^{j+1}x_i) = g'(b^j x_i) f(b^j x_i)^{-1} f(c' b^j x_i)$. For any other $x \in B$, we set $g'(x) = 1$. From the construction of $g'$, it is clear that $\lambda(g') = \lambda(g)$. Similarly, we see that $g'(bx) = g'(x)f(x)^{-1}f(c'x)$ for all $x \in b$. 

    Following Lemma~\ref{lemma:simplified centraliser}, we see that $g'c' \in C_G(fb)$ as $c' \in C_B(f,b)$ and $g'(bx) = g'(x)f(x)^{-1}f(c'x)$ for all $x \in b$. This then implies 
    $$
    C_{G/L}(\lambda(fb)) \subseteq \lambda(C_G(fb)) \subseteq \lambda(C_G(fb)K).
    $$ 
    Furthermore, we see that the group $G/L$ satisfies $\C$-CC by Theorem~\ref{theorem:CC_HCS} as it is $\C$-HCS by Proposition~\ref{proposition:C_HCS_finite_wreath}. By Lemma~\ref{lemma:lazy_lemma_CC}, we see that the element $fb$ has $\C$-CC in $G$, and therefore, the element $hb$ has $\C$-CC in $G$ by Remark \ref{remark:lazy lemma automorphism}, as $\phi(hb) = fb$ for some $\phi \in \Inn(G) \subseteq \Aut(G)$. \newline

    \noindent \textbf{Case 2:} $\ord(b) = \infty$. \newline
    The proof will combine the techniques introduced in the proofs of Lemmas~\ref{lemma:elements of B have C-CC} and~\ref{lemma:elements in A^B have C-CC}. However, significantly more work is required. The intuition for why comes from Proposition~\ref{proposition:centraliser characterisation}: $C_G(fb)$ is ``very small'' - it is virtually cyclic, so extra care needs to be applied when constructing the quotient in order to maintain control over the size of the centraliser of the image of $fb$ in the quotient. We employ methods similar to those used in the proof of Proposition~\ref{proposition:centraliser characterisation} to achieve this.
    
    We start by constructing our $\C$-CC witness. Let $K_A \in \NC(A)$ be a maximal subgroup (with respect to inclusion) such that
    \begin{displaymath}
        \ker\left(\kappa \colon A \wr B \to A/K_A \wr B/K_B \right) \leq K
    \end{displaymath}
    and let $K'$ be the kernel of the natural projection $\kappa$ described above. By construction, $K' \leq K$ and $K' \in \NC(A \wr B)$. Let $H_A \in \NC(A)$ be a subgroup such that the natural projection from $A$ to $A/H_A$ is injective on the set $\Im(f)$. Note that since $\Im(f)$ is finite and $A$ is residually-$\C$, such a subgroup $H_A$ exists. Set $L_A = K_A \cap H_A$.

    Let $X_b = \{x_i \in B \mid i \in \langle b\rangle\backslash B\}$ be a right transversal for $\langle b \rangle$ in $B$, and set 
    $$
    I = \{ i \in \langle b\rangle\backslash B \mid \langle b \rangle x_i \cap \supp(f) \neq \emptyset\}.
    $$
    As the element $fb$ is reduced, we see that if $\langle b \rangle x_i \cap \supp(f) \neq \emptyset$, then there is a unique $n_i \in \mathbb{Z}$ such that $f(b^{n_i}x_i) \neq 1$. Replacing $x_i$ by $b^{-n_i}x_i$, we can without loss of generality assume that $n_i = 0$. For every $\sigma \in \Sigma(f,b)$, we pick some $c_\sigma \in C(f,b)$ realising $\sigma$. We see that $\langle b \rangle x_i \cap \supp(f) \neq \emptyset$ if and only if $\langle b \rangle c_{\sigma}x_i \cap \supp(f) \neq \emptyset$, and thus, if $i \in I$, there is a unique $n_i(c_\sigma) \in \mathbb{Z}$ such that $f(b^{n_i(c_\sigma)}c_\sigma x_i) \neq 1$. Again, replacing $c_\sigma$ by $b^{n_\sigma}c_\sigma$ for some appropriate $n_\sigma < 0$, we can without loss of generality assume that $n_i(c_\sigma) \geq 0$ for all $i \in I$. We then set
    \begin{displaymath}
        N(f,b) = \max_{i \in I, \sigma \in \Sigma(f,b)}\{n_i(c_\sigma) \}.
    \end{displaymath}
    By Lemma~\ref{lemma:order is a multiple}, there exists a subgroup $H_B \in \NC(B)$ such that order of $b H_B$ in $B/H_B$ is a multiple of $|A/K_A| \ord(\kappa'(b))$. As $B$ is residually-$\C$, we can assume that $b, b^2, \dots, b^{N(f,b)} \notin H_B$, so we can assume that order of $bH_B$ in $B/H_B$ is greater than $N(f,B)$.
     
     Set $L_B = K_B \cap N_B \cap H_B$ and let $L$ be the kernel of the natural projection
    \begin{displaymath}
        \lambda \colon A \wr B \to A/L_A \wr B/L_B.
    \end{displaymath}
    We will demonstrate that $L$ is a $\C$-CC witness for $(fb, K)$ in $G$. In particular, we will demonstrate that for every pair $c \in B$ and $g \in A^B$ such that $\lambda(gc) \in C_{G/L}(\lambda(fb))$, there are $c' \in B$, $g' \in A^B$, $\overline{h} \in {A/L_a}^{B/L_B}$ such that the following hold:
    \begin{itemize}
        \item $g'c' \in C_G(fb)$;
        \item $\overline{h} \in \lambda(K)$;
        \item $\lambda(gc) = \overline{h}\lambda(g'c')$;
    \end{itemize}
    thus demonstrating that $C_{G/L}(\lambda(fb))\subseteq \lambda(C_G(fb)K)$. The rest of the proof will be split into several individual claims for reader's convenience.

    Let $g \in A^B$ and $c \in C$ be arbitrary elements such that $\lambda(gc) \in C_{G/L}(\lambda(fb))$. By Lemma~\ref{lemma:admissible permutations}, we see that the element $\lambda(f)\lambda(b)$ is reduced in $G/L$, and thus, by Lemma~\ref{lemma:simplified centraliser}, we see that:
    \begin{itemize}
        \item[(i)] $\lambda(c) \in C_{B/L_B}(\lambda(f), \lambda(b))$
        \item[(ii)] $\lambda(g)(bx L_B) = \lambda(g)(x L_b) \lambda(f)(xL_b)^{-1} \lambda(f)(cx L_B)$ for all $x \in B$.
    \end{itemize}
    Again, $C_{B/L_B}(\lambda(f), \lambda(b)) = \lambda(C_B(f,b))$ by Lemma~\ref{lemma:admissible permutations}. This means that there is an element $c' \in C_B(f,b)$ such that $\lambda(c') = \lambda(c)$. Recall that, by Definition \ref{definition:sigma}, $\Sigma(f,b)$ is the set of all permutations (of the right cosets in $\langle b \rangle \backslash B$ corresponding to the elements of $\supp(f)$) that can be realised by the left action of an element from $C(f,b)$. There is a unique permutation $\sigma \in \Sigma(f,b)$ such that $\langle b \rangle c' = \langle b \rangle c_\sigma$, and  thus, there is a unique integer $k \in \mathbb{Z}$ such that $c' = b^k c_\sigma$. Without loss of generality, we may assume that $0 \leq k < \ord(\lambda(b))$.
    
    For increased readability, let us denote $\overline{f} = \lambda(f)$ and $\overline{g} = \lambda(g)$. \newline
    \noindent \textbf{Claim:} There are elements $g' \in A^B$ and $\overline{h} \in A/L_A^{B/L_B}$ such that $\overline{g} = \lambda(g') \overline{h}$, $g'c' \in C_B(fb)$, and where $\overline{h}$ is constant on right cosets of $\langle \overline{b} \rangle$ in $B/L_B$. \newline

    \noindent \textbf{Proof of Claim:}\newline
    The argument relies heavily on Lemma~\ref{lemma:simplified centraliser}, which, informally speaking, tells us that the function part of an element belonging to a centraliser is almost completely determined by the element it centralises and its acting part.
    
    Recall that the element $\overline{f}\overline{b}$ is reduced, and thus, for every element $\overline{x} \in B/L_B$, we have $|\langle \overline{b}\rangle \overline{x} \cap \supp(\overline{f})| \leq 1$. Furthermore, let us note that $\langle \overline{b}\rangle \overline{x} \cap \supp(\overline{f}) = \emptyset$ if and only if $\langle \overline{b}\rangle \overline{c} \overline{x} \cap \supp(\overline{f}) = \emptyset$ because $\overline{c} \in C_{B/L_B}(\overline{f},\overline{b})$. It then immediately follows that if $\langle \overline{b} \rangle \overline{x}_0 \cap \supp(\overline{f}) = \emptyset$ for some $\overline{x}_0$, then $\overline{g}$ is constant on the coset $\langle \overline{b}\rangle\overline{x}_0$ since $\overline{g}(\overline{b}\overline{x}) = \overline{g}(\overline{x}) \overline{f}(\overline{x})^{-1}\overline{f}(\overline{c}\overline{x})$.
    
    Let $X_{\overline{b}} = \{x_j \in B \mid j \in J\} \subseteq X_b$ be a set of representatives of the double coset space $\langle b \rangle \backslash B / L_B$ and denote $\overline{x_j} = \lambda(x_j)$. Following previously established notation, $\{\overline{x}_i \mid i \in I \}$ is the set of representatives corresponding to the support of $\overline{g}$. By construction, for each $i \in I$, we have
    \begin{displaymath}
        \overline{f}(\overline{b}^e \overline{x}_i) = \begin{cases}
            \lambda_A(f(x_i)) \neq 1   &\mbox{ if $e = 0$ }\\
            1               &\mbox{ otherwise}.
        \end{cases}
    \end{displaymath}
    We note here that the integer $e$ is considered modulo $\ord(\overline{b})$. Recall that, for each $i \in I$, there is a unique integer $0 \leq n_i(c_\sigma)$ such that $f(b^{n_i(c_\sigma)} c_{\sigma}x_i) = f(x_i) \neq 1$, and since $c' = b^k c_\sigma$, we see that $f(b^e c'x_i) \neq 1$ if and only if $e = n_i(c_\sigma) - k$. It then follows that 
    \begin{displaymath}
        \overline{f}(\overline{b}^e \overline{c}\overline{x}_i) = \begin{cases}
            \lambda_A(f(x_i)) \neq 1   &\mbox{ if $e = n_i(c_\sigma) - k$ },\\
            1               &\mbox{ otherwise}.
        \end{cases}
    \end{displaymath}
    We note that $0 \leq n_i(c_\sigma) < \text{ord}(\overline{b})$ by assumption and, therefore, $|n_i(c_\sigma) - k| < \ord(\overline{b})$. For each $j \in I$, we set $\overline{a}_j \in A/L_A$ as follows:
    \begin{displaymath}
        \overline{a}_j = \begin{cases}
            \overline{g}(\overline{x}_j) & \mbox{ if $j \notin I$},\\
            \overline{g}(\overline{x}_j) & \mbox{ if $j \in I$ and $0 \leq n_i(c_\sigma) - k$},\\
            \overline{g}(\overline{b}\overline{x}_j) & \mbox{ if $j \in I$ and $n_i(c_\sigma) - k < 0$}.
        \end{cases}
    \end{displaymath}
    From the identity $\overline{g}(\overline{b}\overline{x}) = \overline{g}(\overline{x}) \overline{f}(\overline{x})^{-1}\overline{f}(\overline{c}\overline{x})$, we immediately see that if $j \notin I$ or $j \in I$ and $n_j(c_\sigma)-k = 0$, then the function $\overline{g}$ is constant on the coset $\langle \overline{b} \rangle \overline{x}_j$ and the unique value it attains is $\overline{a}_j$. Furthermore, if $j \in I$ and $n_j(c_\sigma) - k \neq 0$, we see that, depending on whether $n_j(c_\sigma) - k < 0$ or $n_j(c_\sigma) - k > 0$, either 
    \begin{displaymath}
        \overline{g}(\overline{b}^e x_j) = \begin{cases}
            \overline{a}_j \lambda_A(f(x_j))^{-1}   &\mbox{ if $1 \leq e \leq n_i(c_\sigma)-k$,}\\
            \overline{a}_j                          &\mbox{ otherwise }
        \end{cases}
    \end{displaymath}
    if $0 < n_i(c_\sigma)-k$, or
    \begin{displaymath}
        \overline{g}(\overline{b}^e x_j) = \begin{cases}
            \overline{a}_j \lambda_A(f(x_j))   &\mbox{ if $e = 0$ or $\leq k - n_i(c_\sigma) +1 \leq e \leq \ord(\overline{b})-1$},\\
            \overline{a}_j                          &\mbox{ otherwise }
        \end{cases}
    \end{displaymath}
    if $n_j(c_\sigma) - k < 0$.
    We define a function $\overline{h} \in A/L_B^{B/L_B}$ defined on cosets of $\langle\overline{b} \rangle$ in the following way: $\overline{h}(\overline{x}) = \overline{a}_j$ whenever $\overline{x} \in \langle \overline b\rangle \overline{x_j}$. We set $\overline{g}' = \overline{h}^{-1}\overline{g}$.
    We now define a function $g' \in A^B$ given by
    \begin{displaymath}
        g'(b^e x_i) = \begin{cases}
            f(x_i)^{-1} &\mbox{ if $i \in I$ and $0 \leq e \leq n_i(c_\sigma)-k$},\\
            f(x_i)      &\mbox{ if $i \in I$ and $n_i(c_\sigma)-k <  e \leq 0$},\\
            1           &\mbox{ otherwise}.
        \end{cases}
    \end{displaymath}
    We immediately see that $\lambda(g') = \overline{g}'$. Also, we see that $g'(b x) = g'(x)f(x)^{-1}f(c'x)$ for all $x \in B$, and therefore, $g'c' \in C_G(fb)$ by Lemma~\ref{lemma:simplified centraliser} as $c' \in C_B(f,b)$. This finishes the proof of the above claim. \newline
    
    To finish our proof, let $\kappa \colon A/K_A \wr B/L_B \to A/L_A \wr B/L_B$ be the unique projection map such that $\kappa = \kappa' \circ \lambda$. \newline
    \noindent \textbf{Claim 2:} $\kappa(\overline{h}) = 1$ or, equivalently, $\overline{h} \in \lambda(K')$.\newline
    
    \noindent \textbf{Proof of Claim 2:} The argument is analogous to the one used in the infinite-order case in the proof of Lemma~\ref{lemma:elements of B have C-CC}. Letting $\overline{x} \in B/L'_B$ be arbitrary, we observe that
    \begin{displaymath}
        \kappa(\overline{h})(\kappa(\overline{x})) = \prod_{\overline{y} \in \kappa^{-1}(\kappa(\overline{x}))} \overline{h}(\overline{y}). 
    \end{displaymath}
    Note that if $\overline{y} \in \kappa^{-1}(\kappa(\overline{x}))$, then $\overline{b}^{\ord(\kappa'(b))}\overline{y} \in \kappa^{-1}(\kappa(\overline{x}))$. In particular, we see that if $\overline{y} \in \kappa^{-1}(\kappa(\overline{x}))$, then $\left\langle \overline{b}^{\ord(\kappa'(b))} \right\rangle \overline{y} \subseteq \kappa^{-1}(\kappa(\overline{x}))$. Consequently, there are some elements $\overline{x}_1, \dots, \overline{x}_n$ such that
    \begin{displaymath}
        \kappa^{-1}(\kappa(\overline{x})) = \bigsqcup_{i = 1}^{n}\left\langle \overline{b}^{\ord(\kappa'(b))} \right\rangle \overline{x_i}.
    \end{displaymath}
    As $\overline{h}(\overline{y}) = \overline{h}(\overline{b}\overline{y})$ for all $\overline{y} \in B/L'_B$, we see that
    \begin{displaymath}
                \kappa(\overline{h})(\kappa(\overline{x})) = \prod_{i =1}^{n} \overline{h}(\overline{x}_i)^N     
    \end{displaymath}
    where $N = \left|\left\langle \overline{b}^{\ord(\kappa'(b))} \right\rangle\right|$ in $B/L'_B$. Recall that $\ord(\overline{b}) = C' |A/K_A| \ord(\kappa'(b))$, meaning that $N = C' |A/K_A|$. This means that we can further write
    \begin{displaymath}
                \kappa(\overline{g})(\kappa(\overline{x})) = \prod_{i =1}^{n}\left( \overline{g}(\overline{x}_i)^{|A/K_A|}\right)^{C'} = \prod_{i =1}^{n}\left( 1\right)^{C'} = 1,     
    \end{displaymath}
    demonstrating that $\overline{h} \in \ker(\kappa) = \lambda(K')$, as $\overline{x} \in B/L'_B$ was arbitrary, which proves the claim.\newline 
    
    We have demonstrated that if $\lambda(gc) \in C_{G/L}(\lambda(fb))$, then there are elements $c' \in B$, $g'\in A^B$, and $\overline{h} \in A/L_A ^{B/L_B}$ such that $g'c' \in C_G(fb)$, $\overline{H} \in \lambda(K')$, and $\lambda(gc) = \overline{h}\lambda(g'c')$, meaning that 
    \begin{displaymath}
        \lambda(gc) = \lambda(g')\overline{h} \lambda(c') \in \lambda(g')\lambda(c') \kappa^{-1}(1) = \lambda(g'c')\lambda(K') \subseteq \lambda(g'c')\lambda(K).
    \end{displaymath}
    Therefore, we see that $C_{G/L}(\lambda(fb)) \leq \lambda(C_G(fb)K)$, meaning that $L$ is a $\C$-CC witness for $(fb, K)$ in $G$. This shows that the element $fb$ has $\C$-CC in $G$, and thus, by Remark \ref{remark:lazy lemma automorphism}, we see that the element $hb$ has $\C$-CC in $G$.
\end{proof}

We are now ready to prove Proposition~\ref{proposition:infinite acting group}.
\begin{proof}
    Suppose that $A$ and $B$ are residually-$\C$ groups such that $B$ is infinite. Furthermore, assume that $B$ satisfies $\C$-CC and every cyclic subgroup of $B$ is $\C$-separable in $B$. Combining Lemmas~\ref{lemma:elements of B have C-CC}, \ref{lemma:elements in A^B have C-CC}, and~\ref{lemma:general elements have C-CC}, we see that every element of $A \wr B$ satisfies $\C$-CC, and thus $A \wr B$ satisfies $\C$-CC.

    To prove the `moreover' part of the statement, let us suppose that $A$ is an abelian residually-$\C$ group and $B$ is an infinite $\C$-HCS group such that every cyclic subgroup of $B$ is $\C$-closed in $B$. By Theorem~\ref{theorem:CCS}, we see that $G = A \wr B$ is a $\C$-CS group. By Theorem~\ref{theorem:CC_HCS}, we see that $B$ satisfies $\C$-CC and, by the first part of the proposition, we get that $G$ satisfies $\C$-CC. Thus, by Theorem~\ref{theorem:CC_HCS}, we see that $G$ is a $\C$-HCS group. For the implication in the opposite direction, let $A$ and $B$ be groups such that $B$ is infinite and suppose that $G = A \wr B$ is a $\C$-HCS group. First, let us note that $B$ is a retract in $G$, and thus $B$ must be $\C$-HCS by Remark \ref{remark:retracts are C-HCS}. Then by Theorem~\ref{theorem:CCS}, we see that $A$ must be an abelian residually-$\C$ group and every cyclic subgroup of $B$ must be $\C$-separable in $B$, which concludes the proof.
\end{proof}

We can now state the main result.
\MainTheoremHCS
\begin{proof}
  The proof follows immediately from Propositions~\ref{proposition:infinite acting group} and~\ref{proposition:C_HCS_finite_wreath}.
\end{proof}

\section{
Hereditary conjugacy separability of the Grigorchuk group}
\label{section:grigorchuk}

In this section, we prove that the Grigorchuk group is hereditarily conjugacy separable. In particular, we show that every finite-index subgroup of the Grigorchuk group is conjugacy separable.
We begin by defining the Grigorchuk group as follows.

For each binary sequence $\sigma = \sigma_1\sigma_2\cdots\sigma_k \in \{0,1\}^*$, we write $|\sigma| = k$ for its length.
We then view these sequences as the vertices of an infinite binary rooted tree $\Tree$ with root given by the empty sequence, $\varepsilon$, and each vertex $\sigma\in \{0,1\}^*$ having two children given as $\sigma 0$ and $\sigma 1$.
We then write $\Aut(\Tree)$ for the group of all graph automorphisms of this tree.
Notice that every automorphism $x\in \Aut(\Tree)$ preserves the root and each level of the tree.
Moreover, we see that $\Aut(\Tree) = \Aut(\Tree)\wr \Sym(2)$, that is, for each automorphism $x\in \Aut(\Tree)$, we may write $x = (x_0, x_1) \cdot s$ where each $x_i\in \Aut(\Tree)$ is the action that $x$ has on the subtree rooted at $i$, and $s\in \Sym(2)$ is the permutation that the element performs on the first level of the tree.

Given an automorphism $x = (x_0, x_1)\cdot s \in \Aut(\Tree)\wr \Sym(2) = \Aut(\Tree)$, we write $x|_0 = x_0$ and $x|_1 = x_1$ for the \emph{restrictions} of $x$ to its first-level subtrees.
For each sequence $\sigma \in \{0,1\}^*$, we then define the \emph{restriction} $x|_\sigma$ recursively such that $x|_\varepsilon = x$, $x|_{\alpha 0} = (x|_\alpha)|_0$, and $x|_{\alpha 1} = (x|_\alpha)|_1$ for each $\alpha \in \{0,1\}^*$.

The Grigorchuk group is defined as $\Gamma = \left\langle a,b,c,d\right\rangle\subset\Aut(\Tree)$ where the actions of the generators $a$, $b$, $c$, and $d$ are defined recursively as
\begin{align*}
    a(0\sigma) &= 1 \sigma &
    b(0\sigma) &= 0 a(\sigma) &
    c(0\sigma) &= 0 a(\sigma) &
    d(0\sigma) &= 0\sigma
    \\
    a(1\sigma) &= 0 \sigma &
    b(1\sigma) &= 1 c(\sigma) &
    c(1\sigma) &= 1 d(\sigma) &
    d(1\sigma) &= 1 b(\sigma)
\end{align*}
for each $\sigma \in \{0,1\}^*$.
Alternatively, we may think of the generators the Grigorchuk group as the automorphisms of $\Tree$ described in Figure~\ref{fig:grigorchuk-generators}.

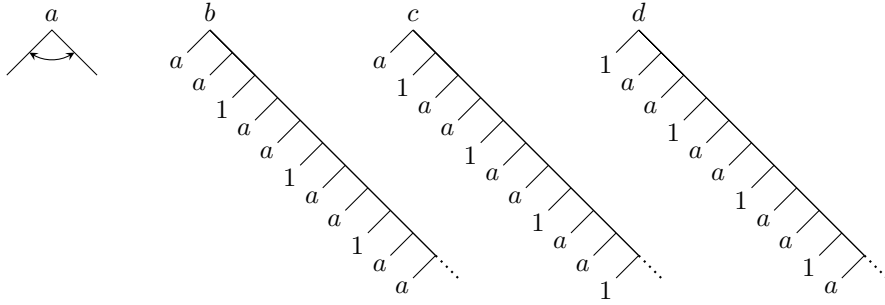
\begin{figure}[ht!]
\centering
\begin{tikzpicture}[>=stealth,scale=0.3]

\draw (0,0) -- (-2,2) node [above] {$a$} -- (-4,0);
\draw [<->] (-3,1) to [bend right] (-1,1);

\begin{scope}[shift={(7,0)}]
    \draw (0,0) -- (-2,2) node [above] {$b$} -- (8,-8);
    \foreach \x in {-2,-1,1,2,4,5,7,8} {
    	\draw (\x,-\x) to +(-1,-1) node [below left=-0.3em] {$a$};
    }
    \foreach \x in {0,3,6} {
    	\draw (\x,-\x) to +(-1,-1) node [below left=-0.3em] {$1$};
    }
    \draw [dotted,thick] (8,-8) to +(1,-1);
    \draw [] (-2,2) to (8,-8); 
\end{scope}
    
\begin{scope}[shift={(16,0)}]
    \draw (0,0) -- (-2,2) node [above] {$c$} -- (8,-8);
    \foreach \x in {-2,0,1,3,4,6,7} {
    	\draw (\x,-\x) to +(-1,-1) node [below left=-0.3em] {$a$};
    }
    \foreach \x in {-1,2,5,8} {
    	\draw (\x,-\x) to +(-1,-1) node [below left=-0.3em] {$1$};
    }
    \draw [dotted,thick] (8,-8) to +(1,-1);
    \draw [] (-2,2) to (8,-8); 
\end{scope}

\begin{scope}[shift={(26,0)}]
    \draw (0,0) -- (-2,2) node [above] {$d$} -- (8,-8);
    \foreach \x in {-1,0,2,3,5,6,8} {
    	\draw (\x,-\x) to +(-1,-1) node [below left=-0.3em] {$a$};
    }
    \foreach \x in {-2,1,4,7} {
    	\draw (\x,-\x) to +(-1,-1) node [below left=-0.3em] {$1$};
    }
    \draw [dotted,thick] (8,-8) to +(1,-1);
    \draw [] (-2,2) to (8,-8);
\end{scope}
\end{tikzpicture}
\caption{Generators of Grigorchuk group}
\label{fig:grigorchuk-generators}
\end{figure}

Let $\ell\colon \Gamma\to \mathbb{N}$ be the norm on $\Gamma$ induced by the word metric with respect to the generating set $X = \{a,b,c,d\}$.
We then have the following lemma, which relates the length of an element to the length of its restrictions.

\begin{lemma}[Lemma~8.2 and Corollary~8.3 in~\cite{grigorchuk2008groups}]\label{lem:length-contraction}
Let $g\in \Gamma$, then
\begin{enumerate}
\item if $\ell(g) \equiv 0 \bmod 2$, then $\ell(g|_0), \ell(g|_1)\leq \ell(g)/2$; and
\item if $\ell(g) \equiv 1 \bmod 2$, then
\begin{enumerate}
\item $\ell(g|_0), \ell(g|_1)\leq \ell(g)/2 + 1$ and
\item $\ell(g|_0) + \ell(g|_1) \leq \ell(g) + 1$.
\end{enumerate}
\end{enumerate}
Hence, we have $\ell(g|_0) + \ell(g|_1) \leq \ell(g) + 1$ for each $g\in \Gamma$.
\end{lemma}

For each $n\geq 1$, we define the \emph{$n$-th level stabiliser subgroup} as
\[
    \mathrm{Stab}_\Gamma(n)
    =
    \{
        g\in \Gamma
    \mid
        g\cdot v = v
        \text{ for each }
        v\in \{0,1\}^*\text{ with }|v|\leq n
    \}.
\]
We notice here that each subgroup $\mathrm{Stab}_\Gamma(n)$ is of finite index and normal in $\Gamma$.
These stabiliser subgroups characterise the finite-index subgroups of $\Gamma$ as follows.

\begin{lemma}[Theorem~3.7~in~\cite{Noce2023}]\label{lem:csp}
    The Grigorchuk group has the \emph{congruence subgroup property (CSP)}, that is, for each finite-index subgroup $H\leq \Gamma$, there exists some $n\geqslant 1$ such that $\mathrm{Stab}_\Gamma(n) \leq H$.
\end{lemma}

We define the function $\Psi_\Gamma\colon \mathrm{Stab}_\Gamma(1)\to \Gamma\times \Gamma$ as
\[
    \Psi_\Gamma(g)
    =
    (g|_0, g|_1)
\]
for each $g\in \mathrm{Stab}_\Gamma(1)\leq \Gamma$.
Notice then that $\Psi_\Gamma$ is an injective homomorphism.
We make use of this function in the following definition.

\begin{definition}\label{def:lift-function}
Suppose that $A,B\leq \mathrm{Stab}_\Gamma(1)$ are finite-index normal subgroups of $\Gamma$ for which $B\times B\leq \Psi_\Gamma(A)$.
Then we define a function
\[
    \Lift_{B,A}\colon
    \Gamma/B \times \Gamma/B  \to \mathcal{P}(\mathrm{Stab}_\Gamma(1)/A)
\]
such that
$g A \in \Lift_{B,A}(g_0 B, g_1 B) $
for each $g = (g_0,g_1) \in \mathrm{Stab}_\Gamma(1)$.
\end{definition}
This function has the following important property.

\begin{lemma}\label{lem:branch-subgroup-prop}
    Let $A,B\leq \mathrm{Stab}_\Gamma(1)$ be finite-index normal subgroups of $\Gamma$ such that $B\times B\leq \Psi_\Gamma(A)$.
    For each pair of elements $g_0,g_1\in \Gamma$, 
    the element given as $g = (g_0,g_1)\in \Aut(T)$ belongs to $\mathrm{Stab}_\Gamma(1)$
    if and only if $\Lift_{B,A}(g_0 B, g_1 B)$ is non-empty.
    Moreover, for such an element $g$, we have $gA \in \Lift_{B,A}(g_0 B, g_1 B)$.
\end{lemma}

\begin{proof}
Let $g_0,g_1\in \Gamma$ be a pair of elements such that $c\in \Lift_{B,A}(g_0 B, g_1 B)$.
Then from the definition of the function $\Lift_{B,A}$, we see that there exists some element $h = (h_0, h_1)\in \mathrm{Stab}_\Gamma(1)$ for which $c = h A$, $g_0 B = h_0 B$, and $g_1 B = h_1 B$.
Let $b_0,b_1\in B$ be such that $g_0 = h_0 b_0$ and $g_1 = h_1 b_1$.

We see that $(b_0, b_1) \in B\times B \leq \Psi_{\Gamma}(A)$, and thus $b = \Psi_\Gamma^{-1}(b_0, b_1)\in A$.
Hence, our desired element is given by $g = hb \in \Gamma$, that is, $g = (g_0, g_1) = (h_0 b_0, h_1 b_1)\in \Gamma$.

Now, suppose that we are given an element $g = (g_0, g_1)$ as before. Then, by the definition of the function $\Lift_{B,A}$, we see that $gA \in \Lift_{B,A}(g_0 B, g_1 B)$.
\end{proof}

\subsection{The conjugacy problem}
Suppose that $N\leq \mathrm{Stab}_\Gamma(1)$ is a finite-index normal subgroup of $\Gamma$.
We then define the function $Q^N\colon \Gamma\times\Gamma \to \mathcal{P}(\Gamma/N)$ as
\[
    Q^N(g,h)
    =
    \{
        x N\in \Gamma/N
    \mid
        x\in \Gamma
        \text{ with }
        x^{-1}gx=h
    \}
\]
for each $g,h\in \Gamma$.
Notice then that the elements $g,h\in \Gamma$ are conjugate in $\Gamma$ if and only if $Q^N(g,h)$ is non-empty.
Suppose that $n\geqslant 1$ is such that $\mathrm{Stab}_\Gamma(n)\leq N$. We then define the function $Q^N_n\colon \Gamma\times\Gamma \to \mathcal{P}(\Gamma/N)$ as
\[
    Q^N_n(g,h)
    =
    \{
        x N\in \Gamma/N
    \mid
        x\in \Gamma
        \text{ with }
        x^{-1}gx\cdot \mathrm{Stab}_\Gamma(n) = h \cdot \mathrm{Stab}_\Gamma(n)
    \}
\]
for each $g,h\in \Gamma$.
We then see that $g,h\in \Gamma$ are conjugate in $\Gamma/\mathrm{Stab}_\Gamma(n)$ if and only if $Q^N_n(g,h)$ is non-empty.
We then see for each $m\geq n$ that
\begin{equation}\label{eq:Qcontainment}
    Q^N(g,h) \subseteq Q^N_{m}(g,h) \subseteq Q^N_n(g,h)
\end{equation}
for each $g,h\in \Gamma$.
In Corollaries~\ref{cor:rec-compute-Q} and~\ref{cor:rec-compute-QFin}, we define a recursive algorithm to compute these functions.
These corollaries follow from Lemma~\ref{lem:lifting-conjugacy} as follows.

\begin{lemma}\label{lem:lifting-conjugacy}
Let $N\leq \mathrm{Stab}_\Gamma(1)$ and let $M\leq \Gamma$ be a normal subgroup of $\Gamma$ (possibly with infinite index) such that $\Psi_\Gamma(N) = M\times M$.
In the following, let $x\in \Gamma$:

\begin{enumerate}[leftmargin=*]
    \item\label{lem:lifting-conjugacy/1} Let $g,h\in \Gamma$ with \[(x^{-1}gx)\, N = h\, N,\] then $g\in \mathrm{Stab}_\Gamma(1)$ if and only if $h\in \mathrm{Stab}_\Gamma(1)$.
    \item\label{lem:lifting-conjugacy/2}  Let $g,h\in \mathrm{Stab}_\Gamma(1)$ with $g=(g_0, g_1)$ and $h = (h_0,h_1)$, then the following hold:
    \begin{enumerate}
        \item if $x = (x_0, x_1) \in \mathrm{Stab}_\Gamma(1)$, then
        \[
        (x^{-1} g x)\, N = h \, N \iff
        \begin{cases}
            (x_0^{-1} g_0 x_0) \, M = h_0 \, M\\
            (x_1^{-1} g_1 x_1) \, M = h_1 \, M;
        \end{cases}
        \]
        \item if $x = (x_0, x_1) a \in \mathrm{Stab}_\Gamma(1) \cdot a$, then
        \[
        (x^{-1} g x) \, N = h \, N
        \iff
        \begin{cases}
            (x_0^{-1} g_0 x_0) \, M = h_1 \, M\\
            (x_1^{-1} g_1 x_1) \, M = h_0 \, M.
        \end{cases}
        \]
    \end{enumerate}
    \item\label{lem:lifting-conjugacy/3}  Let $g,h\in \mathrm{Stab}_\Gamma(1)\cdot a$ with $g=(g_0, g_1)a$ and $h = (h_0,h_1)a$. Then the following hold:
    \begin{enumerate}
        \item if $x = (x_0, x_1) \in \mathrm{Stab}_\Gamma(1)$, then
        \[
        (x^{-1}gx) \, N = h \, N
        \iff
        \begin{cases}
            (x_0^{-1}(g_0 g_1) x_0) \, M = (h_0 h_1) \, M\\
            x_1 \, M = (g_1 x_0 h_1^{-1}) \, M;
        \end{cases}
        \]
        \item if $x = (x_0, x_1) a \in \mathrm{Stab}_\Gamma(1)\cdot  a$, then
        \[
        (x^{-1}gx) \, N = h \, N
        \iff
        \begin{cases}
            (x_0^{-1} (g_0 g_1) x_0) \, M = (h_1 h_0) \, M\\
            x_1 \, M = (g_0^{-1} x_0 h_1) \, M.
        \end{cases}
        \]
    \end{enumerate}
\end{enumerate}
    Notice that in the above cases, we have reduced the problem of checking the conjugacy of $g$ and $h$ to a problem involving the conjugacy of their first-level restrictions.
\end{lemma}
\begin{proof}
We see that (\ref{lem:lifting-conjugacy/1}) in the lemma statement follows as $N \leq \mathrm{Stab}_\Gamma(1)$ is a normal subgroup.
We prove the remaining parts of the statement as follows.

\begin{enumerate}[leftmargin=*]
\item[(2)]
Suppose that $g,h\in \mathrm{Stab}_\Gamma(1)$ with $g = (g_0, g_1)$ and $h = (h_0, h_1)$. We then write $gN = (g_0M, g_1M)$ and $hN = (h_0M, h_1M)$.
\begin{enumerate}
    \item 
    Suppose that $x= (x_0,x_1) \in \mathrm{Stab}_\Gamma(1)$. We then see that $xN = (x_0 M, x_1 M)$ and $x^{-1}N = (x_0^{-1} M, x_1^{-1} M)$. Thus, 
        \[
        (x^{-1} g x)\, N = h \, N \iff
        \begin{cases}
            (x_0^{-1} g_0 x_0) \, M = h_0 \, M\\
            (x_1^{-1} g_1 x_1) \, M = h_1 \, M.
        \end{cases}
        \]
    \item 
    Suppose that $x = (x_0,x_1)\in \mathrm{Stab}_\Gamma(1)\cdot a$, then $xN = (x_0 M, x_1 M) \cdot a$, and thus $x^{-1}N = (x_1^{-1} M, x_0^{-1} M) \cdot a$.
    Therefore, we see that
        \[
        (x^{-1} g x) \, N = h \, N
        \iff
        \begin{cases}
            (x_0^{-1} g_0 x_0) \, M = h_1 \, M\\
            (x_1^{-1} g_1 x_1) \, M = h_0 \, M.
        \end{cases}
        \]
\end{enumerate}
\item[(3)] 
Suppose that $g,h \in \mathrm{Stab}_\Gamma(1)\cdot a$ with $g = (g_0, g_1) \cdot a$ and $h = (h_0,h_1) \cdot a$. From this we can write $gN = (g_0 M, g_1 M) \cdot a$ and $hN = (h_0M, h_1 M) \cdot a$.
\begin{enumerate}
    \item 
    If $x= (x_0,x_1) \in \mathrm{Stab}_\Gamma(1)$, we then see that $xN = (x_0 M, x_1 M)$ and $x^{-1}N = (x_0^{-1} M, x_1^{-1} M)$. Thus, we see that
        \[
        (x^{-1} g x)\, N = h \, N \iff
        \begin{cases}
            (x_0^{-1} g_0 x_1) \, M = h_0 \, M\\
            (x_1^{-1} g_1 x_0) \, M = h_1 \, M.
        \end{cases}
        \]
    Moreover, we see that this system of equations is equivalent to
    \[
        \begin{cases}
            (x_0^{-1} g_0 x_1)(x_1^{-1} g_1 x_0) \, M = h_0 h_1 \, M\\
            x_1^{-1} M = (h_1 x_0^{-1} g_1^{-1})\, M,
        \end{cases}
    \]
    which can then be simplified to the system of equations
    \[
        \begin{cases}
            (x_0^{-1}(g_0 g_1) x_0) \, M = (h_0 h_1) \, M\\
            x_1 \, M = (g_1 x_0 h_1^{-1}) \, M.
        \end{cases}
    \]
    \item 
    Suppose that $x = (x_0,x_1)\in \mathrm{Stab}_\Gamma(1)\cdot a$, then $xN = (x_0 M, x_1 M) \cdot a$, and thus $x^{-1}N = (x_1^{-1} M, x_0^{-1} M) \cdot a$.
    Therefore, we see that
    \[
        (x^{-1}gx)\, N = h\, N
        \iff
        \begin{cases}
            (x_0^{-1} g_0 x_1) \, M = h_1 \, M\\
            (x_1^{-1} g_1 x_0) \, M = h_0 \, M.
        \end{cases}
    \]
    Moreover, we see that this system of equations is equivalent to
    \[
        \begin{cases}
            (x_0^{-1} g_0 x_1)(x_1^{-1} g_1 x_0) \, M = h_1 h_0 \, M\\
            x_1^{-1}\, M = (h_0 x_0^{-1} g_1^{-1})\, M,
        \end{cases}
    \]
    which can then be simplified to the system of equations
    \[
        \begin{cases}
            (x_0^{-1} (g_0 g_1) x_0) \, M = (h_1 h_0) \, M\\
            x_1 \, M = (g_0^{-1} x_0 h_1) \, M.
        \end{cases}
    \]
\end{enumerate}
\end{enumerate}
Thus, we have our desired result.
\end{proof}

From the above lemma, we immediately obtain the following corollary which provides a recursive formula for the function $Q^A\colon \Gamma\times\Gamma\to \mathcal{P}(\Gamma/A)$.

\begin{corollary}\label{cor:rec-compute-Q}
    Let $A,B\leq \mathrm{Stab}_\Gamma(1)$ be finite-index normal subgroups of $\Gamma$ for which $B\times B \leq \Psi_\Gamma(A)$.
    Let $g,h \in \Gamma$. Then $Q^A(g,h)$ can be calculated as follows.
    \begin{enumerate}[leftmargin=*]
        \item\label{cor:rec-compute-Q/1} If $g,h\in \mathrm{Stab}_\Gamma(1)$ with $g = (g_0, g_1)$ and $h = (h_0, h_1)$, then
        \begin{multline*}
            Q^A(g,h)
            =
            \Lift_{ B,A}\left( Q^B(g_0, h_0) \times Q^B(g_1, h_1) \right)
            \\\cup 
            \Lift_{ B,A}\left( Q^B(g_1, h_0) \times Q^B(g_0, h_1) \right)\cdot a.
        \end{multline*}
        \item\label{cor:rec-compute-Q/2} If $g,h \notin \mathrm{Stab}_{\Gamma}(1)$ with $g = (g_0, g_1) \cdot a$ and $h = (h_0, h_1) \cdot a$, then
        \begin{multline*}
            Q^A(g,h)
            =
            \Lift_{ B,A}\left\{
                (x_0 B, x_1 B) \in \Gamma/B \times \Gamma/B
            \ \middle|\
            \begin{aligned}
                x_0 B \in Q^B(g_0 g_1, h_0 h_1)
                \\\text{ and }
                x_1 B = g_1 x_0 h_1^{-1} B
            \end{aligned}
            \right\}
            \\\cup
            \Lift_{ B,A}\left\{
                (x_0 B, x_1 B) \in \Gamma/B \times \Gamma/B
            \ \middle|\
            \begin{aligned}
                x_0 B \in Q^B(g_0 g_1, h_1 h_0)
                \\\text{ and }
                x_1 B = g_0^{-1} x_0 h_1 B
            \end{aligned}
            \right\}\cdot a.
        \end{multline*}
    \end{enumerate}
    In all remaining cases, we have $Q^A(g,h) = \emptyset$.
\end{corollary}
\begin{proof}
Let $g,h,x\in \Gamma$. We then separate our proof into three parts as follows.

Suppose that $g,h\in \mathrm{Stab}_\Gamma(1)$ with $g = (g_0, g_1)$ and $h = (h_0, h_1)$. 
Then, from Lemmas~\ref{lem:branch-subgroup-prop} and~\ref{lem:lifting-conjugacy} (with $N=M = \{1\}$), we see that
\begin{itemize}[leftmargin=*]
\item
if $x = (x_0, x_1) \in \mathrm{Stab}_\Gamma(1)$, then 
\[
x^{-1} g x = h \iff
\left\{
\begin{aligned}
    x_0^{-1} g_0 x_0 = h_0\\
    x_1^{-1} g_1 x_1 = h_1
\end{aligned}
\right.
\iff
\begin{cases}
    x_0 B \in Q^B(g_0, h_0)\\
    x_1 B \in Q^B(g_1, h_1);
\end{cases}
\]
\item
if $x = (x_0,x_1) \cdot a \in \mathrm{Stab}_\Gamma(1)\cdot a$, then
\[
x^{-1} g x = h \iff
\left\{
\begin{aligned}
    x_0^{-1} g_0 x_0 = h_1\\
    x_1^{-1} g_1 x_1 = h_0
\end{aligned}
\right.
\iff
\begin{cases}
    x_0 B \in Q^B(g_0, h_1)\\
    x_1 B \in Q^B(g_1, h_0).
\end{cases}
\]
\end{itemize}
Hence, from Lemma~\ref{lem:branch-subgroup-prop}, we see that
\begin{multline*}
    Q^A(g,h)
    =
    \Lift_{ B,A}\left( Q^B(g_0, h_0) \times Q^B(g_1, h_1) \right)
    \\\cup 
    \Lift_{ B,A}\left( Q^B(g_1, h_0) \times Q^B(g_0, h_1) \right)\cdot a.
\end{multline*}

Suppose that $g,h \in \mathrm{Stab}_\Gamma(1)\cdot a$ with $g = (g_0,g_1) \cdot a$ and $h = (h_0,h_1) \cdot a$.
Then, from Lemmas~\ref{lem:branch-subgroup-prop} and \ref{lem:lifting-conjugacy} (with $N = M = \{1\}$), we see that
\begin{itemize}[leftmargin=*]
\item 
if $x = (x_0, x_1)\in \mathrm{Stab}_\Gamma(1)$, then
\[
x^{-1}gx = h
\iff
\left\{
\begin{aligned}
    x_0^{-1}(g_0 g_1) x_0 = h_0 h_1\\
    x_1 = g_1 x_0 h_1^{-1}
\end{aligned}
\right.
\iff
\begin{cases}
    x_0 B \in Q^B(g_0g_1, h_0h_1) \\
    x_1 B = g_1 x_0 h_1^{-1} B;
\end{cases}
\]
\item 
if $x = (x_0, x_1)a \in \mathrm{Stab}_\Gamma(1)\cdot a$, then
\[
x^{-1}gx = h
\iff
\left\{
\begin{aligned}
    x_0^{-1} (g_0 g_1) x_0 = h_1 h_0\\
    x_1 = g_0^{-1} x_0 h_1
\end{aligned}
\right.
\iff
\begin{cases}
    x_0 B \in Q^B(g_0, g_1, h_1 h_0) \\
    x_1 B = g_0^{-1} x_0 h_1 B.
\end{cases}
\]
\end{itemize}
Hence, from Lemma~\ref{lem:branch-subgroup-prop}, we see that
\begin{multline*}
    Q^A(g,h)
    =
    \Lift_{ B,A}\left\{
        (x_0 B, x_1 B) \in \Gamma/B \times \Gamma/B
    \ \middle|\
    \begin{aligned}
        x_0 B \in Q^B(g_0 g_1, h_0 h_1)
        \\\text{ and }
        x_1 B = g_1 x_0 h_1^{-1} B
    \end{aligned}
    \right\}
    \\\cup
    \Lift_{ B,A}\left\{
        (x_0 B, x_1 B) \in \Gamma/B \times \Gamma/B
    \ \middle|\
    \begin{aligned}
        x_0 B \in Q^B(g_0 g_1, h_1 h_0)
        \\\text{ and }
        x_1 B = g_0^{-1} x_0 h_1 B
    \end{aligned}
    \right\}\cdot a.
\end{multline*}

Since $\mathrm{Stab}_\Gamma(1)$ is a normal subgroup of $\Gamma$ which contains $B$ as a subgroup, it follows that, in all remaining cases, we have $Q^B(g,h) = \emptyset$.
\end{proof}

Furthermore, Lemma~\ref{lem:lifting-conjugacy} can then be used to find recurrence formulas for the functions of the form $Q^A_n \colon \Gamma\times \Gamma \to \mathcal{P}(\Gamma/A)$ as follows.

\begin{corollary}\label{cor:rec-compute-QFin}
    Let $A,B\leq \mathrm{Stab}_\Gamma(1)$ be finite-index normal subgroups of $\Gamma$ for which $B\times B \leq \Psi_\Gamma(A)$, and let $n\in\mathbb{N}$ with $\mathrm{Stab}_\Gamma(n+1)\leq A$ and $\mathrm{Stab}_\Gamma(n)\leq B$.
    Let $g,h \in \Gamma$. Then the function $Q^A_{n+1}(g,h)$ can be calculated as follows.
    \begin{enumerate}[leftmargin=*]
        \item If $g,h\in \mathrm{Stab}_\Gamma(1)$ with $g = (g_0, g_1)$ and $h = (h_0, h_1)$, then
        \begin{multline*}
            Q^A_{n+1}(g,h)
            =
            \Lift_{ B,A}\left( Q^B_n(g_0, h_0) \times Q^B_n(g_1, h_1) \right)
            \\\cup 
            \Lift_{ B,A}\left( Q^B_n(g_1, h_0) \times Q^B_n(g_0, h_1) \right)\cdot a.
        \end{multline*}
        \item If $g,h \notin \mathrm{Stab}_{\Gamma}(1)$ with $g = (g_0, g_1) \cdot a$ and $h = (h_0, h_1) \cdot a$, then
        \begin{multline*}
            Q^A_{n+1}(g,h)
            =
            \Lift_{ B,A}\left\{
                (x_0 B, x_1 B) \in \Gamma/B \times \Gamma/B
            \ \middle|\
            \begin{aligned}
                x_0 B \in Q^B_n(g_0 g_1, h_0 h_1)
                \\\text{ and }
                x_1 B = g_1 x_0 h_1^{-1} B
            \end{aligned}
            \right\}
            \\\cup
            \Lift_{ B,A}\left\{
                (x_0 B, x_1 B) \in \Gamma/B \times \Gamma/B
            \ \middle|\
            \begin{aligned}
                x_0 B \in Q^B_n(g_0 g_1, h_1 h_0)
                \\\text{ and }
                x_1 B = g_0^{-1} x_0 h_1 B
            \end{aligned}
            \right\}\cdot a.
        \end{multline*}
    \end{enumerate}
    In all remaining cases, we have $Q^A_{n+1}(g,h) = \emptyset$.
\end{corollary}
\begin{proof}
In this proof, let $g,h,x\in \Gamma$.

Suppose that $g,h\in \mathrm{Stab}_\Gamma(1)$ with $g = (g_0, g_1)$ and $h = (h_0, h_1)$. 
Then, from Lemmas~\ref{lem:branch-subgroup-prop} and~\ref{lem:lifting-conjugacy} (with $N = \mathrm{Stab}_\Gamma(n+1)$ and $M = \mathrm{Stab}_\Gamma(n)$), we see that
\begin{itemize}[leftmargin=*]
\item
if $x = (x_0, x_1) \in \mathrm{Stab}_\Gamma(1)$, then 
\begin{align*}
(x^{-1} g x) \mathrm{Stab}_\Gamma(n+1) &= h \, \mathrm{Stab}_\Gamma(n+1)
\\&\qquad\iff
\begin{cases}
    (x_0^{-1} g_0 x_0)\, \mathrm{Stab}_\Gamma(n) = h_0 \, \mathrm{Stab}_\Gamma(n)\\
    (x_1^{-1} g_1 x_1)\, \mathrm{Stab}_\Gamma(n) = h_1 \, \mathrm{Stab}_\Gamma(n)
\end{cases}
\\&\qquad\iff
\begin{cases}
    x_0 B \in Q^B_n(g_0, h_0)\\
    x_1 B \in Q^B_n(g_1, h_1);
\end{cases}
\end{align*}
\item
if $x = (x_0,x_1) \cdot a \in \mathrm{Stab}_\Gamma(1)\cdot a$, then
\begin{align*}
(x^{-1} g x) \, \mathrm{Stab}_\Gamma(n+1) &= h \, \mathrm{Stab}_\Gamma(n+1)
\\&\qquad\iff
\begin{cases}
    (x_0^{-1} g_0 x_0) \, \mathrm{Stab}_\Gamma(n) = h_1 \, \mathrm{Stab}_\Gamma(n)\\
    (x_1^{-1} g_1 x_1) \, \mathrm{Stab}_\Gamma(n) = h_0 \, \mathrm{Stab}_\Gamma(n)
\end{cases}
\\&\qquad\iff
\begin{cases}
    x_0 B \in Q^B_n(g_0, h_1)\\
    x_1 B \in Q^B_n(g_1, h_0).
\end{cases}
\end{align*}
\end{itemize}
Hence, from Lemma~\ref{lem:branch-subgroup-prop}, we see that
\begin{multline*}
    Q^A_{n+1}(g,h)
    =
    \Lift_{ B,A}\left( Q^B_n(g_0, h_0) \times Q^B_n(g_1, h_1) \right)
    \\\cup 
    \Lift_{ B,A}\left( Q^B_n(g_1, h_0) \times Q^B_n(g_0, h_1) \right)\cdot a.
\end{multline*}

Suppose that $g,h \in \mathrm{Stab}_\Gamma(1)\cdot a$ with $g = (g_0,g_1) \cdot a$ and $h = (h_0,h_1) \cdot a$.
Then, from Lemmas~\ref{lem:branch-subgroup-prop} and \ref{lem:lifting-conjugacy} (with $N = \mathrm{Stab}_\Gamma(n+1)$ and $M = \mathrm{Stab}_\Gamma(n)$), we see that
\begin{itemize}[leftmargin=*]
\item 
if $x = (x_0, x_1)\in \mathrm{Stab}_\Gamma(1)$, then
\begin{align*}
(x^{-1}gx) \, \mathrm{Stab}_\Gamma(n+1) &= h \, \mathrm{Stab}_\Gamma(n+1)
\\&\qquad\iff
\begin{cases}
    (x_0^{-1}(g_0 g_1) x_0) \, \mathrm{Stab}_\Gamma(n) = (h_0 h_1) \, \mathrm{Stab}_\Gamma(n)\\
    x_1 \, \mathrm{Stab}_\Gamma(n) = (g_1 x_0 h_1^{-1}) \, \mathrm{Stab}_\Gamma(n)
\end{cases}
\\&\qquad\iff
\begin{cases}
    x_0 B \in Q^B_n(g_0g_1, h_0h_1) \\
    x_1 B = (g_1 x_0 h_1^{-1})\, B;
\end{cases}
\end{align*}
\item 
if $x = (x_0, x_1)a \in \mathrm{Stab}_\Gamma(1)\cdot a$, then
\begin{align*}
(x^{-1}gx) \, \mathrm{Stab}_\Gamma(n+1) &= h \, \mathrm{Stab}_\Gamma(n+1)
\\&\qquad\iff
\begin{cases}
    (x_0^{-1} (g_0 g_1) x_0)\, \mathrm{Stab}_\Gamma(n) = (h_1 h_0) \, \mathrm{Stab}_\Gamma(n)\\
    x_1\, \mathrm{Stab}_\Gamma(n) = (g_0^{-1} x_0 h_1)\, \mathrm{Stab}_\Gamma(n)
\end{cases}
\\&\qquad\iff
\begin{cases}
    x_0 B \in Q^B_n(g_0, g_1, h_1 h_0) \\
    x_1 B = g_0^{-1} x_0 h_1 B.
\end{cases}
\end{align*}
\end{itemize}
Hence, from Lemma~\ref{lem:branch-subgroup-prop}, we see that
\begin{multline*}
    Q^A_{n+1}(g,h)
    =
    \Lift_{ B,A}\left\{
        (x_0 B, x_1 B) \in \Gamma/B \times \Gamma/B
    \ \middle|\
    \begin{aligned}
        x_0 B \in Q^B_n(g_0 g_1, h_0 h_1)
        \\\text{ and }
        x_1 B = g_1 x_0 h_1^{-1} B
    \end{aligned}
    \right\}
    \\\cup
    \Lift_{ B,A}\left\{
        (x_0 B, x_1 B) \in \Gamma/B \times \Gamma/B
    \ \middle|\
    \begin{aligned}
        x_0 B \in Q^B_n(g_0 g_1, h_1 h_0)
        \\\text{ and }
        x_1 B = g_0^{-1} x_0 h_1 B
    \end{aligned}
    \right\}\cdot a.
\end{multline*}

Since $\mathrm{Stab}_\Gamma(1)$ is a normal subgroup of $\Gamma$ which contains $B$ as a subgroup, it follows that, in all remaining cases, we have $Q^B_{n+1}(g,h) = \emptyset$.
\end{proof}

\subsection{The subgroup \texorpdfstring{$K$}{K}}
Let $K\leq \Gamma$ be the normal closure of the element $(ab)^2$ in the group $\Gamma$.
Then it is known from the Proposition on p.~230 of \cite{harpe2000} that
\begin{itemize}
    \item $
    K
    =
    \left\langle
        (ab)^2,
        (bada)^2,
        (abad)^2
    \right\rangle$;
    \item $K$ is a finite-index normal subgroup of $\Gamma$ with index $[\Gamma:K]=16$;
    \item $\mathrm{Stab}_\Gamma(3)\leq K \leq \mathrm{Stab}_\Gamma(1)$; and
    \item $K\times K\leq \Psi(K)$. Notice then that $K \times K$ can be viewed as a subgroup of  the intersection $\Gamma \times \Gamma \cap \mathrm{Stab}_\Gamma(1)$ taken in $\Gamma$.
\end{itemize}
From these properties, we see that Corollaries~\ref{cor:rec-compute-Q} and~\ref{cor:rec-compute-QFin} apply with $A=B=K$.
Moreover, from Figure~2 in \cite{Lysenok2010}, we see that the cosets of $K$ in $\Gamma$, with respect to the generating set $\{a,b,d\}$, form the Schreier graph as in Figure~\ref{fig:schr-graph} where
\begin{align*}
    z_0 &= K, &
    z_1 &= Kd, &
    z_2 &= Kda,&
    z_3 &= Kdad,
    \\
    z_4 &= K(ad)^2,&
    z_5 &= Kada, &
    z_6 &= K ad, &
    z_7 &= K a,
    \\
    z_8 &= Kb,&
    z_9 &= K c,&
    z_{10} &= Kca,&
    z_{11} &= Kcad,
    \\
    z_{12} &= K badad,&
    z_{13} &= K bada, &
    z_{14} &= K bad, &
    z_{15} &= K ab.
\end{align*}
Moreover, from Lemma~2.10 in \cite{Lysenok2010}, we see that
\[
    z_i \in \Lift_{K,K}(z_j, z_k)
\]
if and only if $(j,k), i$ is an entry in Table~\ref{tab:lifting}.

\begin{figure}[ht!]
\centering
\begin{tikzpicture}

\node [draw,circle,inner sep=0pt, minimum size=2em] (g0) at (0,0) {$z_{0}$};
\node [draw,circle,inner sep=0pt, minimum size=2em] (g7) at (2,0) {$z_{7}$};
\node [draw,circle,inner sep=0pt, minimum size=2em] (g6) at (4,0) {$z_{6}$};
\node [draw,circle,inner sep=0pt, minimum size=2em] (g5) at (6,0) {$z_{5}$};

\node [draw,circle,inner sep=0pt, minimum size=2em] (g1) at (1,1.5) {$z_{1}$};
\node [draw,circle,inner sep=0pt, minimum size=2em] (g2) at (3,1.5) {$z_{2}$};
\node [draw,circle,inner sep=0pt, minimum size=2em] (g3) at (5,1.5) {$z_{3}$};
\node [draw,circle,inner sep=0pt, minimum size=2em] (g4) at (7,1.5) {$z_{4}$};

\node [draw,circle,inner sep=0pt, minimum size=2em] (g8) at  (0,2.5) {$z_{8}$};
\node [draw,circle,inner sep=0pt, minimum size=2em] (g15) at (2,2.5) {$z_{15}$};
\node [draw,circle,inner sep=0pt, minimum size=2em] (g14) at (4,2.5) {$z_{14}$};
\node [draw,circle,inner sep=0pt, minimum size=2em] (g13) at (6,2.5) {$z_{13}$};

\node [draw,circle,inner sep=0pt, minimum size=2em] (g9) at  (1,4) {$z_{9}$};
\node [draw,circle,inner sep=0pt, minimum size=2em] (g10) at (3,4) {$z_{10}$};
\node [draw,circle,inner sep=0pt, minimum size=2em] (g11) at (5,4) {$z_{11}$};
\node [draw,circle,inner sep=0pt, minimum size=2em] (g12) at (7,4) {$z_{12}$};

\draw (g0) -- node [below] {$a$} (g7);
\draw (g7) -- node [below] {$d$} (g6);
\draw (g6) -- node [below] {$a$} (g5);

\draw (g8) -- node  [above,pos=0.8] {$a$} (g15);
\draw (g15) -- node [above,pos=0.8] {$d$} (g14);
\draw (g14) -- node [above,pos=0.8] {$a$} (g13);

\draw (g9) -- node  [above] {$a$} (g10);
\draw (g10) -- node [above] {$d$} (g11);
\draw (g11) -- node [above] {$a$} (g12);

\draw (g8) -- node [above left] {$d$} (g9);
\draw (g12) -- node [above left] {$d$} (g13);

\draw (g0) -- node [left,pos=0.2] {$b$} (g8);
\draw (g7) -- node [left,pos=0.2] {$b$} (g15);
\draw (g6) -- node [left,pos=0.2] {$b$} (g14);
\draw (g5) -- node [left,pos=0.2] {$b$} (g13);

\draw [dashed] (g1) -- node [right,pos=0.8] {$b$} (g9);
\draw [dashed] (g2) -- node [right,pos=0.8] {$b$} (g10);
\draw [dashed] (g3) -- node [right,pos=0.8] {$b$} (g11);
\draw  (g4) -- node [right,pos=0.8] {$b$} (g12);

\draw [dashed] (g0) -- node [below right] {$d$} (g1);
\draw (g5) -- node [below right] {$d$} (g4);

\draw [dashed] (g1) -- node  [above,pos=0.2] {$a$} (g2);
\draw [dashed] (g2) -- node [above,pos=0.2] {$d$}  (g3);
\draw [dashed] (g3) -- node [above,pos=0.2] {$a$}  (g4);

\end{tikzpicture}
\caption{Schreier graph of $K$ in $G$ (where $z_0 \in K$) with respect to the generating set $\{a,b,d\}$ by right-multiplication \cite[Figure~2]{Lysenok2010}.}
\label{fig:schr-graph}
\end{figure}
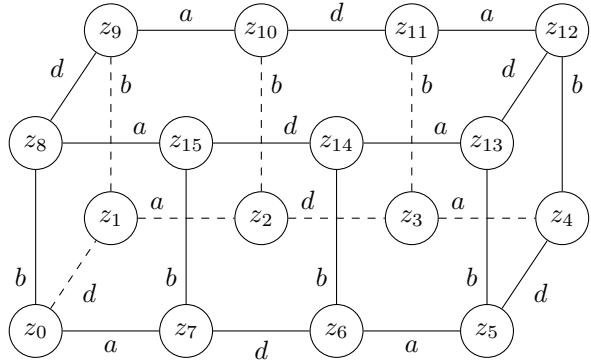

\begin{table}[ht!]
    \centering
    \begin{tabular}{| c | c | p{1em} | c | c | p{1em} | c | c | p{1em} | c | c | }
        \hhline{--~--~--~--}
        $(w_0,w_1)$ & $w$ &&
        $(w_0,w_1)$ & $w$ &&
        $(w_0,w_1)$ & $w$ &&
        $(w_0,w_1)$ & $w$
        \\
        \hhline{==~==~==~==}
        (0,0) & 0 &&
        (8,0) & 5 &&
        (4,4) & 0 &&
        (12,4) & 5
        \\
        \hhline{--~--~--~--}
        (0,8) & 1 &&
        (8,8) & 4 &&
        (4,12)& 1 &&
        (12,12)&4
        \\
        \hhline{--~--~--~--}
        (1,7) & 13 &&
        (9,7) & 8 &&
        (5,3) & 13 &&
        (13,3) & 8
        \\
        \hhline{--~--~--~--}
        (1,15) & 12 &&
        (9,15) & 9 &&
        (5,11) & 12 &&
        (13,11) & 9
        \\
        \hhline{--~--~--~--}
        (2,6) & 4 &&
        (10,6) & 1 &&
        (6,2) & 4 &&
        (14,2)&1
        \\
        \hhline{--~--~--~--}
        (2,14) & 5 &&
        (10,14) & 0 &&
        (6,10) & 5 &&
        (14,10) & 0
        \\
        \hhline{--~--~--~--}
        (3,5) & 9 &&
        (11,5) & 12 &&
        (7,1) & 9 &&
        (15,1) & 12
        \\
        \hhline{--~--~--~--}
        (3,13) & 8 &&
        (11,13) & 13 &&
        (7,9) & 8 &&
        (15,9) & 13
        \\
        \hhline{--~--~--~--}
    \end{tabular}
    
    \caption{Lifting map \cite[Table~1]{Lysenok2010}}
    \label{tab:lifting}
\end{table}

In our proofs within this section, we make use of the commutator subgroup of $\Gamma$ which is given by
\[
    [\Gamma,\Gamma]
    =\left\langle\{ [x,y] \mid x,y\in \{a,b,c,d\}\right\rangle.
\]
In particular, we see that
\[
    \Gamma/[\Gamma,\Gamma]
    =
    \left\langle
        a,b,c,d\ \middle|\ 
        \begin{aligned}
            a^2=b^2=d^2=bcd=1,\text{ and}\\
            [x,y]=1\text{ for all }x,y\in \{a,b,c,d\}
        \end{aligned}
    \right\rangle.
\]
From the Schreier graph in Figure~\ref{fig:schr-graph}, we see that $[\Gamma,\Gamma] = K \cup K (ad)^2$.
In particular, this means that $[\Gamma,\Gamma]$ is a normal subgroup of $\Gamma$ of index 8, and that
\[
    \Gamma = [\Gamma,\Gamma]\{ 1,a,b,c,d,ab,ad,ac \}.
\]
We then have the following results.

\begin{lemma}\label{lem:comm-subgroup}
    Let $g\in \Gamma$ and suppose that $g = (g_0,g_1)\cdot a$. Then
    \begin{itemize}
        \item if $g \in [\Gamma,\Gamma] a$, then $g_0 g_1,\, g_1 g_0 \in [\Gamma,\Gamma]$,
        \item if $g \in [\Gamma,\Gamma] ab$, then $g_0 g_1,\, g_1 g_0 \in [\Gamma,\Gamma]ac$,
        \item if $g \in [\Gamma,\Gamma] ac$, then $g_0 g_1,\, g_1 g_0 \in [\Gamma,\Gamma]ad$, and
        \item if $g \in [\Gamma,\Gamma] ad$, then $g_0 g_1,\, g_1 g_0 \in [\Gamma,\Gamma]b$.
    \end{itemize}
\end{lemma}
\begin{proof}
These results follow from the fact that $\Gamma/[\Gamma,\Gamma]$ is abelian, and from the recursive definition of the generators $a$, $b$, $c$, and $d$.
Alternatively, one can derive this result from Figure~\ref{fig:schr-graph}, Table~\ref{tab:lifting}, and the property that $[\Gamma,\Gamma] = K\cup K(ad)^2$.
\end{proof}

We then have the following result for $Q^K\colon \Gamma\times\Gamma\to\mathcal{P}(\Gamma/K)$.

\begin{lemma}[p.~157 of \cite{grigorchuk2005solved}]\label{lem:base-cong}
    We have
   \begin{align*}
       Q^K(a,a) &= K\{1,a,dad,(ad)^2\}, &
       Q^K(b,b) &= K\{1,b,c,d\},\\
       Q^K(c,c) &= K\{1,b,c,d\},\ \ \text{and}&
       Q^K(d,d) &= K\{1,b,c,d,ada,(ad)^2,bada, badad\}.
   \end{align*}
   Moreover, $Q(1,1)=\Gamma/K$, and $Q(x,y)=\emptyset$ for each $x,y\in \{a,b,c,d,1\}$ with $x\neq y$.
\end{lemma}

The subgroup $K$ is of particular interest due to the following lemma.

\begin{lemma}\label{lem:Q-agreement}
    If $g,h\in \Gamma$ with $g \cdot \mathrm{Stab}_\Gamma(1)\neq h\cdot  \mathrm{Stab}_\Gamma(1)$, then $Q_3(g,h) = \emptyset$.
    Moreover, if $g,h\in \Gamma$ with $\ell(g),\ell(h)\leq 1$, then $Q^K(g,h) = Q^K_6(g,h)$.
\end{lemma}

\begin{proof}
The first half of the lemma statement follows from Corollaries~\ref{cor:rec-compute-Q} and~\ref{cor:rec-compute-QFin}.

Throughout this proof, we implicitly use the property that
\[
    Q^K(g,h)\subseteq Q^K_n(g,h) \subseteq Q^K_m(g,h)
\]
for each $n\geq m \geq 3$, which follows from the definition of these functions.
Moreover, we also use the property that $a$, $b$, $c$, and $d$ are the elements in $\Gamma$ of length 1.

From  the definition of the elements $a,b,c,d\in \Gamma$, we see that
\[
    a,b,c,d\notin \mathrm{Stab}_\Gamma(3)\leq K
    \qquad
    \text{and}
    \qquad
    1\in \mathrm{Stab}(3)\leq K,
\]
and thus
\[
    Q^K(g,1) = Q^K(1,h) = Q^K_3(g,1) = Q^K_3(1,h) = \emptyset
\]
for each $g,h\in \{a,b,c,d\}$.
Moreover, by definition, it is clear that
\[
    Q^K(1,1) = Q^K_3(1,1) = \Gamma/K.
\]
Thus, it is now sufficient for us to consider only the case where $x,y\in \{a,b,c,d\}$.

Notice that $b,c,d\in \mathrm{Stab}_\Gamma(1)$ and $a\notin \mathrm{Stab}_\Gamma(1)$.
In particular, this implies that
\[
    Q^K(a,h) = Q^K(g,a) = Q^K_3(a,h) = Q^K_3(g,a)=\emptyset
\]
for each $g,h\in \{b,c,d\}$.

Considering Corollary~\ref{cor:rec-compute-QFin}, we see that
\begin{align*}
    Q^K_4(b,d) &=
    \Lift_{K,K}(Q^K_3(a,1) \times  Q^K_3(c,b))
    \cup\Lift_{K,K}(Q^K_3(c,1) \times  Q^K_3(a,b))\cdot a
    \\
    &=
    \Lift_{K,K}(\emptyset \times  Q^K_3(c,d))
    \cup\Lift_{K,K}(\emptyset \times  \emptyset)\cdot a = \emptyset,
    \\
    Q^K_4(c,d) &=
    \Lift_{K,K}(Q^K_3(a,1)\times Q^K_3(d,b))
    \cup\Lift_{K,K}(Q^K_3(d,1) \times Q^K_3(a,b))\cdot a
    \\
    &=
    \Lift_{K,K}(\emptyset\times Q^K_3(d,b))
    \cup\Lift_{K,K}(\emptyset \times \emptyset)\cdot a = \emptyset.
\end{align*}
Again, from Corollary~\ref{cor:rec-compute-QFin}, we see that
\begin{align*}
    Q^K_5(b,c)
    &=
    \Lift_{K,K}(Q^K_4(a,a) \times Q^K_4(c,d))
    \cup\Lift_{K,K}(Q^K_4(c,a)\times Q^K_4(a,d))
    \\
    &=
    \Lift_{K,K}(Q^K_4(a,a) \times \emptyset)
    \cup\Lift_{K,K}(\emptyset\times\emptyset)\cdot a
    =\emptyset.
\end{align*}
From the definition of the function $Q^K_n\colon \Gamma\times\Gamma\to\mathcal{P}(\Gamma/K)$, we then see that 
\[
    Q^K(g,h) = Q^K_5(g,h) = \emptyset
\]
for each $g,h \in \{b,c,d\}$ with $g\neq h$.

Thus, all that remains is to show that $Q^K(g,g) = Q^K_6(g,g)$ for $g\in \{a,b,c,d\}$. From Corollary~\ref{cor:rec-compute-QFin}, we see that
\begin{align*}
    Q_4(a,a)
    &=
    \Lift_{K,K}\left\{
        (x_0 K, x_1 K)
    \ \middle|\
    \begin{aligned}
        x_0 K \in Q^K_3(1,1)
        \\\text{ and }
        x_1 K = x_0 K
    \end{aligned}
    \right\}
    \\&\qquad\qquad\cup
    \Lift_{K,K}\left\{
        (x_0 K, x_1 K)
    \ \middle|\
    \begin{aligned}
        x_0 K \in Q^K_3(1,1)
        \\\text{ and }
        x_1 K = x_0 K
    \end{aligned}
    \right\}\cdot a
    \\
    &\subseteq
    \Lift_{K,K}(\{(x,x) \mid x\in G/K\})
    \cup\Lift_{K,K}(\{(x,x) \mid x\in G/K\})\cdot a.
\end{align*}
Thus, from Table~\ref{tab:lifting} and Figure~\ref{fig:schr-graph}, we see that
\[
    Q^K_4(a,a)\subseteq \{z_0, z_4\}\cup \{z_0, z_4\}\cdot a
    =\{z_0, z_4\}\cup\{ z_7, z_3\}
    = K\{ 1,(ad)^2, a, dad \}.
\]
Thus, from Lemma~\ref{lem:base-cong}, we see that $Q^K(a,a) = Q^K_4(a,a)$.

From Corollary~\ref{cor:rec-compute-QFin}, we see that
\begin{align*}
    Q_5^K(b,b) &=
    \Lift_{K,K}(Q^K_4(a,a)\times Q^K_4(c,c)) \cup 
    \Lift_{K,K}(Q^K_4(c,a)\times Q^K_4(a,c))\cdot a
    \\
    &=
    \Lift_{K,K}(Q^K_4(a,a)\times Q^K_4(c,c)) 
    \\
    &\subseteq
    \Lift_{K,K}(Q^K(a,a)\times G/K)
    =
    \Lift_{K,K}(\{z_0, z_3, z_4, z_7\}\times G/K)
\end{align*}
and
\begin{align*}
    Q_5^K(c,c) &=
    \Lift_{K,K}(Q^K_4(a,a)\times Q^K_4(d,d)) \cup 
    \Lift_{K,K}(Q^K_4(d,a)\times Q^K_4(a,d))\cdot a
    \\
    &=
    \Lift_{K,K}(Q^K_4(a,a)\times Q^K_4(d,d)) 
    \\
    &\subseteq
    \Lift_{K,K}(Q^K(a,a)\times G/K)
    =
    \Lift_{K,K}(\{z_0, z_3, z_4, z_7\}\times G/K).
\end{align*}
From Table~\ref{tab:lifting}, we then see that
\[
    Q^K_5(b,b), Q^K_5(c,c)
    \subseteq
    \{ z_0,z_1,z_8,z_9 \}
    =
    K\{ 1,d,b,c \}.
\]
From Lemma~\ref{lem:base-cong}, we conclude that $Q^K(b,b) = Q^K(c,c)= Q^K_5(b,b) = Q^K_5(c,c)$.

From Corollary~\ref{cor:rec-compute-QFin}, we see that
\begin{align*}
    Q_6^K(d,d) &=
    \Lift_{K,K}(Q^K_5(1,1)\times Q^K_5(b,b)) \cup 
    \Lift_{K,K}(Q^K_5(b,1)\times Q^K_5(1,b))\cdot a
    \\
    &=
    \Lift_{K,K}(Q^K_5(1,1)\times Q^K_5(b,b))
    \\
    &=
    \Lift_{K,K}(Q^K(1,1)\times Q^K(b,b))
    =
    \Lift_{K,K}(G/K\times \{z_0,z_1,z_8,z_9\}).
\end{align*}
From Table~\ref{tab:lifting}, we then see that
\[
    Q^K_6(d,d)
    =
    \{
        z_0, z_1, z_4,z_5,z_8,z_9,z_{12}, z_{13}
    \}
    =
    K
    \{
    1,d,(ad)^2,ada,b,c,badad,bada
    \}.
\]
Then, from Lemma~\ref{lem:base-cong}, we see that $Q^K(d,d) = Q^K_6(d,d)$.

Thus, we have proven all cases of our lemma.
\end{proof}

We may then generalise the above lemmas as follows.

\begin{lemma}\label{lem:Q-agreement2}
    If $x,y\in \Gamma$ with $\ell(x),\ell(y)\leq 2$, then $Q^K(x,y) = Q^K_{10}(x,y)$.
\end{lemma}
\begin{proof}
From Lemma~\ref{lem:Q-agreement}, we see that if $\ell(x),\ell(y)\leq 1$, then $Q^K(x,y) = Q^K_6(x,y)$.
Thus, in the remainder of this proof, we assume that either $\ell(x) = 2$ or $\ell(y)=2$.

We first note that the set of all elements of length 2 can be given as
\[
    S_2=
    \{
        ab,ba,ac,ca,ad,da
    \}.
\]
Notice then that $g\notin \mathrm{Stab}_\Gamma(1)$ for each $g\in S_2$.

If $g\in S_2$, it then follows from Corollary~\ref{cor:rec-compute-QFin} that
\[
    Q^K(g,y) = Q^{K}_3(g,y)= Q^K(x,g) = Q^{K}_3(x,g) = \emptyset
\]
for each $x,y \in \{1,b,c,d\} \subset \mathrm{Stab}_\Gamma(1)$.

If $g = (g_0,g_1) \cdot a\in S_2$ and $h = (h_0,h_1) \cdot a\in S_2$, then, from Lemma~\ref{lem:length-contraction}, we have $\ell(g_0 g_1),\ell(g_1 g_0),\ell(h_0 h_1),\ell(h_1 h_0)\leq 2$.
From Corollary~\ref{cor:rec-compute-QFin}, we see that
        \begin{multline*}
            Q^K_{7}(g,a)
            =
            \Lift_{ K,K}\left\{
                (x_0 K, x_1 K)
            \ \middle|\
            \begin{aligned}
                x_0 K \in Q^K_6(g_0 g_1, 1)
                \\\text{ and }
                x_1 K = g_1 x_0 K
            \end{aligned}
            \right\}
            \\\cup
            \Lift_{K,K}\left\{
                (x_0 K, x_1 K)
            \ \middle|\
            \begin{aligned}
                x_0 K \in Q^K_6(g_0 g_1, 1)
                \\\text{ and }
                x_1 K = g_0^{-1} x_0 K
            \end{aligned}
            \right\}\cdot a
        \end{multline*}
and
        \begin{multline*}
            Q^K_7(a,h)
            =
            \Lift_{ K,K}\left\{
                (x_0 K, x_1 K)
            \ \middle|\
            \begin{aligned}
                x_0 K \in Q^K_6(1, h_0 h_1)
                \\\text{ and }
                x_1 K = x_0 h_1^{-1} K
            \end{aligned}
            \right\}
            \\\cup
            \Lift_{ K,K}\left\{
                (x_0 K, x_1 K)
            \ \middle|\
            \begin{aligned}
                x_0 K \in Q^K_6(1, h_1 h_0)
                \\\text{ and }
                x_1 K = x_0 h_1 K
            \end{aligned}
            \right\}\cdot a.
        \end{multline*}
Hence, each element of the form $g_i g_j$ and $h_i h_j$ is either of length 2 and thus does not belong to $\mathrm{Stab}_\Gamma(1)$, or of length at most 1. Therefore, from earlier in this proof and from Lemma~\ref{lem:Q-agreement}, we see that
\[ Q^K_6(g_i g_j,1)=Q^K(g_i g_j,1)\quad\text{and}\quad Q^K_6(1,h_i h_j) = Q^K(1,h_i h_j)\] for each $(i,j)\in \{(0,1),(1,0)\}$.
Therefore, from Corollary~\ref{cor:rec-compute-QFin}, we see that \[Q^K_7(g,a) = Q^K(g,a)\quad\text{and}\quad Q^K_7(a,h) = Q^K(a,h)\] for each $g,h \in S_2$.
Thus, all that remains is to prove our result for $x,y \in S_2$.

Suppose that $g  =(g_0,g_1) \cdot a\in S_2$ and $h=(h_0,h_1) \cdot a\in S_2$. Lemma~\ref{lem:length-contraction} then implies $\ell(g_0 g_1),\ell(g_1 g_0),\ell(h_0 h_1),\ell(h_1 h_0)\leq 2$.
From Corollaries~\ref{cor:rec-compute-Q} and~\ref{cor:rec-compute-QFin}, if
\begin{align*}
    Q^K_n(g_0 g_1, h_0 h_1)&=Q^K(g_0 g_1, h_0 h_1)
    \quad\text{and}\\
    Q^K_n(g_0 g_1, h_1 h_0)&=Q^K(g_0 g_1, h_1 h_0)
\end{align*}
for some $n\geq 4$, then $Q^K_{n+1}(g,h) = Q^K(g,h)$.
Notice that since $g\in S_2$, we see that $g \in [\Gamma, \Gamma]\{ab,ac,ad\}$.
We consider these cases one after the other as follows.
\begin{itemize}
    \item
    If $g\in [\Gamma,\Gamma]ad$, then, from Lemma~\ref{lem:comm-subgroup}, we have $g_0 g_1, g_1 g_0 \notin S_2$. Thus,
    \begin{align*}
        Q^K_7(g_0 g_1, h_0 h_1) &= Q^K(g_0 g_1, h_0 h_1)\quad\text{and}\\
        Q^K_7(g_0 g_1, h_1 h_0) &= Q^K(g_0 g_1, h_1 h_0),
    \end{align*}
    and therefore, $Q^K_8(g,h) = Q^K(g,h)$.
    \item
    If $g\in [\Gamma,\Gamma]ac$, then, from Lemma~\ref{lem:comm-subgroup}, we have $g_0 g_1, g_1 g_0 \in [\Gamma,\Gamma]ad$. Thus
    \begin{align*}
        Q^K_8(g_0 g_1, h_0 h_1) &= Q^K(g_0 g_1, h_0 h_1)\quad\text{and}\\
        Q^K_8(g_0 g_1, h_1 h_0) &= Q^K(g_0 g_1, h_1 h_0),
    \end{align*}
    and hence, $Q^K_9(g,h) = Q^K(g,h)$.
    \item
    If $g\in [\Gamma,\Gamma]ab$, then, from Lemma~\ref{lem:comm-subgroup}, we have $g_0 g_1, g_1 g_0 \in [\Gamma,\Gamma]ac$. Therefore,
    \begin{align*}
        Q^K_9(g_0 g_1, h_0 h_1) &= Q^K(g_0 g_1, h_0 h_1)\quad\text{and}\\
        Q^K_9(g_0 g_1, h_1 h_0) &= Q^K(g_0 g_1, h_1 h_0),
    \end{align*}
    and thus $Q^K_{10}(g,h) = Q^K(g,h)$.
\end{itemize}
We now conclude that $Q^K_{10}(g,h) = Q^K(g,h)$ for each $\ell(g),\ell(h)\leq 2$.
\end{proof}

\subsection{Splitting trees}
Let $g,h\in \Gamma$ and $m\geq 3$. We now construct a finite rooted tree $\mathcal T_{g,h,m}$, which we call a \emph{splitting tree}, as follows:
\begin{enumerate}
    \item The root of the tree is labelled $(m;x,y)$.
    \item For each vertex labelled as $(n+1; x,y)$ with $n\geq 3$,
    \begin{itemize}
        \item if $x\cdot\mathrm{Stab}_\Gamma(1)\neq y\cdot\mathrm{Stab}_\Gamma(1)$, then the vertex does not have any children;
        \item if $x,y\in \mathrm{Stab}_\Gamma(1)$ with $x=(x_0,x_1)$ and $y=(y_0,y_1)$, then the vertex has 4 children as in Figure~\ref{fig:splitting_case_1}; and
        \item if $x,y\notin \mathrm{Stab}_\Gamma(1)$ with $x=(x_0,x_1)\cdot a$ and $y=(y_0,y_1)\cdot a$, then the vertex has 2 children as in Figure~\ref{fig:splitting_case_2}.
    \end{itemize}
\end{enumerate}
Notice that the tree $\mathcal{T}_{g,h,m}$ is locally finite and has depth at most $m-3$.
Thus, the tree $\mathcal{T}_{g,h,m}$ has finitely many vertices.
Notice also that the splitting tree $\mathcal{T}_{g,h,m}$ corresponds to the computation one would perform in order to calculate $Q^K_m(g,h)$ from the recursive formulas given in Corollary~\ref{cor:rec-compute-QFin}.

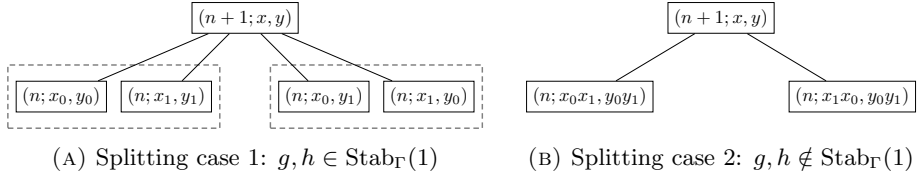
\begin{figure}[ht!]
    \begin{subfigure}{.5\linewidth}
        \centering
        \resizebox{\linewidth}{!}{\begin{tikzpicture}[>=stealth]
\node[draw] (n) at (0,0) {$(n+1; x,y)$};

\node[draw] (nl0) at (-3.5,-1.5) {$(n;x_0, y_0)$};
\node[draw] (nl1) at (-1.5,-1.5) {$(n;x_1, y_1)$};

\node[draw] (nr0) at (3.5,-1.5) {$(n;x_1, y_0)$};
\node[draw] (nr1) at (1.5,-1.5) {$(n;x_0, y_1)$};

\node[draw=gray, thick, 
        densely dashed,
        inner xsep=1ex, inner ysep=2ex,
        fit=(nl0) (nl1)] {};
\node[draw=gray, thick, 
        densely dashed,
        inner xsep=1ex, inner ysep=2ex,
        fit=(nr0) (nr1)] {};

\draw[] (nl0) to (n);
\draw[] (nl1) to (n);
\draw[] (nr0) to (n);
\draw[] (nr1) to (n);
\end{tikzpicture}}
        \caption{Splitting case 1: $g,h\in \mathrm{Stab}_\Gamma(1)$}
        \label{fig:splitting_case_1}
    \end{subfigure}%
    \begin{subfigure}{.5\linewidth}
        \centering
        \resizebox{\linewidth}{!}{\begin{tikzpicture}[>=stealth]
\node[draw] (n) at (0,0) {$(n+1;x, y)$};
\node[draw] (nl0) at (-2.5,-1.5) {$(n;x_0 x_1, y_0 y_1)$};
\node[draw] (nr0) at (2.5,-1.5) {$(n;x_1 x_0, y_0 y_1)$};
\draw[]  (nl0) to (n);
\draw[]  (nr0) to (n);

\phantom{
\node[draw] (n) at (0,0) {$(n+1;x,y)$};
\node[draw] (nl0) at (-3.5,-1.5) {$(n;x_0, y_0)$};
\node[draw] (nl1) at (-1.5,-1.5) {$(n;x_1, y_1)$};
\node[draw] (nr0) at (3.5,-1.5) {$(n;x_1, y_0)$};
\node[draw] (nr1) at (1.5,-1.5) {$(n;x_0, y_1)$};
\node[draw=gray, thick, 
        densely dashed,
        inner xsep=1ex, inner ysep=2ex,
        fit=(nl0) (nl1)] {};
\node[draw=gray, thick, 
        densely dashed,
        inner xsep=1ex, inner ysep=2ex,
        fit=(nr0) (nr1)] {};
\draw[] (nl0) to (n);
\draw[] (nl1) to (n);
\draw[] (nr0) to (n);
\draw[] (nr1) to (n);
}

\end{tikzpicture}}
        \caption{Splitting case 2: $g,h\notin\mathrm{Stab}_\Gamma(1)$}
        \label{fig:splitting_case_2}
    \end{subfigure}
\caption{Splitting cases}\label{fig:splitting}
\end{figure}

We now prove some properties of splitting trees as follows.

\begin{lemma}\label{lem:splitting-tree-lem}
    Let $\mathcal{T}_{g,h,m}$ be a splitting tree.
    If $(n;g,h)$ labels a leaf of the tree with $n\geq 4$, then $Q^K_n(g,h) = Q^K(g,h)=\emptyset$.
    Moreover, suppose that there is a subset $V$ of the vertices of $\mathcal{T}_{g,h,m}$ such that
    \begin{itemize}
        \item 
        each vertex in $V$ has a label of the form $(n;x,y)$ with $Q^K_n(x,y) = Q^K(x,y)$,
        \item
        and every root-to-leaf path in $\mathcal T_{g,h,m}$ contains some vertex in $V$,
    \end{itemize}
    then we have $Q^K_m(g,h) = Q^K(g,h)$.
\end{lemma}
\begin{proof}
From the definition of the splitting tree, if a leaf of the tree has a label of the form $(n;x,y)$ with $n\geq 4$, then $x\cdot\mathrm{Stab}_\Gamma(1) \neq y\cdot\mathrm{Stab}_\Gamma(1)$.
Thus, from Corollary~\ref{cor:rec-compute-QFin}, we see that $Q^K_n(x,y) = Q^K(x,y) = \emptyset$.

Suppose that $V$ is a set of vertices as described in the lemma statement.
From Corollaries~\ref{cor:rec-compute-Q} and~\ref{cor:rec-compute-QFin} and the definition of splitting trees, we may lift the equalities $Q^K_n(x,y) = Q^K(x,y)$, which correspond to the labels $(n;x,y)$ of vertices in $V$, to obtain $Q^K_m(g,h) = Q^K(g,h)$.
\end{proof}

Using splitting trees, we strengthen Lemma~\ref{lem:Q-agreement2} as follows.

\begin{lemma}\label{lem:additional-prop-K}
    Let $r\in \mathbb{N}$ and $m = 4\lceil \log_2(2r)\rceil+10$. Then \[Q^K(g,h) = Q^K_m(g,h)\] for each $g,h\in \Gamma$ with $\ell(g), \ell(h)\leq r$.
\end{lemma}
\begin{proof}
Consider the splitting tree $\mathcal{T}_{g,h,m}$.

In the following, we construct a set of vertices $V$ as in Lemma~\ref{lem:splitting-tree-lem}.

For each root-to-leaf path $p = v_0 \to v_1 \to \cdots \to v_k$ in $\mathcal T_{g,h,m}$, we proceed as follows:
\begin{itemize}
    \item 
    If $k \leq m-10$, then $v_k$ is labelled as $(n;x,y)$ for $n\geq 10$. 
    From Lemma~\ref{lem:splitting-tree-lem}, we see that $Q^K_n(x,y) = Q^K(x,y)$, thus we add $v_k$ to $V$.
    \item 
    Otherwise, from Lemma~\ref{lem:comm-subgroup} and the definition of a splitting tree, we see that every 4 consecutive vertices $v_i\to v_{i+1}\to v_{i+2} \to v_{i+3}$ in $p$ must contain at least one vertex labelled as $(n;x,y)$ where $x,y\in \mathrm{Stab}_\Gamma(1)$.
    From Lemma~\ref{lem:length-contraction} and the definition of splitting trees, we see that for each $i$ if $v_i$ is labelled by $(m-i;x,y)$ and the vertex $v_{i+3}$ is labelled by $(m-i+3;x',y')$, then $\ell(g')\leq \ell(g)/2 + 1$ and $\ell(h')\leq \ell(h)/2 + 1$.

    Notice that in this case, the vertex $v_{m-10}$ must be labelled by $Q^K_{10}(g,h)$ where $\ell(g),\ell(h) \leq 2$.
    From Lemma~\ref{lem:Q-agreement2}, we see that $Q^K_{10}(g,h) = Q^K(g,h)$.
    Thus, we add $v_{m-10}$ to the set $V$.
\end{itemize}
From these two cases, we see that we may construct a subset of vertices $V$ as in Lemma~\ref{lem:splitting-tree-lem}, and thus $Q^K_m(x,y) = Q^K(x,y)$ as desired.
\end{proof}

\subsection{The subgroups \texorpdfstring{$K_m$}{Km}}
From the subgroup $K$, we can define a sequence of subgroups $K_m$ as
\[
    K_0 = K
    \quad
    \text{and}
    \quad
    K_{m+1} = \Psi_\Gamma^{-1}(K_m\times K_m)
    \quad\text{for each}\quad
    m\geqslant 0.
\]
We then see that each of these subgroups $K_m$ is normal in $\Gamma$ by induction on $m$ as follows.
By definition, we see that $K_0=K$ is a normal subgroup.
Suppose that $K_m$ is a normal subgroup, and let $x = (x_0,x_1)\in \Gamma$ and $k = (k_0,k_1) \in K_{m+1}$. We then have
\[
    x^{-1}k x = (
    x_0^{-1} k_0 x_0,
    x_1^{-1} k_1 x_1
    ) \in K_{m+1}
\]
since $k_i \in K_{m}$ implies that $x_i^{-1}k_ix_i \in K_m$.
Thus, $K_{m+1}$ is normal if $K_m$ is normal.

Recall that $\mathrm{Stab}_\Gamma(3) \leq K\leq \mathrm{Stab}_\Gamma(1)$.
Therefore, it follows that \[\mathrm{Stab}_\Gamma(m+3)\leq K_m\leq \mathrm{Stab}_\Gamma(m+1)\] for each $m\geqslant 0$.
Notice then that it follows that each $K_m$ is a finite-index subgroup of $\Gamma$.
We generalise Lemma~\ref{lem:additional-prop-K} to obtain the following lemma.

\begin{lemma}\label{lem:additional-prop-K-general}
    Let $r, m \in \mathbb N$. Then there exists some $n$ such that \[Q^{K_m}(g,h) = Q^{K_m}_n(g,h)\] for each $g,h\in \Gamma$ with $\ell(g), \ell(h)\leq r$.
\end{lemma}
\begin{proof}
From Corollaries~\ref{cor:rec-compute-Q} and \ref{cor:rec-compute-QFin}, we see that for each $m>1$ and $n\geq 4$, we have $Q^{K_m}(g,h) = Q^{K_m}_n(g,h)$ if
\begin{align*}
    Q^{K_{m-1}}(g|_{u_1}, h|_{v_1})
    &=Q^{K_{m-1}}_{n-1}(g|_{u_1}, h|_{v_1})\quad\text{and}
    \\
    Q^{K_{m-1}}((g|_{u_1}) (g|_{u_2}), (h|_{v_1}) (h|_{v_2}))
    &= Q^{K_{m-1}}_{n-1}((g|_{u_1}) (g|_{u_2}), (h|_{v_1}) (h|_{v_2}))
\end{align*}
for all $(u_1,u_2),(v_1,v_2)\in \{(0,1),(1,0)\}$.
Thus, from Lemma~\ref{lem:length-contraction}, we see that
\[
    \ell(g|_{u_1}),\ell(h_{v_1}), \ell((g|_{u_1}) (g|_{u_2})), \ell((h|_{v_1}) (h|_{v_2})))\leq \max\{\ell(g),\ell(h)\} +1.
\]
From this, we see that if
\[
    Q^{K_{m-1}}(x,y) = Q^{K_{m-1}}_{n-1}(x,y)
\]
for each $x,y\in \Gamma$ with $\ell(x),\ell(y)\leq \max\{\ell(g),\ell(h)\} + 1$, then
\[
    Q^{K_{m}}(x,y) = Q^{K_{m}}_{n}(x,y).
\]
From induction on $n$ and $m$, we see that, for all $m>1$ and $n\geq m+4$, if
\[
    Q^{K}(x,y) = Q^K_{n-m}(x,y)
\]
for each $x,y\in \Gamma$ with $\ell(x),\ell(y)\leq \max\{\ell(g),\ell(h)\} + m$, then
\[
    Q^{K_{m}}(g,h) = Q^{K_{m}}_{n}(g,h).
\]
From Lemma~\ref{lem:additional-prop-K}, we see that
\[
    Q^{K}(x,y) = Q^{K}_{n-m}(x,y)
\]
for each $x,y\in \Gamma$ with $\ell(x),\ell(y)\leq r+m$ where $n = 4\lceil\log_2(2(r+m))\rceil + 10 + m$.
Thus, from our earlier induction, we have
\[
    Q^{K_m}(g,h) = Q^{K_m}_n(g,h)
\]
for each $\ell(g),\ell(h)\leq r$ where $n = 4\lceil\log_2(2(r+m))\rceil + 10 + m$.
\end{proof}

\subsection{Main result}
We now prove the main theorem of this section as follows.

\GrigTheorem

\begin{proof}
Suppose that $H \leq \Gamma$ is some finite-index subgroup. Then, from Lemma~\ref{lem:csp}, we have 
$K_m \leq \mathrm{Stab}_\Gamma(m)\leq H$ for some $m\geqslant 1$.
We then see that, for any two elements $g,h\in H$, we have that $g$ and $h$ are conjugate in $H$ if and only if the intersection $Q^{K_m}(g,h) \cap H/K_m$ is non-empty.
From Lemma~\ref{lem:additional-prop-K-general}, we then know that there exists some $n\geqslant 1$ such that this intersection can be calculated as
\[
    Q^{K_m}(g,h)\cap H/K_m = Q^{K_m}_n(g,h) \cap H/K_m.
\]
Thus, $g,h\in H$ are conjugate in $H$ if and only if they are conjugate in $H/\mathrm{Stab}_H(n)$.
Thus, the finite-index subgroup $H$ is conjugacy separable.

It is well known that the Grigorchuk group is a 2-group (see Theorem~17 on p.~222 of~\cite{harpe2000}).
Thus, the subgroup $H$, and hence the quotient $H/\mathrm{Stab}_H(n)$, are also 2-groups.
From this, we then conclude that $H$ is 2-conjugacy separable, and thus the Grigorchuk group is 2-hereditarily conjugacy separable.
\end{proof}

\GrigCor
\begin{proof}
    By Theorem \ref{thm:hered-conj-sep}, we see that $G$ is hereditarily conjugacy separable. Then, by Corollary \ref{corollary:centralisers in the completion}, we see that $\overline{C_G(g)} = C_{\widehat{G}^{\C}}(g)$ for every $g \in G$, meaning that $C_G(g)$ is dense in $C_{\widehat{G}^{\C}}(g)$.
\end{proof}

Our proof of the above theorem (i.e.\@ Theorem~\ref{thm:hered-conj-sep}) does not generalise to other groups acting on rooted trees that are known to be conjugacy separable, such as Gupta-Sidki groups $\mathrm{GS}(p)$, which raises the natural question: are the Gupta-Sidki groups $GS(p)$ hereditarily conjugacy separable, and thus $p$-hereditarily conjugacy separable? In particular, we pose the following question.
\begin{question}
    Let $G$ be a regular branch group and suppose that $G$ is contracting and that $G$ has the congruence subgroup property, i.e.~every subgroup of finite index contains a level stabiliser subgroup. Is it true that $G$ is hereditarily conjugacy separable?
\end{question}

\section{Wreath products of cyclic subgroup separable groups}\label{section:wreath products of CSS groups}
The main aim of this section is to show the following proposition.
    \PropositionCSS
If $\C$ is the class of all finite groups, then it satisfies the assumption of Proposition~\ref{proposition:wreath products of cyclic subgroup separable}. Applying Proposition~\ref{proposition:wreath products of cyclic subgroup separable} in the context of the class of all finite groups thus immediately gives us the following corollary.
\begin{corollary}
    \label{corollary:wreath products of CSS groups}
    Suppose that $A$ and $B$ are cyclic subgroup separable groups. Then the wreath product $A \wr B$ is a cyclic subgroup separable group if and only if either of the following is true:
    \begin{enumerate}[label=(\roman*)]
        \item $B$ is finite, or
        \item $A$ is abelian.
    \end{enumerate}
\end{corollary}

Let us note that, unlike in Theorems~\ref{theorem:CCS} and \ref{theorem:CC_HCS}, in the case when the acting group $B$ is infinite, it is not enough to assume that the base group $A$ is abelian and residually-$\C$. Indeed, it is true that every residually-$\C$ abelian group is $\C$-conjugacy separable, but it is not true that every residually-$\C$ abelian group is $\C$-CSS\@. For example, the locally cyclic group
    \begin{displaymath}
        \mathbb{Z}\left[1/2\right] = \langle a_1, a_2, \dots \mid  a_{i+1}^2 = a_i \mbox{ for $i= 1,2,\dots$}\rangle
    \end{displaymath}
is abelian and residually finite, but it is not cyclic subgroup separable. Clearly, $a_j \notin \langle a_i \rangle$ whenever $i < j$. At the same time, it is generally known that 
 every finite quotient of $\mathbb{Z}\left[1/2\right]$ is a cyclic group of odd order, meaning that whenever $\pi \colon \mathbb{Z}\left[1/2\right] \to Q$ is a homomorphism, where $|Q| < \infty$, we see that $\langle \pi(a_i)\rangle = \langle \pi(a_j) \rangle$ for any pair $i,j \in \mathbb{N}$. 

Before we proceed to proving Proposition~\ref{proposition:wreath products of cyclic subgroup separable}, we make some easy observations about how the presence of cyclic groups of prime orders affects which residually-$\C$ groups can be $\C$-CSS\@. Recall that a residually-$p$ group $G$ is $p$-CSS only if $G$ is a torsion group, i.e.~$G$ does not contain an element of infinite order. In fact, it follows that a residually-$p$ group is $p$-CSS if and only if it is a torsion group: if all cyclic subgroups of $G$ are finite, then they are $p$-separable in $G$ because $G$ is a residually-$p$ group. The following lemma generalises this observation.
\begin{lemma}
    \label{lemma:CSS - missing primes iff torsion}
    Let $\C$ be an extension-closed pseudovariety of finite groups and suppose that $C_p \notin \C$ for some prime number $p \in \mathbb{P}$. Then a residually-$\C$ group is $\C$-CSS if and only if it is a torsion group.
\end{lemma}
\begin{proof}
    Suppose that $G$ is a residually-$\C$ group. If $G$ is torsion, then every cyclic subgroup of $G$ is finite and therefore must be $\C$-separable since $G$ is residually-$\C$.
    
    In the opposite direction, suppose that $G$ contains an element $g$ of infinite order. Clearly, $g \notin \langle g^p \rangle$. Since $C_p \notin \C$, we see that $C_{pk} \notin \C$ for any $k \in \mathbb{N}$. This means that, whenever $\pi \colon G \to G/N$ is a natural projection for some $N \in \NC(G)$, the cyclic group $\langle \pi(g)\rangle \leq G/N \in \C$ is isomorphic to $C_n$ for some $n \in \mathbb{N}$ such that $\gcd(p,n) = 1$. That means $\pi(g)^p$ generates $\langle \pi(g)\rangle$ and $\pi(g) \in \langle \pi(g)^p \rangle$. We see that $\langle g^p \rangle$ is not $\C$-closed in $G$ and therefore $G$ cannot be $\C$-CSS.
\end{proof}

Let $\C_1$ and $\C_2$ be pseudovarieties of finite groups and suppose that $\C_1 \subseteq \C_2$, i.e.\@ for every $F \in \C_1$ we have $F \in \C_2$. Given a group $G$, we see that for every $N \in \N_{\C_1}$ we have $N \in \N_{\C_2}$, meaning that the pro-$\C_2$ topology on $G$ is a refinement of the pro-$\C_1$ topology on $G$. One can easily check that, given two pseudovarieties $\C_1$ and $\C_2$, $\proC_2(G)$ is a refinement of $\proC_1(G)$ for every group $G$ if and only if $\C_1 \subseteq \C_2$.
\begin{lemma}
    \label{lemma:CSS infinite order needs all primes}
    Let $\C$ be an extension-closed pseudovariety of finite groups. Then the following are equivalent:
    \begin{itemize}
        \item[(i)] $C_p \in \C$ for every prime $p \in \mathbb{P}$;
        \item[(ii)] $\C$ contains every finite cyclic group;
        \item[(iii)] the infinite cyclic group is $\C$-CSS;
        \item[(iv)] for any group $G$, the pro-$\C$ topology on $G$ is a refinement of the pro-(finite solvable) topology on $G$.
    \end{itemize}
\end{lemma}
\begin{proof}
    We see that (ii) implies (i). Similarly, since every cyclic group can be constructed from cyclic groups of prime order through a finite sequence of group extensions, we see that (i) implies (ii).

    If (ii) holds, then $\C$ must contain every finite solvable group; hence we see that (ii) implies (iv). On the other hand, the class of all finite solvable groups must contain every finite cyclic group; thus (iv) implies (ii).

    Suppose that (ii) holds, and let $\langle g \rangle$ be the infinite cyclic group. Then for any $n \in \mathbb{N}$, we have that $\langle g \rangle /\langle g^n\rangle \simeq C_n \in \C$, and we then immediately see that $\langle g^n\rangle$ is $\C$-closed in $\langle g\rangle$ since $\langle g^n\rangle \in \NC(\langle g\rangle)$.

    Suppose that (iii) holds, and let $\langle g \rangle$ be the infinite cyclic group. Clearly, $\langle g \rangle$ is residually-$\C$, and we then see that $C_p \in \C$ for every prime number $p$ by Lemma~\ref{lemma:CSS - missing primes iff torsion}.
\end{proof}

\subsection{Finite acting group}
\begin{lemma}
    \label{lemma:restriction_cyclic}
    Let $G$ be a $\C$-CSS group, and let $C \leq G$ be a cyclic subgroup. Then $\proC(C)$ is a restriction of $\proC(G)$.
\end{lemma}
\begin{proof}
    Let $N \leq C$ be $\C$-open in $C$. As $C$ is cyclic, by necessity, $N$ is cyclic as well. Since $G$ is a $\C$-CSS group, the subgroup $N$ is $\C$-closed in $G$. Using Lemma~\ref{lemma:restriction}, we immediately get the result. 
\end{proof}

It is easy to see that, unlike $\C$-conjugacy separability, the property of being $\C$-CSS is inherited by subgroups. Thus, if $G$ is a $\C$-CSS group, $H \leq G$, and $h, g \in H$ are given such that $g \notin \langle h \rangle$, then there is some $K \in \NC(G)$ such that $gK \notin \langle h \rangle K$ in $G/K$. Thus, one may consider the restriction $\pi|_H$ of the natural projection $\pi \colon G \to G/K$. The implication in the opposite direction does not hold in general, but holds if $H \leq G$ is $\C$-open in $G$, which we prove as follows.
\begin{lemma}
    \label{lemma:inheriting C-CSS upwards}
    Let $G$ be a group, and suppose that $H \leq G$ is $\C$-open in $G$. If $H$ is $\C$-CSS, then $G$ is $\C$-CSS.
\end{lemma}
\begin{proof}
    Let $H \leq G$ be as above and let $g \in G$ be arbitrary. Let us note that $G$ is residually-$\C$: $H$ is residually-$\C$ by assumption and thus $G$ is residually-$\C$ by \cite[Corollary 2.8]{mf}. Therefore, if $g$ is a torsion element, the cyclic subgroup $\langle g \rangle$ is $\C$-closed in $G$. We see that, without loss of generality, we may assume that $\ord(g) = \infty$. By Lemma~\ref{lemma:closed/open subgroups}, we see that $H \fileq G$; therefore $\langle g\rangle \cap H \fileq \langle g\rangle$, meaning that $\langle g\rangle \cap H = \langle g^k\rangle$ for some $1 \leq k \leq |G:H|$. By assumption, $\langle g^k \rangle$ is 
    $\C$-closed in $H$, and therefore, by Lemma~\ref{lemma:restriction}, it is $\C$-closed in $G$ as well. It then follows that $\langle g \rangle = \bigcup_{i = 0}^{k-1} g^i \langle g^k\rangle$ is $\C$-separable in $G$ as it is a union of finitely many $\C$-separable subsets of $G$.
\end{proof}

\begin{lemma}
    \label{lemma:C-CSS direct products}
    Let $\C$ be an extension-closed pseudovariety of finite groups. Then the class of $\C$-CSS groups is closed under forming direct products.
\end{lemma}
\begin{proof}
    If $C_p \notin \C$ for some prime number, then following Lemma~\ref{lemma:CSS - missing primes iff torsion}, the class of $\C$-CSS groups contains only residually-$\C$ torsion groups. Clearly, a direct product of residually-$\C$ torsion groups is again a residually-$\C$ torsion group and therefore a $\C$-CSS group.
    
    Now, suppose that $C_p \in \C$ for every prime number $p$, let $G_1, G_2$ be $\C$-CSS groups, and let $H \leq G = G_1 \times G_2$ be cyclic. Let $h = (h_1, h_2) \in G_1 \times G_2$ be a generator of $H$. Set $H_1 = \langle h_1 \rangle \leq G_1$ and $H_2 = \langle h_2 \rangle \leq G_2$. Using Lemma~\ref{lemma:restriction_cyclic}, we see that $\proC(H_1)$ is a restriction of $\proC(G_1)$, and $\proC(H_2)$ is a restriction of $\proC(G_2)$. It follows by Lemma~\ref{lemma:restriction_product} that $\proC(H_1 \times H_2)$ is a restriction of $\proC(G_1 \times G_2)$. If we show that $H$ is $\C$-separable in $H_1 \times H_2$, then we are done. We have three cases to consider:
    \begin{enumerate}
        \item[(i)] both $h_1$ and $h_2$ are of finite order,
        \item [(ii)] exactly one of $h_1$ and $h_2$ is of infinite order,
        \item[(iii)] both $h_1$ and $h_2$ are of infinite order.
    \end{enumerate}
    
    Thus, if $H_1 \times H_2$ is finite, then $H < H_1 \times H_2$ is finite as well, and therefore, it is $\C$-closed in $G$ as $G$ is residually-$\C$.
    
    Suppose that exactly one of the elements $h_1, h_2$ is of infinite order; without loss of generality we may assume it is $h_1$. Then we have a short exact sequence
    \begin{displaymath}
        1 \to H_1 \to H_1\times H_2 \to H_2 \to 1.
    \end{displaymath}
    Note that $H_1 \times H_2$ is a residually-$\C$ group since it is a subgroup of a residually-$\C$ group $G$. By Lemma~\ref{lemma:finite subgroups}, we see that $H_2 \in \C$, and therefore, $H_1 \in \NC(H_1 \times H_2)$. Note that since $G_1$ is $\C$-CSS, we see that $\langle h_1^n \rangle$ is $\C$-separable in $G_1$, and therefore, it is $\C$-closed in $H_1$. Furthermore, we see that $H_1$ is $\C$-CSS, and thus we see that $H_1 \times H_2$ is $\C$-CSS by Lemma~\ref{lemma:inheriting C-CSS upwards}. Hence, we see that $H$ is $\C$-closed in $H_1 \times H_2$.
    
    Finally, suppose that both $h_1$ and $h_2$ are of infinite order, i.e.\ both $H_1$ and $H_2$ are infinite cyclic groups. We see that the infinite cyclic group is $\C$-CSS, and thus, by Lemma~\ref{lemma:CSS infinite order needs all primes}, we see that $C_p \in \C$ for every prime number $p$. Let $m,n \in \mathbb{Z}$ be arbitrary integers such that $(h_1^m,h_2^n) \notin H = \langle (h_1,h_2)\rangle$. Clearly, $m \neq n$. Pick some prime number $p \in \mathbb{P}$ such that $p > |m-n|.$ This ensures that $p$ does not divide $m-n$, i.e.\ $m \not\equiv n \mod p$. Consider the natural projection
    \begin{displaymath}
        \pi_{p} \colon H_1 \times H_2 \to  \langle \left(h_1 \rangle/\langle h_1^p \rangle \right) \times  \left(\langle h_2 \rangle/\langle h_2^p \rangle \right) \simeq C_p \times C_p.
    \end{displaymath}
    Since $C_p \in \C$, we see that $C_p \times C_p \in \C$. Suppose that $\pi_p((h_1^m, h_2^n))\in \langle\pi_p((h_1, h_2))\rangle$, i.e.\ $ \pi_p((h_1^m, h_2^n)) = \pi((h_1^e,  h_2^e ))$ for some $e \in \mathbb{Z}$. This means that $e \equiv m \mod p$ and $e \equiv n \mod p$. That implies $m \equiv n \mod p$, which is a contradiction with the choice of $p$. Hence, we see that $\pi_p((h_1^m, h_2^n))\notin \langle\pi_p((h_1, h_2)) \rangle = \pi_p(H)$ and, consequently, $H$ is $\C$-closed in $H_1 \times H_2$.
\end{proof}

Together with Lemma~\ref{lemma:inheriting C-CSS upwards}, the lemma we just stated immediately allows us to show that the class of $\C$-CSS groups is closed under forming wreath products with a finite acting group, thus proving the first half of Proposition~\ref{proposition:wreath products of cyclic subgroup separable}.
 \begin{corollary}
    \label{corollary:C-CSS wreath products, finite act}
     Let $\C$ be an extension-closed pseudovariety of finite groups. If $A$ is a $\C$-CSS group and $B \in \C$, then the wreath product $G = A \wr B$ is a $\C$-CSS group.
 \end{corollary}
 \begin{proof}
     First, let us note that $G = A^B \rtimes B$, and thus $A^B \in \NC(G)$. As $A$ is $\C$-CSS, $A^B$ is cyclic subgroup separable by Lemma~\ref{lemma:C-CSS direct products} as it is a direct product of finitely many copies of $A$. The rest follows by Lemma~\ref{lemma:inheriting C-CSS upwards}.
 \end{proof}

 \subsection{Infinite acting group}
In the rest of this section, we will assume that the base group $A$ is abelian. For the ease of writing, we will use the additive notation for elements of the abelian groups $A$ and $A^B$, while maintaining multiplicative notation for elements of $B$. In particular, we will use the symbol $+$ to denote the group operation in $A$ and $A^B$. Furthermore, for $a \in A$ and $h \in A^B$, we will use $-a$ and $-h$ to denote their respective inverses. In fact, the abelian group $A^B$ can be seen as a $\mathbb{Z}$ module on which $B$ acts by conjugation. For all $f,g \in A^B$, $b,c \in B$, and $k \in \mathbb{Z}$, we then can write:
\begin{align*}
    b\cdot(f+g) &= b \cdot f + b\cdot g,\\
    b \cdot(-f) &= - b\cdot f\\
    b \cdot(kf) &= k (b\cdot f),\\
    (fb) (gc) &= (f + b \cdot g)bc,\\
    (fb)^{-1} &= (-b^{-1}\cdot f) b^{-1}.  
\end{align*}

For improved readability, the proof of the second half of Proposition~\ref{proposition:wreath products of cyclic subgroup separable} will be split into three separate statements.
\begin{lemma}
    \label{lemma:CSS subgroup reduction 1}
     Let $\C$ be an extension-closed pseudovariety of finite groups and suppose that $A, B$ are $\C$-CSS groups such that $A$ is abelian.  For any $f \in A^B$ and $b \in B$, we have that the cyclic subgroup $\langle fb\rangle \leq A \wr B$ is $\C$-closed in $A \wr B$ if and only if it is $\C$-closed in the subgroup $A^B \rtimes \langle b \rangle $.
\end{lemma}
\begin{proof}
Let $A$ and $B$ be groups as in the lemma statement.
Moreover, let $G = A \wr B$ and $G_b = A^B \rtimes \left\langle b\right\rangle$.
We consider the directions of the statement as follows.

\smallskip

\noindent
$(\Longrightarrow)$
Suppose that the  cyclic subgroup $\langle fb\rangle$ is $\C$-closed in $G$ and let $x \in G_b \setminus \langle fb\rangle$ be arbitrary. By assumption, there is $N \in \NC(G)$ such that $xN \cap \langle fb\rangle N = \emptyset$ in $G$. Clearly, $N \cap G_b \in \NC(G_b)$ and $x(N \cap G_b)\cap \langle fb\rangle (N\cap G_b) = \emptyset$ in $G_b$ and thus we see that $\langle fb \rangle$ is $\C$-closed in $G_b$.

\smallskip
\noindent
$(\Longleftarrow)$
Suppose that $\left\langle fb \right\rangle$ is $\C$-closed in $G_b$.
Let $x \in G\setminus \left\langle fb \right\rangle$ be an arbitrary element. Then it is sufficient to show that $x$ is $\C$-separable from the cyclic subgroup $\left\langle f b \right\rangle$.

We begin by noting that $G_b = \rho^{-1}(\left\langle b \right\rangle)$ where $\rho\colon A^B\rtimes B \to B$ is the canonical retraction.
Since $B$ is $\C$-CSS, we then see that $\left\langle b \right\rangle$ is $\C$-closed in $B$. Thus, applying Lemma~\ref{lemma:continuous} with $\rho$, we see that $G_b$ is $\C$-closed in $G$ as it is a preimage of a closed set.

Since $G_b$ is $\C$-closed, we see that if $x\notin G_b$, then $x$ is $\C$-separable from $\left\langle fb\right\rangle\leq G_b$, and we are done.
Hence, in the remainder of this proof, we assume that $x = g b^n$ for some $g\in A^B$ and some $n\in \mathbb Z$.

Since $\left\langle fb \right\rangle$ is $\C$-closed in $G_b$, there exists some subgroup $M_b\in \NC(G_b)$ such that
\[
  gb^n M_b \cap \left\langle fb\right\rangle M_b = \emptyset.
\]
Moreover, since $M_b$ is of finite index in $G_b$, there exists an $m\in \mathbb N$ with $M_b\cap \left\langle b \right\rangle = \left\langle b^m\right\rangle$.
Then, by Lemma~\ref{lemma:order is a multiple}, there is a subgroup $K_1 \in \NC(B)$ such that $K_1 \cap \left\langle b \right\rangle \leq \left\langle b^m \right\rangle$; in particular, there exists some $k\in \mathbb N$ for which $K_1\cap \left\langle b \right\rangle = \left\langle b^{km}\right\rangle$.

To simplify the remainder of this proof, we define the set $\mathbb B$ as $\mathbb B \coloneqq \mathbb Z$ if $b$ is not torsion, and $\mathbb B \coloneqq \{0,{\penalty51}1,{\penalty51}2,{\penalty51}\ldots, \mathrm{ord}(b)-1\}$ if $b$ is torsion of order $\mathrm{ord}(b)$.

We now define $A_q \leq A^B$ as
\begin{displaymath}
    A_q = \left\{ h \in A^B \,\middle|\, \supp(h) \subseteq \{q\}\right\} \leq A^B
\end{displaymath}
for each $q\in B$. Clearly, $A_q \simeq A$. Notice then that $q' \cdot A_q = A_{q'q}$ for each $q'\in B$.
Let $T_b\subseteq B$ be a right traversal of $\left\langle b \right\rangle$ in $B$. We then can write
\[
  G_b
  =
  \left(
    \bigoplus_{c\in T_b} \bigoplus_{i\in \mathbb B} A_{b^i c}
  \right)\rtimes \left\langle b \right\rangle.
\]
Moreover, for each $c\in T_b$, we see that $\bigoplus_{i\in \mathbb Z} A_{b^i c}$ is a normal subgroup of $G_b$.
This becomes clear after considering that
\[
  h \left(\bigoplus_{i\in \mathbb B} A_{b^i c}\right) h^{-1}
  =
  \bigoplus_{i\in \mathbb B} A_{b^i c}
  \qquad\text{and}\qquad
  \left\langle b \right\rangle \cdot \left(\bigoplus_{i\in \mathbb B} A_{b^i c}\right)
  =
  \bigoplus_{i\in \mathbb B} A_{b^i c}
\]
for each $c\in T_b$ and $h\in A^B$.

We now define a finite set of coset representatives $C\subseteq T_b$ as
\[
  C
  =
  \left\{
    c \in T_b \ \middle|\
    \left\langle b \right\rangle c \in \mathrm{supp}(g)
    \ \ \mathrm{or}\ \ 
    \left\langle b \right\rangle c \in \mathrm{supp}(f)
  \right\}.
\]
We may assume, without loss of generality, that $C$ is non-empty. Otherwise, we would have $f =1 =g$ and, consequently, $x = g b^n = b^n \in \left\langle b \right\rangle = \left\langle fb\right\rangle$, which is a contradiction with the choice of $x$.
Moreover, we see that $C$ is finite as $g$ and $f$, by definition, have finite support.

Since $C$ is finite and $B$ is $\C$-CSS, there exists some subgroup $K_2\in \NC(B)$ such that $\left\langle b \right\rangle c_1 K_2 = \left\langle b\right\rangle c_2 K_2$ implies that $\left\langle b\right\rangle c_1 = \left\langle b \right\rangle c_2$. If $\ord(b)\leq \infty$ then we also pick $K_3 \in \NC(B)$ such that $\langle b\rangle \cap K_3 = \{1\}$, i.e.\ the finite subgroup $\langle b \rangle$ embeds into $B/K_3$. Otherwise, if $\ord(b) = \infty$, we set $K_3 = B$.

Let $K_B = K_1 \cap K_2 \cap K_3$.
From Lemma~\ref{lemma:map_extension}, we see that, with
\[
  L_B
  \coloneqq
  \left\{
    h \in A^B
  \ \middle|\ 
    \sum_{k\in K_B} h(kq)=0
    \ \ \mathrm{for\ each}\ q\in B
  \right\},
\]
we have that $L_B K_B$ is the normal closure of $K_B$ in $G$, and $G/L_B K_B = A\wr (B/K_B)$. It follows from Corollary~\ref{corollary:C-CSS wreath products, finite act} that $G/L_B K_B = A\wr (B/K_B)$ is $\C$-CSS.

Let $\pi \colon G \to G/L_B K_B$ denote the quotient map.
We now consider the restriction of this map $\pi_C = \pi|_{G_C}$ where
\[
  G_C
  \coloneqq
  \left(
    \bigoplus_{c\in C} \bigoplus_{i\in \mathbb B} A_{b^i c}
  \right)\rtimes \left\langle b \right\rangle.
\]
From our definition of $C$, we see that $fb, gb^n \in G_C$.

\medskip

\noindent
\textbf{Claim:} $\ker(\pi_C) \leq M_b$.

\smallskip

\noindent
\textbf{Proof of Claim:}
For each $c_1,c_2\in C$, we see that if $c_1 \neq c_2$, then
\[
  \left\langle b\right\rangle c_1 K_B \cap
  \left\langle b\right\rangle c_2 K_B
  \subseteq
  \left\langle b\right\rangle c_1 K_2 \cap
  \left\langle b\right\rangle c_2 K_2
  =\emptyset.
\]
From this, we then observe that
\[
  \pi_C(G_C)
  =
  \left(
    \bigoplus_{c \in C} \bigoplus_{0 \leq i < \beta} A_{(\pi(b))^i \pi(c)}
  \right)
  \rtimes \left\langle \pi(b) \right\rangle
\]
where $\beta = \mathrm{ord}(\pi_C(b))$. By construction, if $\ord(b) < \infty$, then $\beta = \ord(b)$ and, consequently, we see that $\pi_C$ is injective and $G_C \simeq \pi_C(G_C)$. In particular, we see that $\ker(\pi_C) = \{1\} \leq M_b$. 

Now, suppose that $\ord(b) = \infty$. Notice that $K_B\cap \left\langle b \right\rangle \leq K_1 \cap \left\langle b \right\rangle \leq \left\langle b^m\right\rangle$.
Thus, there exists some $k'\in \mathbb N$ for which $K_B\cap \left\langle b \right \rangle = \langle b^{k'm}\rangle$.
Moreover, we note here that $\beta = k'm$.

We then see that for $h \in A^B$ we have $h \in \ker(\pi_C)$ if and only if for every $c\in C$ and every $j\in \{0,1,2,\ldots, \beta-1\}$ the following holds:
\begin{equation}\label{eq:sum1}
  \sum_{i\in \mathbb Z} h(b^{j+ik'm} c) = \sum_{i\in \mathbb Z} h(b^{j+i\beta}) = 0.
\end{equation}
Further, we see that the normal closure of $\langle b^m \rangle$ in $G_b$ is generated by elements of the form $hb^{m}h^{-1} b^i = (h - b^{m} \cdot h)b^i$, where $h \in A^B$ and $i\in \mathbb{Z}$. We see that $h$ belongs to the normal closure of $\left\langle b^m\right\rangle = M_b\cap \left\langle b \right\rangle$ in $G_b$ if and only if
\begin{equation}\label{eq:sum2}
  \sum_{i\in \mathbb Z} h(b^{im}q)=0
\end{equation}
for each $q \in B$.

It can be seen that if $h\in A^B$ satisfies property~(\ref{eq:sum1}), then it also satisfies property~(\ref{eq:sum2}).
That is, $\ker(\pi_C)$ is a subgroup of the normal closure $\left\llangle b^m \right\rrangle_{G_b}$.

We then see that
\[
  \ker(\pi_C) \leq \left\llangle b^m \right\rrangle_{G_b} =
  \left\llangle \left\langle b \right\rangle \cap M_b \right\rrangle_{G_b}
  \leq M_b.
\]
This holds as the subgroup $M_b \in \NC(G_b)$ is normal in $G_b$ by definition, which proves the claim. \newline

Recall that $M_b$ was chosen such that
$
  gb^n M_b \cap \left\langle fb\right\rangle M_b = \emptyset.
$
Thus, from our claim, we have
$
  \pi_C(gb^n) \notin \left\langle \pi_C(fb)\right\rangle.
$
Hence,
$
  \pi(gb^n)\notin \left\langle \pi(fb)\right\rangle.
$
This then immediately implies that
$
  \pi(gb^n)\notin \pi(\left\langle fb\right\rangle).
$
Thus, $gb^n$ is $\C$-separable from $\left\langle fb\right\rangle$ and we see that the cyclic subgroup $\left\langle fb\right\rangle$ is $\C$-closed in $G$.
\end{proof}
Similar to Lemma \ref{lemma:CC - B in C,g in B}, one could use topological arguments in the pro-$\C$ completion to give a shorter and less technical proof. As before, we decided to give an elementary proof that does not involve the use of pro-$\C$ completion.

As an immediate application of Lemma \ref{lemma:CSS subgroup reduction 1}, we get the following statement.
\begin{corollary}
    \label{corollary:CSS torsion acting elements}
     Let $\C$ be an extension-closed pseudovariety of finite groups and suppose that $A, B$ are $\C$-CSS groups such that $A$ is abelian. If $b \in B$ is a torsion element, then, for any $f \in A^B$, the cyclic subgroup $\langle fb\rangle \leq A \wr B$ is $\C$-closed in $A \wr B$. Furthermore, if $B$ is a torsion group, then $A \wr B$ is a $\C$-CSS group.
\end{corollary}
\begin{proof}
    Suppose that $b \in B$ is torsion and let $f\in A^B$ be arbitrary. By the previous lemma, the cyclic subgroup $\langle fb \rangle$ is $\C$-closed in $A \wr B$ if and only if it is $\C$-closed in the subgroup $A^B \rtimes \langle b \rangle$. Let us note that $A^B = \bigoplus_{b \in B}A$ is a direct sum of $\C$-CSS groups, and therefore it is a $\C$-CSS group by Lemma \ref{lemma:C-CSS direct products}. Further, by Lemma \ref{lemma:finite subgroups}, we see that $\langle b\rangle \in \C$ and therefore $A^B \in \NC(A^B \rtimes \langle b\rangle)$. By Lemma \ref{lemma:inheriting C-CSS upwards}, we then see that $A^B \rtimes \langle b \rangle$ is a $\C$-CSS group as well and, consequently, the cyclic subgroup $\langle fb \rangle$ is $\C$-closed in $A^B \rtimes \langle b \rangle$. We see that $\langle fb \rangle$ is $\C$-closed in $A^B$. As a consequence, we immediately see that if $B$ is a torsion group, then $A \wr B$ is a $\C$-CSS group.
\end{proof}

Before we proceed, let us note that using the additive notation, for $f \in A^B$, $b \in B$, and $n \in \mathbb{Z}$, we may write
\begin{align*}
    (fb)^n  &= (f + b\cdot f + b^2 \cdot f + \dots + b^{n-1} \cdot f)b^n\\
            &= \left( \sum_{i=0}^{n-1} b^i \cdot f\right)b^n. 
\end{align*}
For the ease of writing, we will use $f^{b\otimes n}$ to denote $\sum_{i=0}^{n-1} b^i \cdot f$. One can easily check that for $m,n \in \mathbb{N}$, we have $f^{b\otimes(m+n)} = f^{b\otimes n} + b^n \cdot f^{b\otimes m}$. Using this notation, if $\ord(b) = \infty$, we see that $gb^m \in \langle fb \rangle$ if and only if $g = f^{b \otimes m}$. Similarly, if $\ord(b) = n$, then $gb^m \in \langle fb \rangle$ if and only if $g = f^{b \otimes (m + kn)} = f^{b \otimes m} + k f^{b \otimes n}$ for some $k \in \mathbb{Z}$.
\begin{lemma}
    \label{lemma:CSS lamplighter reduction}
     Let $A$ be abelian, and let $\langle b \rangle \simeq \mathbb{Z}$ be the infinite cyclic group. Then for any $f,g \in A^{\langle b \rangle}$ and $m \in \mathbb{Z}$ such that $gb^m \notin \langle fb \rangle$, there is $n \in \mathbb{N}$ such that for the natural projection
    \begin{displaymath}
        \pi_n \colon A \wr \langle b \rangle \to A \wr \left(\langle b \rangle / \langle b^n\rangle \right),
    \end{displaymath}
    we have that $\pi_n(gb^m) \notin \langle \pi(fb) \rangle$.
\end{lemma}
\begin{proof}
    First, let us note that we may without loss of generality assume that $\supp(f) \subseteq \{1\}$ (notice that if $f$ were the trivial function, then $\supp(f)=\emptyset\subset \{1\}$). By Lemma~\ref{lemma:support in distinct cosets}, there is an element $h \in A^{\langle b \rangle}$ such that the element $h (fb)(-h) = f'b$, where $f' = h + f - b\cdot h$ and all elements of $\supp(f')$ lie in distinct cosets of $\langle b \rangle$. Since $b$ generates the whole acting group $\langle b \rangle$, we see that $|\supp(f')| \leq 1$, i.e.\ $\supp(f') \subseteq \{b^k\}$ for some $k \in \mathbb{Z}$. Furthermore, if $k \neq 0$, we may conjugate the element $f'b$ by $b^{-k}$. Clearly, for any $K \unlhd G$, we have that 
    \begin{align*}
        (b^{-k}h) gb^n (hb^{-k})^{-1}K \cap \langle (hb^{-k})fb(hb^{-k})^{-1} \rangle K &=\\
        (hb^{-k})(gb^n K \cap \langle fb \rangle K)(hb^{-k})^{-1} &= \emptyset
    \end{align*}
    if and only if $gb^n K \cap \langle fb \rangle K = \emptyset$.

    By assumption, $gb^n \notin \langle fb \rangle$, which means $g \neq f^{b \otimes m}$. Therefore, there is some $j\in \mathbb{Z}$ such that $b^j \in S  = \supp(f^{{b \otimes m}})\cup \supp(g)$ and $g(b^j) \neq f^{{b \otimes m}}(b^j)$. Write $S_\mathbb{Z} = \{i \in \mathbb{Z} \mid b^i \in S\}$, and let $l_1, l_2 \in \mathbb{Z}$ be such that $S_{\mathbb{Z}} \subseteq [l_1,l_2]$. Set $n = 2(l_2 - l_1)$, and consider the map
    \begin{displaymath}
        \pi \colon A \wr \langle b \rangle \to A \wr (\langle b \rangle / \langle b^n \rangle) \simeq A \wr C_n,
    \end{displaymath}
    where $C_n$ is the cyclic group of order $n$. We note that, by construction, $\pi(f^{b \otimes m})(b^i \langle b^n \rangle) = f^{b \otimes m}(b^i)$ and $\pi(g)(b^i \langle b^n \rangle) = g(b^i)$ for every $i \in S_{\mathbb{Z}}$. Suppose that $\pi(gb^m) \in \langle \pi(fb)\rangle$, i.e.\ $\pi(gb^n) = \pi(fb)^e$ for some $e \in \mathbb{Z}$. By necessity, $e \equiv m \mod n$, meaning that $e = m + kn$ for some $k \in \mathbb{Z}$. Then
    \begin{align*}
        \pi(fb)^e = \pi\left((fb)^e\right) &= \pi\left((fb)^{m + kn}\right)\\
                                &= \pi\left(f^{b \otimes(m+kn)} b^{m+kn}\right)\\
                                &= \pi\left(f^{b \otimes(m+kn)}\right)\pi\left( b^{m+kn}\right)\\
                                &= \pi\left(f^{b \otimes m} + b^{kn}\cdot f^{b \otimes kn}\right) \pi(b^m)\\
                                &= \left(\pi\left(f^{b \otimes m}\right) + k\pi\left(f^{b \otimes kn}\right)\right) \pi(b^m).
    \end{align*}
    That means $\pi(g) = \pi\left(f^{b \otimes m}\right) + k\pi\left(f^{b \otimes kn}\right)$. Denote $f(1) = a$. Since $\supp(f) = \{1\}$, we see that
    \begin{displaymath}
        f^{b \otimes n}(b^i) = \begin{cases}
            a &\mbox{ if $i \in \{0, 1, \dots, n-1\}$},\\
            0    &\mbox{ otherwise}.
        \end{cases}
    \end{displaymath}
    In particular, we see that $\pi(g) - \pi(f^{b \otimes m}) = \pi(g - f^{b \otimes m}) = k \overline{a}$ where $\overline{a} \colon \langle b \rangle / \langle b^n\rangle \to A$ is the function given by $\overline{a}(x) = a$ for all $x \in \langle b \rangle / \langle b^n\rangle$, i.e.\ $\pi(g - f^{b \otimes m})$ must be constant on $\langle b \rangle / \langle b^n\rangle$. Since $g \neq f^{b \otimes m}$, there is some $j \in S_{\mathbb{Z}}$ such that $g(b^j) \neq f^{{b \otimes m}}(b^j)$, meaning that $\pi(g - f^{b \otimes m})(r^j\langle b^n \rangle) \neq 0$. At the same time,
    \begin{displaymath}
        \pi(g)(b^i \langle b^n \rangle) = 0 = \pi(f^{b \otimes m})(b^i \langle b^n \rangle)
    \end{displaymath}
    for all $i \in [0, \dots, n-1] \setminus S \mod n$, which is a contradiction. We see that $\pi(gb^m) \notin \langle \pi(fb)\rangle$.
\end{proof}

\begin{lemma}
    \label{lemma:CSS subgroup reduction 2}
    Let $\C$ be an extension-closed pseudovariety of finite groups and suppose that $A, B$ are $\C$-CSS groups such that $A$ is abelian. For any $f \in A^B$ and $b \in B$, we have that the cyclic subgroup $\langle fb\rangle \leq A \wr B$ is $\C$-closed in $A^B \rtimes \langle b \rangle$.
\end{lemma}
\begin{proof}
    For this proof, let us denote $G = A \wr B$ and $G_b = A^B \rtimes \langle b \rangle$. Before we proceed with the main part of the proof, let us note that if $b$ is a torsion element then the cyclic subgroup $\langle fb\rangle$ is $\C$-closed in $A \wr B$ by Corollary \ref{corollary:CSS torsion acting elements}. We see that we may without loss of generality assume that $\ord(b) = \infty$. Then, since it is a subgroup of a $\C$-CSS group, we see that $\langle b\rangle \cong \mathbb{Z}$ is also a $\C$-CSS group and, by Lemma \ref{lemma:CSS infinite order needs all primes}, we may assume that $\C$ contains all finite cyclic groups.

     Let $g \in A^B$ and $m \in \mathbb{Z}$ be given such that $gb^m \notin \langle fb \rangle$. Note that this is the case if and only if $g \neq f^{b \otimes m}$.
    
    Let $T_b \subseteq B$ be a right transversal for $\langle b \rangle$ in $B$. For $h \in A^B$ and $c \in T_b$, we define $h_c \in A^B$ as
    \begin{displaymath}
        h_c(x) = \begin{cases}
            h(x) &\mbox{ if $x \in \langle b \rangle$c},\\
            0    &\mbox{ otherwise}.   
        \end{cases}
    \end{displaymath}
    We see that $h = \sum_{c \in T_b} h_c$, and for any pair $h,h' \in A^B$, we have $h = h'$ if and only if $h_c = h'_c$ for all $c \in T_b$. It follows that $g = f^{b\otimes m}$ if and only if $g_c = f^{b\otimes m}_c$ for all $c \in T_b$. By assumption, $g \neq f^{b\otimes m}$, so there is some $c' \in T_b$ such that $g_{c'} \neq f^{b\otimes m}_{c'}$. Consider the map
    \begin{displaymath}
        \pi_{c'} \colon G_b = \left( \bigoplus_{c \in T_b} \left( \bigoplus_{i \in \mathbb{Z}} A_{b^i c}\right) \right)\rtimes \langle b \rangle \to \left( \bigoplus_{i \in \mathbb{Z}} A_{b^i c'}\right) \rtimes \langle b \rangle \simeq A \wr \langle b\rangle
    \end{displaymath}
    given by $\pi_{c'}(hb^i) = h_{c'} b^i$. Let us note that, for the ease of writing, we are slightly abusing notation, and consider $h_{c'}$ both as function defined on $B$, i.e.\ \begin{displaymath}
        h_{c'} \in A^B = \bigoplus_{q \in B} A_q
    \end{displaymath}as well as a function defined on the coset $\langle b \rangle c'$, i.e.
    \begin{displaymath}
        h_{c'} \in A^{\langle b \rangle c'} = \bigoplus_{i \in \mathbb Z} A_{b^i c'}.
    \end{displaymath}
    Clearly, $\pi_{c'}(g) \neq \pi_{c'}(f^{b \otimes m})$, and so, $\pi_{c'}(gb^m) \notin \langle \pi_{c'}(fb)\rangle$. By Lemma~\ref{lemma:CSS lamplighter reduction}, there is $n \in \mathbb{N}$ such that for the map $\pi_n \colon A \wr \langle b\rangle \to A \wr C_n$, we have that $\pi_n(\pi_{c'}(gb^m)) \notin \langle \pi_n(\pi_{c'}(fb) \rangle$. Let us denote $\pi = \pi_n \circ \pi_{c'} \colon G_b \to A \wr C_n$. By assumption, $C_n \in \C$, and therefore $A \wr C_n$ is $\C$-CSS by Corollary~\ref{corollary:C-CSS wreath products, finite act}. As we just demonstrated, $\pi(gb^m) \notin \langle \pi(fb) \rangle$ and $A \wr C_n$ is $\C$-CSS, we see that the cyclic subgroup $\langle fb \rangle$ is $\C$-closed in $G_b$ by Lemma~\ref{lemma:lazy lemma separability}.
\end{proof}

\begin{corollary}
    \label{corollary:C-CSS wreath products, infinite act}
    Let $\C$ be an extension-closed pseudovariety of finite groups and suppose that $A, B$ are $\C$-CSS groups such that $A$ is abelian. Then the wreath product $G = A \wr B$ is a $\C$-CSS group.
\end{corollary}
\begin{proof}
    Let $f \in A^B$ and $b \in B$ be arbitrary. By Lemma~\ref{lemma:CSS subgroup reduction 1}, we see that the subgroup $\langle fb\rangle$ is $\C$-closed in $G$ if and only if it is $\C$-closed in $G_b = A^B \rtimes \langle b \rangle$. By Lemma~\ref{lemma:CSS subgroup reduction 2}, we see that the cyclic subgroup $\langle fb \rangle$ is $\C$-closed in $G_b$, and therefore, it is $\C$-closed in $G$, which concludes the proof.
\end{proof}

We are now ready to prove Proposition~\ref{proposition:wreath products of cyclic subgroup separable}.
\begin{proof}[Proof of Proposition~\ref{proposition:wreath products of cyclic subgroup separable}]
Suppose that the wreath product $G = A \wr B$ is a $\C$-CSS group. By necessity, $G$ is also residually-$\C$. Then, by \cite[Theorem~3.1]{gruenberg}, we see that either $B \in \C$ or $A$ is abelian.

Now, suppose that $A$ and $B$ are given as above. If $B \in \C$, then $G = A \wr B$ is $\C$-CSS by Corollary~\ref{corollary:C-CSS wreath products, finite act}. If $B$ is infinite and $A$ is abelian, then $G$ is $C$-CSS by Corollary~\ref{corollary:C-CSS wreath products, infinite act}.
 \end{proof}

Let $S_{r,k}$ denote the free solvable group of rank $r$ and derived length $k$, i.e. $S_{r,k} = F_r / F_r^{(k)}$ where $F_r$ is the free group of rank $r$ and $F_r^{(k)}$ is the $k$-th derived subgroup of $F_r$. There exists a well-known embedding $\mathfrak{M}_{r,k+1} \colon S_{r,k+1} \to \mathbb{Z}^r \wr S_{r,k}$ called Magnus' embedding, known to satisfy a number of valuable properties (see \cite[11.3]{drutu_kapovich}). For the purpose of this note, we only need to know that $\mathfrak{M}_{r,k}$ is an injective group homomorphism and therefore, with a slight abuse of notation, we can write $S_{r,k+1} \leq \mathbb{Z}^r \wr S_{r,k}$.
\freesolvableCSS
\begin{proof}
    Suppose that $C_n \notin \C$ for some $n \in \mathbb{N}$. Then $S_{r,k}$ cannot be $\C$-CSS by Lemma \ref{lemma:CSS infinite order needs all primes} as it is a torsion-free group and therefore contains an element of infinite order.

    Let $\C$ be an extension-closed pseudovariety of finite groups containing all finite cyclic groups.  We fix $r > 1$ and proceed by induction on $k$. If $k=1$ then $S_{r,1} = \mathbb{Z}^r$ is $\C$-CSS by Lemmas \ref{lemma:CSS infinite order needs all primes} and \ref{lemma:C-CSS direct products}. Now,suppose that the statement holds for $k$, i.e. $S_{r,k}$ is $\C$-CSS, and consider $S_{r,k+1}$. By Proposition \ref{proposition:wreath products of cyclic subgroup separable}, we see that $\mathbb{Z}^r \wr S_{r,k}$ is a $\C$-CSS group. Since $S_{r,k+1} \leq \mathbb{Z}^r \wr S_{r,k}$, we see that $S_{r,k+1}$ is $\C$-CSS as the class of $\C$-CSS groups is closed under taking subgroups.
\end{proof}

Another useful property of Magnus' embedding, established in \cite[Theorem 1]{magnus_remeslennikov}, is that it is a \emph{conjugacy embedding}, i.e. for $f,g \in S_{r,k+1}$ we have that $\mathfrak{M}_{r,k+1}(f) \sim \mathfrak{M}_{r,k+1}(g)$ in $\mathbb{Z} \wr S_{r,k}$ if and only if $f \sim g$ in $S_{r,k+1}$. This allows us to state the corollary. Note that if $\C$ is the class of all finite groups, we recover \cite[Corollary 6]{magnus_remeslennikov}.

\begin{lemma}
    Let $\C$ be an extension-closed pseudovariety of finite groups. Then the free solvable group $S_{r,k}$ is $\C$-CS whenever $\C$ contains all finite cyclic groups. In particular, $ S_{r,k}$ is conjugacy separable.
\end{lemma}
\begin{proof}
    Let $\C$ be an extension-closed pseudovariety of finite groups containing all finite cyclic groups.  We fix $r > 1$ and proceed by induction on $k$. If $k=1$, then $S_{r,1} = \mathbb{Z}^r$ is abelian and therefore it is $\C$-CS if and only if it is residually-$\C$, and free abelian groups are residually-$\C$ for any extension-closed pseudovariety of finite groups. 
    
    Now, suppose that the statement holds for $k \geq 1$, i.e. $S_{r,k}$ is $\C$-CSS, and consider $S_{r,k+1}$ and let $f,g \in S_{r, k+1}$ be given such that $f \not\sim g$ in $S_{r,k+1}$. As stated above, by \cite[Theorem 1]{magnus_remeslennikov}, we have that $\mathfrak{M}_{r,k+1}(f) \not\sim \mathfrak{M}_{r,k+1}(g)$ in $\mathbb{Z} \wr S_{r,k}$. By Corollary \ref{corollary:free solvable groups are CSS}, we have that $S_{r,k}$ is a $\C$-CS group and thus, by Theorem \ref{theorem:CCS}, we see that the wreath product $\mathbb{Z}^r \wr S_{r,k}$ is a $\C$-CS group, so there is $N \in \NC(\mathbb{Z}^r \wr S_{r,k})$ such that
    \begin{displaymath}
        \pi_N(\mathfrak{M}_{r,k+1}(f)) \not\sim \pi_N(\mathfrak{M}_{r,k+1}(g)) \mbox{ in } (\mathbb{Z} \wr S_{r,k})/N \in \C,
    \end{displaymath}
    where $\pi_N \colon \mathbb{Z} \wr S_{r,k} \to (\mathbb{Z} \wr S_{r,k})/N$ is the natural projection. Let $\pi \colon S_{r,k+1} \to (\mathbb{Z} \wr S_{r,k})/N$ be the homomorphism given by $\pi = \pi_N \circ \mathfrak{M}_{r,k+1}$ and set $K = \ker(\pi) \normfileq S_{r, k+1}$. By construction, $\pi(f) \not\sim \pi(g)$ in $(\mathbb{Z} \wr S_{r,k})/N$. Further, $S_{r,k+1}/K \leq (\mathbb{Z} \wr S_{r,k})/N$, and thus $K \in \NC(S_{r, k+1})$ as the class $\C$ is closed under taking subgroups. We see that the free solvable group $S_{r,k+1}$ is $\C$-CS.
\end{proof}

\section{Wreath products of solvable groups with abelian base}
\label{section: Wreath prouducts of solvable groups}
In this section, we study wreath products of solvable groups. In particular, we show that an iterated wreath product of $k$ residually finite, cyclic subgroup separable abelian groups produces a $k$-step solvable group that is hereditarily conjugacy separable and cyclic subgroup separable.

The following lemma will be used to establish a lower bound on the length of the derived series of a group of the form $A \wr G$, where $A$ is abelian and $G$ is solvable.
\begin{lemma}
    \label{lemma:descending series}
     Let $G$ be a group, and let $G = G_0 > G_1 > \dots > G_k$ be a proper descending series of subgroups. Let $x_0, x_1, \dots, x_{k -1} \in G$ be a sequence of non-trivial elements such that for all $i \in \{0, \dots ,k-1\}$, we have $x_i \in G_i \setminus G_{i+1}$. Then for the subsets $S_i \subset G$, where $S_0 = \{1\}$ and $S_{i+1} = S_i \cup x_iS_i$, we have $S_i \cap x_iS_i = \emptyset$.
 \end{lemma}
 \begin{proof}
    We will prove the lemma by induction on $k$. If $k = 0$ then, since $x_0$ is non-trivial, we see that the sets $S_0 = \{1\}$ and $x_0 S_0 = \{x_0\}$ are disjoint. Now, suppose that the statement holds for $i >0$, and let $x_0, x_1, \dots, x_i \in G$ be a sequence of non-trivial elements satisfying the conditions of this lemma.
    
     \noindent\textbf{Claim}: Every element of $S_i$ can be written as a product of the form $x_{j_1} x_{j_2} \dots x_{j_n}$ with $i-1 \geq n$ and $i - 1 \geq j_1 > j_2 > \dots > j_n \geq 0$.

     \noindent \textbf{Proof of Claim:} We proceed by induction on $i$. For $i = 1$, we see that $S_1 = S_0 \cup x_0S_0 = \{1, x_0\}$ and the statement holds trivially. Now, suppose that the statement holds for $i-1$, and consider some $w \in S_i = S_{i-1} \cup x_{i-1} S_{i-1}$. If $w \in S_{i-1}$, then by the induction hypothesis, $w$ can be written as a product $w= x_{j_1} x_{j_2} \dots x_{j_n}$ with $i-2 \geq n$ and $i - 2 \geq j_1 > j_2 > \dots > j_n \geq 0$. Similarly, if $w \in x_{i-1}S_{i-1}$, then by the induction hypothesis, $w$ can be written as a product $w= x_{i-1}x_{j_1} x_{j_2} \dots x_{j_n}$ with $i-2 \geq n$ and $i - 2 \geq j_1 > j_2 > \dots > j_n \geq 0$, and we are done.
     
     Now, suppose that $x_i S_i \cap S_i \neq \emptyset$, meaning that $x_i s = t$ for some $s,t \in S_i$. Since, following the claim above, every element of $S_i$ can be written as a product of elements in $\{x_0, \dots, x_{i-1}\}$, we can decompose both $s$ and $t$ as $s = ux_l w$ and $t = vx_{m}w$, where $l \neq m$. Therefore, we have $x_i u x_l = v x_m$.  Note that it follows that $u \in G_{i_l}$ for some $i_l > l$ and $v \in G_{i_m}$ for some $i_m > m$. If $l< m$, we may write $x_l = u^{-1} x_i^{-1} v x_m $. However, since $x_i, x_m u,v \in G_{\min\{i_l,i_m\}} \leq G_{l+1},$ we have $x_l = u^{-1} x_i^{-1} v x_m \in G_{l+1}$, which is a contradiction. Similarly, if $l > m$, then we can write $x_m = v^{-1} x_i u x_l \in G_{m+1}$, which is again a contradiction. It follows that the sets $S_i$ and $x_i S_i$ must be disjoint.
 \end{proof}

\begin{lemma}\label{lem:wreath product solvable iterated}
     Let $G$ be a solvable group of derived length $k$, and let $A$ be an abelian group. Then $A \wr G$ is a solvable group of derived length $k+1$. Furthermore, if $G$ is infinite, then $A \wr G$ is not polycyclic.
 \end{lemma}
\begin{proof}
    It is an easy exercise to check that if $H$ is a solvable group arising as a group extension
    \begin{displaymath}
        1 \to K \to H \to Q \to 1,
    \end{displaymath}
    then the derived length of $H$ is bounded from above by the sum of derived lengths of $K$ and $Q$. In our case, we have
    \begin{displaymath}
        1 \to \bigoplus_{g \in G}A \to A \wr G \to G \to 1.
    \end{displaymath}
    By assumption, $\bigoplus_{g \in G}A$ is abelian and therefore of derived length 1 and $G$ is of derived length $k$, so we see that the derived length of $A \wr G$ is bounded from above by $k+1$.
    
    To see that $A \wr G$ is of derived length exactly $k+1$, we will show that $\left(A \wr G\right)^{(k)}$ is non-trivial by inductively constructing a sequence of functions $f_0, f_1, \dots, f_k$ such that $f_i \in \left(A \wr G\right)^{(i)} \setminus \left(A \wr G \right)^{(i+1)}$ for $i = 0, \dots, k$. Let $a \in A$ be non-trivial, and let $f_{0} \in A^G$ be a function given by
    \begin{displaymath}
        f_0(x) =
        \begin{cases}
            a \mbox{ if } x=1,\\
            0 \mbox{ otherwise.}
        \end{cases}
    \end{displaymath}
    Finally, let $x_0, x_1, \dots, x_{k-1} \in G$ be a sequence of elements such that $x_i \in G^{(i)} \setminus G^{(i+1)}$ for $i = 0, \dots, k$. By Lemma~\ref{lemma:descending series}, if we set $S_0 = \{1\}$ and $S_{i+1} = S_i \cup x_iS_i$, the sets $S_i$ and $x_iS_i$ are always disjoint. For $i = 1, \dots, k$, we then define 
    $$
    f_i = [f_{i-1}, x_{i-1}] = f_{i-1} - x_{i-1}\cdot f_{i-1}.
    $$ Clearly, $f_i \in (A \wr G )^{(i)}$. Furthermore, for $i=0, \dots, k-1$, we claim  that $\supp(f_i) = S_i$. Indeed, we proceed by induction, and note that $\supp(f_0) = \{x_0\} = S_0$. Since the sets $S_i$ and $x_i S_i$ are disjoint, we see that $\supp(f_i) \cap x_i \cdot \supp(f_i) = \emptyset$. We then get
    \begin{displaymath}
        \supp(f_{i+1}) = \supp(f_i - x_i \cdot f_i) = \supp(f_i) \cup x_i \supp(f_i) = S_i \cup x_i S_i = S_{i+1},
    \end{displaymath}
    giving our claim. This means that $f_i$ is non-trivial for every $i \in \{1, \dots, k\}$. In particular, we see that $f_k \in (A\wr G)^{(k)}$ is non-trivial, and hence, $A\wr G$ is of derived length at least $k+1$.  

    Finally, we note that if $G$ is infinite, then either $\mathbb{Z}^\infty$ is a subgroup or $\mathbb{F}_p^\infty$ for some prime $p \in \mathbb{P}$ is a subgroup of $A \wr G$. As neither of those is a polycyclic group, we see that $A \wr G$ is not polycyclic.
\end{proof}

 \begin{lemma}
     Suppose that $\mathcal{A} = \left(A_i \right)_{i=0}^{\infty}$ is a sequence of residually finite, cyclic subgroup separable abelian groups. Set $G_1(\mathcal{A}) = A_1$, and recursively define
     \begin{displaymath}
         G_{i+1}(\mathcal{A}) = A_{i+1} \wr G_i(\mathcal{A}).
     \end{displaymath} Then $G_k(\mathcal{A})$ is a $k$-step solvable, hereditarily conjugacy separable, and cyclic subgroup separable group. Furthermore, if $A_i$ is torsion-free for every $i \leq k$, then $G_k(\mathcal{A})$ is torsion-free as well. Furthermore, if $A_j$ is infinite for some $j \geq 2$, then $G_{j'}(\mathcal{A})$ is not polycyclic for all $j' \geq j$.
 \end{lemma}
 \begin{proof}
     Using induction, Corollary~\ref{corollary:wreath products of CSS groups} implies that $G_k(\mathcal{A})$ is a cyclic subgroup separable group for all $k$, and Lemma~\ref{lem:wreath product solvable iterated} implies that $G_k$ is a $k$-step solvable group by induction. Similarly, using Theorem~\ref{theorem:HCS} and induction, we see that $G_k(\mathcal{A})$ is a hereditarily conjugacy separable group. Altogether, we see that $G_k(\mathcal{A})$ is a $k$-step solvable group and it is hereditarily conjugacy separable. We finish by noting that $G_k(\mathcal{A})$ is clearly seen to be torsion-free when $A_i$ is torsion-free for all $i \leq k$.
 \end{proof}

\IteratedWreath

\appendix
\section{\texorpdfstring{$\C$}{C}-centraliser condition and the \texorpdfstring{$\C$}{C}-completion}

 The aim of the appendix is to characterise residually-$\C$ groups satisfying the $\C$-centraliser condition in terms of their pro-$\C$ completion, thus generalising \cite[Corollary~12.2]{ashot_raags} to the setting of pro-$\C$ topologies, where $\C$ is an extension-closed pseudovariety of finite groups. Background on pro-$\C$ completions can be found in the books \cite{rz} and \cite{ribes2017profinite}.

First, let us fix some notation. For the rest of the appendix, we will assume that $G$ is a residually-$\C$ group. For $N \in \NC(G)$, we will write $\pi_N \colon G \to G/N$ to denote the natural projection. Given $N_1, N_2 \in \NC(G)$ such that $N_1 \leq N_2$, we will write $\pi_{N_2, N_1} \colon G/N_1 \to G/N_2$ to denote the unique homomorphism such that $\pi_{N_2} = \pi_{N_1, N_2} \circ \pi_{N_1}$, meaning that $\pi_{N_2, N_1}(gN_1) = gN_2$ for all $gN_1 \in G/N_1$. Clearly, if $N_1, N_2, N_3 \in \NC(G)$ such that $N_1 \leq N_2 \leq N_3$, then $\pi_{N_3, N_2} \circ \pi_{N_2, N_1} = \pi_{N_3, N_1}$. This naturally defines an inverse system on the collection of finite groups of the form $G/N$, where $N \in \NC(G)$, and the pro-$\C$ completion $\widehat{G}^{\C}$ of $G$ is the inverse limit of this system. Moreover, there exists a canonical embedding of $\widehat{G}^{\C}$ into $\prod_{N \in \NC(G)} G/N$. Thus, $\widehat{G}^{\C}$ can be equipped with the topology induced by the product topology on $\prod_{N \in \NC(G)}G/N$ where each finite group $G/N$ is equipped with the discrete topology. Hence, by Tychonoff's theorem, $\widehat{G}^{\C}$ is a totally disconnected compact topological group. The map $\iota \colon G \to \widehat{G}^{\C}$ defined by $\iota(x) = (\pi_N(x))_{N \in \NC(G)}$ is a homomorphism, and since $G$ is residually-$\C$, the homomorphism $\iota$ is an injection. Hence, with a slight abuse of notation, we may omit writing $\iota$ and just write $G \leq \widehat{G}^{\C}$. For a subset $X \subseteq G$, it then makes sense to write $\overline{X}$ to denote its closure in $\widehat{G}^{\C}$ as $\overline{X}$. Finally, each homomorphism $\pi_N$ can be uniquely extended to a continuous homomorphism $\widehat{\pi}_N \colon \widehat{G}^{\C} \to G/N$, where $\ker(\widehat{\pi}_N) = \overline{N}$.

Following \cite[Definition 4.3]{mf}, for $H \leq G$ and $g \in H$, we say that the pair $(H,g)$ satisfies the $\C$-centraliser condition in $G$ ($\C\text{-CC}_G$ for short) if for every subgroup $K \in \NC(G)$, there is a subgroup $L \in \NC(G)$ such that $L \leq K$ and
\begin{displaymath}
    C_{\pi_L(H)}(\pi_L(g)) \subseteq \pi_L(C_H(g)K) \mbox{ in G/L}. 
\end{displaymath}

The following statement is a generalisation of \cite[Proposition~12.1]{ashot_raags} to the general pro-$\C$ setting, and our proof closely follows that of \cite[Proposition~12.1]{ashot_raags}.
\begin{proposition}\label{appendix: first statement}
    Let $G$ be a residually-$\C$ group, and let $H \leq G, g \in G$ be arbitrary. Then the following are equivalent:
    \begin{enumerate}
        \item[(i)] the pair $(H,g)$ satisfies $\C\text{-CC}_G$,
        \item[(ii)] $C_{\overline{H}}(g) = \overline{C_H(g)}$ in $\widehat{G}^{\C}$.
    \end{enumerate}
\end{proposition}
\begin{proof}    
    Suppose that the pair $(H,g)$ satisfies $\C\text{-CC}_G$. We want to show that $C_{\overline{H}}(g) = \overline{C_H(g)}$ in $\widehat{G}^{\C}$. Let us note that the inclusion $\overline{C_H(g)} \subseteq C_{\overline{H}}(g)$ holds trivially; hence, we only need to show the inclusion in the opposite direction, i.e.~$C_{\overline{H}}(g) \subseteq \overline{C_H(g)}$. Let $h \in \overline{H}\setminus \overline{C_H(g)}$ be arbitrary. By definition, there exists some subgroup $K \in \NC(G)$ such that $\widehat{\pi}_K(h) \notin C_{\pi_K(H)}(g)$. As the pair $(H,g)$ satisfies $\C\text{-CC}_G$, there is a subgroup $L \in \NC(G)$ such that $L \leq K$ and
        \begin{displaymath}
            C_{\pi_L(H)}(\pi_L(g)) \subseteq \pi_L(C_H(g)K) \mbox{ in G/L}. 
        \end{displaymath}
    Since $L \leq K$, we see that $\widehat{\pi}_K$ factors through $\widehat{\pi}_L$, namely $\widehat{\pi}_K = \pi_{K,L} \circ \widehat{\pi}_L$. This means that $\widehat{\pi}(h) \notin C_{\pi_L(H)}(g)$. Therefore, $h \notin C_{\overline{H}}(g)$. This shows that $C_{\overline{H}}(g) \subseteq \overline{C_H(g)}$ in $\widehat{G}^{\C}$.

    Now, suppose that $C_{\overline{H}}(g) = \overline{C_H(g)}$ in $\widehat{G}^{\C}$. Choose any $K \in \NC(G)$, and denote $\mathcal{L} = \{ L \in \NC(G) \: | \: L \leq K\}$. We proceed by contradiction and suppose that, for each $L \in \mathcal{L}$, there is an $x_L \in H$ such that $\pi_L(x_L) \in C_{\pi_L(H)}(\pi_L(g)) \backslash \pi_L(C_H(g)K)$. We observe that $\mathcal{L}$ is a directed set where if $L_1, L_2 \in \mathcal{L},$ then $L_1 \preceq L_2$ if and only if $L_2 \leq L_1$. Hence, $(x_L)_{L \in \mathcal{L}}$ is a net in $\widehat{G}^{\C}$, and since $\widehat{G}^{\C}$ is compact, this net has a cluster point $h \in \overline{H} \leq \widehat{G}^{\C}$.

    Consider any $N \in \NC(G)$ and set $L=N \cap K \in \NC(G)$. According to the definition of the topology of $\widehat{G}^{\C}$, there is a subgroup $M \in \mathcal{L}$ such that $M \leq L$ and $\pi_L(x_M) = \widehat{\pi}_L(h)$. By construction, we have that $\pi_L(x_M) \in C_{\pi_L(H)}(\pi_L(g))$. Therefore, $\widehat{\pi}_L(h) \in C_{\pi_L(H)}(\pi_L(g))$, which implies that $\widehat{\pi}_N(h) \in C_{\pi_N(H)}(\pi_N(g))$ because $L \leq N$. Since the latter holds for every $N \in \NC(G)$, we have $h \in C_{\overline{H}}(g)$.

    On the other hand, given that $h$ is a cluster point of the net $(x_L)_{L\in \mathcal{L}}$ and $K \in \mathcal{L}$, there exists $M \in \mathcal{L}$ where $\pi_K(x_M) = \widehat{\pi}_H(h)$. Since $M \leq K$, we have $x_M \notin C_H(g)KM = C_H(g)K = \pi_K^{-1}(\pi_K(C_H(g)))$. Therefore, $\widehat{\pi}_K(h) = \pi_K(x_M) \notin \pi_K(C_H(g))$. Hence, $h \notin \overline{C_H(g)}$, which is a contradiction. Hence, $(H,g)$ satisfies $\C\text{-CC}_G$.
\end{proof}

Therefore, we have the following corollary which generalises \cite[Corollary~12.2]{ashot_raags} to the pro-$\C$ setting.
\begin{corollary}
    A residually-$\C$ group $G$ satisfies $\C$-CC if and only if $\overline{C_G(g)} = C_{\widehat{G}^{\C}}(g)$ for every $g \in G$.
\end{corollary}
Let us note that the above corollary also follows from \cite[Lemma 2.2]{boggi2025}, which states that, in a residually-$\C$ group $G$ and arbitrary $g \in G$, the following two statements are equivalent:
\begin{itemize}
    \item[(i)] the conjugacy class $g^K$ is closed in $\proC(G)$ for every $K \in \NC(G)$;
    \item[(ii)] the conjugacy class $g^G$ is closed in $\proC(G)$ and $\overline{C_G(g)} = C_{\widehat{G}^{\C}}(g)$.
\end{itemize}

We note that $\C$-conjugacy separability of a residually-$\C$ group $G$ is equivalent to the condition
$$
g^{\widehat{G}^{\C}} \cap G = g^G \text{ in } \widehat{G}^{\C}, \text{ for all } g \in G.
$$
In other words, the above condition says that two elements $g$ and $h$ of $G$ are conjugate in $\widehat{G}^{\C}$ if and only if they are conjugate in $G$.

We finish by reformulating the $\C$-hereditary conjugacy separability of $G$ in purely pro-$\C$ terms.
\begin{corollary}
    \label{corollary:centralisers in the completion}
    Suppose that $G$ is a residually-$\C$ group. Then $G$ is $\C$-hereditarily conjugacy separable if and only if, for every $g \in G$, both of the following hold in the pro-$\C$ completion $\widehat{G}^{\C}$ of $G$:
    \begin{itemize}
        \item $h^{\widehat{G}^{\C}} \cap G = g^G;$
        \item $\overline{C_G(g)} = C_{\widehat{G}^{\C}}(g)$.
    \end{itemize}
\end{corollary}

It might be worthwhile to formulate a normaliser analogue of the $\C$-centraliser condition, possibly for subgroups, and study its connection to hereditary aspects of $\C$-conjugacy distinguished subgroups and $\C$-conjugacy separable subgroups.

For a group $G$, we write $\Sub(G) = \{H \leq G\}$ to denote the space of all subgroups of $G$. We may want to define subspaces of $\Sub(G)$ by specifying a specific subset of subgroups, such as
\begin{itemize}
    \item $\Sub_{fg}(G) = \{H \leq G \mid \mbox{$H$ is finitely generated}\}$,
    \item $\Sub_{\C}(G) = \{H \leq G \mid H \in \C\}$,
\end{itemize}
and many more. One can extend the pro-$\C$ topology on $G$ to $\Sub(G)$ as the topology with the base of open sets consisting of all the sets of the form
\begin{displaymath}
    W(H,N) = \{H' \leq G \mid H'N = HN\},
\end{displaymath}
where $H \in \Sub(G)$ and $N \in \NC(G)$. Informally speaking, $W(H,N)$ is the neighbourhood of $H$ consisting of all subgroups of $G$ that are equal to $H$ modulo $N$. It then follows that $\U \subseteq \Sub(G)$ is $\C$-closed if for every $H \in \Sub(G) \setminus \U$ there is $N \in \NC(G)$ such that $\pi_N(U) \neq \pi_N(H)$ for every $U \in \U$. The pro-$\C$ topologies on the subspaces $\Sub_{fg}(G)$, $\Sub_{fp}(G)$, and $\Sub_{\C}(G)$ are then the respective subspace topologies.

Using this terminology, the following have been defined in the literature, mostly in the setting when $\C$ is the class of all finite groups:
\begin{itemize}
    \item $H \leq G$ is \emph{$\C$-conjugacy distinguished} in $G$ if the subset $H^G = \{ghg^{-1} \mid h \in H, g \in G\}$ is $\C$-closed in $G$ \cite{ribes2016conjugacy};
    \item $G$ is \emph{$\C$-subgroup conjugacy distinguished} if every $H \in \Sub_{fg}(G)$ is \emph{$\C$-conjugacy distinguished} in $G$ \cite{chagas2026subgroup};
    \item $H \leq G$ with $H \in \C$ is \emph{$\C$-hereditarily conjugacy distinguished} if the conjugacy class $H^K = \{kHk^{-1} \mid k \in K\}\subseteq \Sub_{\C}(G)$ is $\C$-closed in $\Sub_{\C}(G)$ whenever $K \leq G$ is $\C$-open in $G$ \cite{boggi2025};
    \item $H \leq G$ is \emph{$\C$-conjugacy separable} in $G$ if the conjugacy class $H^G = \{gHg^{-1}\} \subseteq \Sub(G)$ is $\C$-closed in $\Sub(G)$ \cite{bogopolski2010subgroup};
    \item $G$ is \emph{$\C$-subgroup conjugacy separable} if for every $H \in \Sub_{fg}(G)$ the conjugacy class $H^G = \{gHg^{-1}\} \subseteq \Sub(G)$ is $\C$-closed in $\Sub_{fg}(G)$ \cite{chagas2026subgroup};
    \item $G$ is \emph{$\C$-hereditarily subgroup conjugacy separable} if every $H \leq G$ is \emph{$\C$-subgroup conjugacy separable} whenever $H$ is $\C$-open in $G$ \cite{chagas2026subgroup}.
\end{itemize}
The relation between the density of normalisers of subgroups in normalisers in the pro-$\C$ completion has been partially investigated in the literature.

In \cite[Lemma 2.3]{boggi2025}, it is shown that $H \in \Sub_{\C}(G)$ is $\C$-hereditarily subgroup conjugacy distinguished in $G$ if and only if $H$ is $\C$-conjugacy distinguished in $G$ and $N_{\widehat{G}}(H) = \overline{N_G(H)}$.

In \cite[Proposition 2.5]{chagas2026subgroup}, it is shown that if $G$ is a hereditarily subgroup conjugacy separable group, then $N_{\widehat{G}}(\overline{R}) = \overline{N_G(R)}$ for all finitely generated subgroups. Subsequently, \cite[Corollary 2.6]{chagas2026subgroup} demonstrates that this condition holds for all limit groups. Later in the same article, the authors demonstrate in \cite[Lemma 4.9]{chagas2026subgroup} that if $L$ is a full subgroup of a product of limit groups $G$, then $N_{\widehat{G}}(\overline{L}) = \overline{N_G(L)}$. Using this, they demonstrate that finitely presented residually free groups are finitely presented subgroup conjugacy separable in \cite[Theorem 4.11]{chagas2026subgroup}.

In a similar fashion, \cite[Theorem B]{aguiar_zalesskii} demonstrates that if $G$ is a virtually free group and $H$ is a finitely generated subgroup, then $N_G(H)$ is dense in $N_{\widehat{G}}(\overline{H})$. In fact, they demonstrate that a similar result holds for the pro-$\C$ completion in free-by-$\C$ groups, where $\C$ is an extension-closed pseudovariety of finite groups. They also demonstrate that virtually free groups are subgroup conjugacy separable \cite[Theorem D]{aguiar_zalesskii}, and more generally, they demonstrate that free-by-$\C$ groups are $\C$-subgroup conjugacy separable.

These examples demonstrate the connection between a possible normaliser analogue of the $\C$-centraliser condition and $\C$-conjugacy distinguished subgroups. Thus, it would be interesting to see whether an appropriate formalisation of the subgroup normaliser analogue of the $\C$-centraliser condition is an equivalent formulation for hereditary subgroup conjugacy separability (or distinguishability) by proving the implication in the opposite direction.

\bibliographystyle{plain}
\bibliography{references}

\end{document}